\documentclass[a4paper,12pt]{article}

\usepackage[margin=3cm,footskip=1cm]{geometry}
\usepackage[T1]{fontenc}
\usepackage[utf8]{inputenc}
\usepackage{natbib}
\usepackage{rotating}
\usepackage[raggedright]{titlesec}
\usepackage[all]{onlyamsmath}
\usepackage{amsfonts,amsmath,bm,amssymb,amsxtra,amsthm,bm,mathtools,xspace,booktabs,url,lscape,graphicx,etoolbox,soul,color,multirow,subfigure,authblk,siunitx,lipsum,hyperref,fixme,float,latexsym}
\usepackage{eucal,mathrsfs,rotating,srcltx}

\numberwithin{equation}{section}
\theoremstyle{plain} % the default, bold header, italic body
\newtheorem{theo}{Theorem}[section]

\theoremstyle{definition} % bold header, upright body

\makeatletter
\newcommand\CorrespondingAuthor[1]{%
  \begingroup%
  \def\@makefnmark{}%
  \footnotetext{Corresponding author: #1}%
  \endgroup%
}
\makeatother
\makeatletter
\renewenvironment{abstract}{%
  \small%
  \begin{center}%
    \bfseries \abstractname\vspace{-.5em}\vspace{\z@}%
  \end{center}%
  \quote%
}{\endquote}
\makeatother
\makeatletter
\DeclareRobustCommand*\subref{\@ifstar\sf@@subref\sf@subref}
\makeatother

\DeclareMathOperator\E{E}

\newcommand{\R}{\mathbb{R}}

\begin{document}

\title{
%On Spatial (Skew) $t$ Processes and Applications
Non-Gaussian  Geostatistical Modeling using (skew) t Processes
% \textcolor{red}{Alternative: On modeling and estimating geo-referenced spatial data with heavy tails}
}

\author[1]{Moreno Bevilacqua}
\affil[1]{Departamento de Estad\'istica, Universidad of Valpara\'iso, Chile}
\affil[1]{Millennium Nucleus Center
for the Discovery of Structures
in Complex Data , Chile, \texttt{moreno.bevilacqua@uv.cl}}

\author[2]{Christian Caama\~no}
\affil[2]{Departamento de Estad\'istica, Universidad del B\'io-B\'io, Concepci\'on, Chile,  \texttt{chcaaman@ubiobio.cl}}

\author[3]{Reinaldo B. Arellano Valle}
\affil[3]{Departamento de Estad\'istica, Pontificia Universidad Cat\'olica de Chile, Santiago, Chile,  \texttt{reivalle@mat.uc.cl}}

\author[4]{V\'ictor Morales-O\~nate}
\affil[4]{Departamento de Desarrollo, Ambiente y Territorio, Facultad Latinoamericana de Ciencias Sociales, Quito, Ecuador     \texttt{victor.morales@uv.cl}}

\date{\today}

\maketitle

\begin{abstract}
We propose a new  model for regression and dependence analysis when  addressing  spatial  data %or spatiotemporal data
 with possibly   heavy tails and an asymmetric marginal distribution.
We first propose a  stationary process  with  $t$  marginals  obtained through scale mixing of a Gaussian
process with an inverse square root process with Gamma marginals.
We then  generalize this construction
by considering   a skew-Gaussian process, thus
obtaining a process with skew-t marginal distributions. For the  proposed (skew) $t$ process
we study the second-order and geometrical properties  and in the $t$ case, we provide  analytic expressions for the bivariate distribution.
In an extensive simulation   study, we investigate the use of the weighted pairwise likelihood as a method of estimation
for the $t$ process.
Moreover we compare the performance of the optimal linear predictor of the $t$ process versus the optimal Gaussian predictor.
Finally, the effectiveness of our methodology  is illustrated by analyzing a georeferenced dataset on maximum temperatures in Australia.\\

\noindent\textit{Keywords:}
  Heavy-tailed processes; Hypergeometric functions; Multivariate skew-normal distribution; Scale mixing; Pairwise likelihood.
\end{abstract}

\section{Introduction} % (fold)
\label{sec:introduction}
The geostatistical approach models data coming from a limited number of monitoring stations %and possibly temporal instants
as a partial realization from a spatial %spatio-temporal
stochastic process (or random field) defined on the continuum space. % and time.
Gaussian stochastic processes  are among the most popular tools for analyzing spatial data %and spatio-temporal data
because  a mean structure and a valid covariance
function completely characterized the associated finite dimensional distribution.
Additionally, optimal prediction  at an unobserved site %and temporal instant
 depends on the knowledge of the covariance function of the process.
%Since a covariance function must be positive definite, practical estimation generally requires the selection of some parametric classes of covariances and the corresponding
%estimation of these parameters.
Unfortunately, in many geostatistical applications, including climatology,
oceanography, the environment and the study of natural resources, the
Gaussian framework is unrealistic  because the observed data have
specific features such as negative or positive asymmetry and/or
heavy tails.

The focus of this work is on non-Gaussian models for stochastic processes  that vary continuously in the euclidean  space,
even if the proposed methodology can be easily extended to the space-time framework or to the  spherical space. %or space-time.
In particular, we aim to accommodate heavier tails than the ones induced by Gaussian processes and wish to allow possible asymmetry.
In recent years, different approaches have been proposed in
order to analyze these kind of data. Transformation of Gaussian
(trans-Gaussian) processes  is a general  method  to model
non-Gaussian spatial data  obtained by applying some nonlinear
transformations to the original data
\citep{Oliveira_et_al:1997,Allcroft:Glasbey:2003,DeOliveira:2006}.
Then statistical analyses can be carried out on the transformed
data using any techniques available for Gaussian processes.
However, it can be difficult to find an adequate nonlinear
transformation and some appealing properties of the latent
Gaussian process  may not be inherited by the transformed process.
A flexible trans-Gaussian process  based on the Tukey $g-h$
distribution has been proposed in \cite{Xua:Genton:2017}.

\cite{Wallin:Bolin:2015} proposed non-Gaussian processes derived
from stochastic partial differential equations to model
non-Gaussian spatial data. However this approach is
restricted to the Mat\'ern covariance model with integer smoothness parameter  and its statistical
properties are much less understood than those of  the Gaussian process.

The copula framework \citep{Joe:2014} has been adapted in the spatial context in order  to account for possible deviations from the Gaussian distribution.
Even
though which copula model to use for a given analysis is not generally known a priori, the
copula based on the multivariate Gaussian distribution has gained a general consensus  \citep{Kazianka:Pilz:2010,Masarotto:Varin:2012,BG:2014} since the definition of the
multivariate dependence relies again on the specification of the correlation function.
However, Gaussian copula could be too restrictive in some cases since it expresses a symmetrical  and elliptical dependence.

%This kind of  approach inherits
%the features of the Gaussian distribution:
%data example, the kind of the dependence described by the Gaussian copula could be too
%restrictive.

Convolution  of  Gaussian and non-Gaussian processes  is an  appealing
strategy for modeling spatial data with  skewness.
For instance,
\cite{Zhang:El-Shaarawi:2010} proposed  a Gaussian-Half Gaussian
convolution in order to construct a process with  marginal
distributions
of the skew-Gaussian type
\citep{Azzalin:Capitanio:2014}. %Maximum likelihood inference was studied trough a stochastic EM alghoritm.
%\cite{ALEGRIA2017} extend it   to the multivariate setting and propose the use of pairwise likelihood as viable method of estimation.
\cite{zareifard2018}
developed bayesian inference for the estimation of
a process with asymmetric marginal distributions obtained  through
convolution of  Gaussian and  Log-Gaussian processes.
\cite{BM2017} proposed a skew-Gaussian process using the skew-model proposed in \cite{sahu2003}.
The resulting process is not mean-square continuous and as a consequence it is not a suitable model
for data exhibiting
smooth behavior of the realization.
%Recently, \cite{BM2018}
%show that some of the skew Gaussian processes proposed in the literature, as for instance \cite{Kim:Mallick:2003} and \cite{ZA36},
 %are ill-defined.

On the other hand, mixing   of  Gaussian and non-Gaussian processes
is a useful strategy for modeling spatial data with heavy tails.
For instance,
 \cite{Palacios:Steel:2006} and  \cite{ZAREIFARD201316}
proposed a (skew) Gaussian-Log-Gaussian scale mixing approach in order
 to accommodate the presence of possible outliers for spatial data.

The $t$ distribution is a parametric model  that is able to accommodate  flexible tail
behavior, thus  providing robust estimates against extreme data
 and it  has been studied extensively  in recent   years \citep{AB95,TT:1989,FF2008,qqpo,qqpo1}.
 %  Some contribution can be found in \cite{Branco:Dey:2001}, \cite{Gupta:Chang:Huang:2002}, \cite{Azzalini:Capitanio:2003},
  % \cite{Arellano:Genton:2005}, \cite{Azzalini:Genton:2008},
  %\textcolor{red}{more ref} to name a few.
Stochastic processes with marginal $t$ distributions have been
introduced in \cite{RHO:2006},
 \cite{Ma:2009},   \cite{Ma:2010} and  \cite{DEB2015},
but as outlined in  \cite{genton:Zhang:2012},  these models are not be identifiable
when only a single realization  is available (which is typically the case for spatial data).

In this paper, %we exploit the classical stochastic representation of
%the $t$  random variable in order to construct a process with
%marginal distribution of the $t$ type. Specifically,
 we propose a
process with marginal $t$ distributions obtained though scale mixing  of a standard Gaussian process with  an inverse square
root  process with Gamma marginals. The latter is obtained through a rescaled
sum of independent  copies of  a  standard squared  Gaussian process.
%\citep{Bevilacqua:2018ab}.
Although  this can be viewed as a natural  way  to define a $t$   process, the associated second-order,
geometrical properties and bivariate distribution are somewhat  unknown to the best of our knowledge.
Some   results can be found in \cite{heyde2005} and \cite{seneta2006}. We study  the second-order and geometrical
properties  of the $t$ process and  we provide  analytic expressions  for the correlation
and the bivariate distribution. It turns out that both depend on  special functions,  particularly
the Gauss hypergeometric and Appell function of the fourth type \citep{Gradshteyn:Ryzhik:2007}. In addition,
the bivariate distribution is not of elliptical type.

We then focus on processes with asymmetric  marginal distributions and heavy tails. We first review the skew Gaussian process proposed in \cite{Zhang:El-Shaarawi:2010}.
For this process we provide an explicit expression of the finite dimensional distribution generalizing previous results in \cite{ALEGRIA2017}.
%Using the finite-dimensional distribution it can be easily shown that the criterion given in  Mahmedian (2018 SPL)  is met.
We then propose a process with marginal distribution of the skew-$t$ type  \citep{Azzalin:Capitanio:2014}
obtained
through scale mixing of a skew-Gaussian with
an inverse square
root process with Gamma marginals.
%For this process we study the correlation and the associated geometrical properties.

Our proposals for the $t$ and skew-$t$ processes have two main features. First, they allow removal of any problem of  identifiability   \citep{genton:Zhang:2012},
 and as a consequence, all the parameters can be estimated using one realization of the process.
Second,   the $t$ and skew-$t$  processes inherit some of  the geometrical properties of the   underlying Gaussian process.
This implies that
the mean-square continuity and differentiability %(under specific conditions)
 of the $t$ and skew-$t$ processes  can be modeled using  suitable parametric
correlation models as  the Mat{\'e}rn model   \citep{Matern:1960} or  the Generalized Wendland model \citep{Gneiting:2002b,Bevilacqua_et_al:2018}.

% From this point of view, two flexible correlation models are the   Mat{\'e}rn model   \citep{Matern:1960} and the Generalized Wendland model \citep{Gneiting:2002b,Bevilacqua_et_al:2018}.

%Since  in this case maximum likelihood estimation  is computational unfeasible even for
%relatively small datasets,
For the $t$ process estimation we propose the method of weighted pairwise likelihood \citep{Lindsay:1988,Varin:Reid:Firth:2011,Bevilacqua:Gaetan:2015}
exploiting
 the bivariate distribution given in Theorem \ref{theo3}. %and the derivation of an optimal predictor that minimizes  the mean square prediction error.
 In an extensive simulation study we investigate the performance of the  weighted pairwise likelihood ($wpl$) method under different scenarios
  %(temporal, spatial,  and spatiotemporal)
  including
 when the degrees of freedom are supposed to be unknown.
 We also study the performance of the   $wpl$  estimation by assuming a Gaussian process  in the estimation step  with correlation function
  equal to the correlation function of the  $t$ process.
   It turns out that  the Gaussian  misspecified $wpl$ (see   \cite{cppp} with the references
therein) leads to a less efficient  estimator, as expected. However, 
   the method has some computational benefits.%  since  it does not require the computation of the Appell function.

Additionally, we compare the performance of the optimal linear predictor of the $t$ process
with the optimal predictor of the Gaussian process.
Finally we apply the proposed methodology  by analyzing a real data set of maximum temperature in Australia
where, in this case, we consider  a $t$ process   defined on a portion of the sphere (used as an approximation of the planet Earth)
and use a correlation model depending  on the great-circle distance \citep{gneiting2013}.

The methodology   considered in this paper is  implemented  in the R package
\texttt{GeoModels} \citep{Bevilacqua:2018aa}.
% In particular,  the $wpl$ estimation
%method has been implemented
%using the Open Computing Language (OpenCL)
%in order to reduce  the computational costs associated with  the  Appell function evaluation.
The remainder of the paper is organized as follows.  In Section \ref{sec:2} we introduce the $t$ process, study the second-order and geometrical properties
and   provide an analytic expression for the bivariate distribution. In Section \ref{sec:3}, we first study the  finite dimensional distribution of the skew Gaussian process,
and then we study the second-order properties of the skew-$t$ process.
%and then we propose a general class of processes with asymmetric marginals and then we focus on the Gaussian and $t$ cases.
In Section \ref{sec:4}, we present a simulation study in order to investigate the performance of the (misspecified) $wpl$ method
when estimating  the $t$ process and the performance of the associated optimal linear predictor versus the optimal Gaussian predictor.
In Section \ref{sec:5}, we analyze a real data set of maximum temperature in Australia.
Finally, in Section \ref{sec:6}, we give some conclusions.
All the  proofs has been
deferred to the Appendix.

\section{A stochastic process with $t$ marginal distribution} % (fold)
\label{sec:2}

For the rest of the paper, given  a  process $Q=\{Q(\bm{s}), \bm{s} \in A \}$
with  $\E(Q(\bm{s}))=\mu(\bm{s})$ and
$Var(Q(\bm{s}))=\sigma^2$,
we denote by
$\rho_Q(\bm{h})=Corr(Q(\bm{s}_i),Q(\bm{s}_j))$ its correlation function, where $\bm{h}=\bm{s}_i-\bm{s}_j  $
is the  lag separation vector. For any   set of  distinct points $(\bm{s}_1,\ldots,\bm{s}_n)^T$, $n\in \mathcal{N}$, we denote by $\bm{Q}_{ij}=(Q(\bm{s}_i),Q(\bm{s}_j))^T$, $i \neq j$, the bivariate random vector and by $\bm{Q}=(Q(\bm{s}_1),\ldots,
Q(\bm{s}_n))^T$ the multivariate random vector. Moreover, we denote with $f_{Q(\bm{s})}$ and  $F_{Q(\bm{s})}$
the marginal probability density function (pdf) and cumulative distribution function (cdf) of  $Q(\bm{s})$ respectively, with  $f_{\bm{Q}_{ij}}$ the pdf of $\bm{Q}_{ij}$ and with
$f_{\bm{Q}}$ the pdf of $\bm{Q}$.
Finally, we denote with $Q^*$ the standardized weakly stationary  process, $i.e.$,
$Q^*(\bm{s}):=(Q(\bm{s})-\mu(\bm{s}))/\sigma$.
%\begin{equation}\label{opok}
%Q^*(\bm{s}):=(Q_{\nu}(\bm{s})-\mu(\bm{s}))/\sigma.
%\end{equation}

%For simplicity of presentation, we  restrict the treatment to the spatial Euclidean setting $A\subseteq  \R^d $.
%Nevertheless,  the results presented
%in this paper can be applied to the  spatiotemporal
%or  spherical
%setting (see Section \ref{sec:5}).

%or, more in general, to a Banach space.
 % or in general to any \textcolor{red}{Qualsiasi spazio dove posso definire un processo Gaussiano   con una funzione di covarianza? Come si chiama? }

%by considering
%a weakly stationary spatio-temporal
%RFs $Q=\{Q(\bm{s}, t), \bm{s} \in A \subset \mathcal{R}^d, t\in B  \subset \mathbf{R} \}$
%with correlation function $\rho_Q(\bm{h},u)=Corr(Q(\bm{s}_i, t_l),Q(\bm{s}_j), t_k)$, where $u=t_l-t_k$,
%with bivariate random vector $\bm{Q}_{ijlk}=(Q(\bm{s}_i,t_l),Q(\bm{s}_j,t_k))^T$
%and
 %finite dimensional distribution
 %$\bm{Q}=(Q(\bm{s}_1,t_1),\ldots,
%Q(\bm{s}_n,t_m))^T$, for any  set of spatial locations  $(\bm{s}_1,\ldots,\bm{s}_n)$ and temporal  instanrs
%$(t_1,\ldots,t_m)$  $n\in \mathbf{N}$, $m\in \mathbf{N}$.

%According to \cite{genton:Zhang:2012}, the identifiability problem
%of the skew elliptically contoured process defined by
%Equation (xxx) can be removed by considering the modification
%\begin{equation}\label{skewelliprep}
%Y_{\nu}(\bm{s})=\frac{(R_{\nu}(\bm{s})-E(R_{\nu}(\bm{s}))}{[E(R_{\nu}(\bm{s})-E(R_{\nu}(\bm{s}))^2]^{1/2}}\frac{(Z(\bm{s})-E(Z(\bm{s}))}{[E(Z(\bm{s})-E(Z(\bm{s}))^2]^{1/2}}
%\end{equation}
%
%We consider (\ref{skewelliprep}) non-standardized, defined as

As outlined in  \cite{Palacios:Steel:2006}, given a positive process
$M=\{M(\bm{s}), \bm{s} \in A\}$ and  an independent  standard Gaussian  process $G^*=\{G^*(\bm{s}), \bm{s} \in A\}$,
a general class
of non-Gaussian processes with marginal  heavy tails %$T=\{T(\bm{s}), \bm{s} \in A\}$
 can be obtained as scale mixture of
$G^*$, $i.e.$
$\mu(\bm{s})+\sigma M(\bm{s})^{-\frac{1}{2}}G^*(\bm{s})$,
%\begin{equation}\label{kkkkk}
%T(\bm{s}):=\mu(\bm{s})+\sigma M(\bm{s})^{-\frac{1}{2}}G^*(\bm{s})
%\end{equation}
where %$M=\{M(\bm{s}), \bm{s} \in A\}$
%is a positive process  independent of $G^*$,
$\mu(\bm{s})$ is the location dependent  mean and  $\sigma>0$  is a scale parameter. %In particular they focus on a log-gaussian  mixing process
%\citep{Oliveira_et_al:1997} obtaining what they called
 %the \emph{Gaussian-log-Gaussian} process.
 A typical
parametric specification for the  mean is given by
$\mu(\bm{s})=X(\bm{s})^T\bm{\beta}$ where  $X(\bm{s}) \in  \R^k$ is a
vector of  covariates and  $\bm{\beta}  \in  \R^k$ but other types of
parametric or nonparametric functions can be considered.

 Henceforth, we call $G^*$ the `parent' process and
with some abuse of notation we set $\rho(\bm{h}):=\rho_{G^{*}}(\bm{h})$ and $G:=G^{*}$.
 Our proposal  considers a mixing process $W_{\nu}=\{W_{\nu}(\bm{s}), \bm{s} \in A\}$  with marginal distribution
 $ \Gamma (\nu/2,\nu/2)$ defined as $W_{\nu}(\bm{s}):=\sum_{i=1}^\nu
G_i(\bm{s})^2/\nu$
where $G_i$, $i=1,\ldots \nu $  are  independent copies of $G$
with  $\E(W_{\nu}(\bm{s}))=1$,
$Var(W_{\nu}(\bm{s}))=2/\nu$ and $\rho_{W_{\nu}}(\bm{h}) =
\rho^2(\bm{h})$
\citep{Bevilacqua:2018ab}.
If we   consider a process $Y^*_{\nu}=\{Y^*_{\nu}(\bm{s}), \bm{s} \in A\}$ defined as
 \begin{equation}\label{rut}
 Y^*_{\nu}(\bm{s}):=W_{\nu}(\bm{s})^{-\frac{1}{2}}G(\bm{s}),
 \end{equation}
then, by construction,    $Y^*_{\nu}$ has  the marginal $t$ distribution with $\nu$ degrees of freedom  with pdf given by:
\begin{equation}\label{ut}
f_{Y^{*}_{\nu}(\bm{s})}(y;\nu)=\frac{\Gamma\left(\frac{\nu+1}{2}\right)}{\sqrt{\pi\nu}\Gamma\left(\frac{\nu}{2}\right)}\left(1+\frac{y^2}{\nu}\right)^{-\frac{(\nu+1)}{2}}.
\end{equation}

Then, we  define the location-scale transformation process $Y_{\nu}=\{Y_{\nu}(\bm{s}), \bm{s} \in A\}$ as:
\begin{equation}\label{kkkkk2}
Y_{\nu}(\bm{s}):=\mu(\bm{s})+\sigma Y^*_{\nu}(\bm{s})
\end{equation}
with   $\E(Y_{\nu}(\bm{s}))=\mu(\bm{s})$ and
$Var(Y_{\nu}(\bm{s}))=\sigma^2\nu/(\nu-2)$, $\nu>2$.
%Henceforth with a $t_v$ process we denote

%Following (\textcolor{red}{ref arxiv}),  we first define a process with Gamma
%marginals. Let $G_i$, $i=1,\ldots \nu $  be independent copies of
%$G=\{G(\bm{s}), \bm{s} \in A\}$, a zero mean and  unit variance
%second order Gaussian process with  correlation function
%$\rho_G(\bm{h})$. A process $W_{\nu}=\{W_{\nu}(\bm{s}) \in A\}$ with   Gamma
%marginal distributions $W_{\nu}(\bm{s})\sim \Gamma (\nu/2,\nu/2)$,
%can be defined as  $W_{\nu}(\bm{s}):=\sum_{i=1}^\nu
%G_i(\bm{s})^2/\nu$ with

 {\bf Remark 1}:
A  possible drawback for  the Gamma process  $W_{\nu}$   is that it is  a limited model due to the restrictions
to the  half-integers for  the shape parameter.
%\textcolor{red}{
%The Laplace transform of the random vector  $W(\bm{s}_1),\ldots,W(\bm{s}_n)$
%is given by
%\begin{equation}
%\E\left(\exp\left(-\sum_{i=1}^n\alpha_i W(\bm{s}_i)\right)\right)=\frac{1}{|I+(2/m)AR|^{m/2}},
%\end{equation}
%where $I$ is the $n \times n$ identity matrix,
%$A=\operatorname{diag}(\alpha_1,\ldots,\alpha_n)$
%$A$ is the diagonal matrix with $a_{i,i} = 2\,\alpha_i/m >0$
% and $R = \{\rho(s_i , s_j )\}$ is the correlation matrix of $Z(s_1),\ldots,Z(s_n)$ \citep[Proposition 4.3]{Vere-Jones:1997}.
 %}
Actually, in some special  cases, it can assume any positive value greater than zero.
% $\nu>2$ and the resulting model is much more flexible.
 This feature is intimately related to the infinite divisibility of the squared Gaussian process $G^2=\{G^2(\bm{s}), \bm{s} \in A\}$
as shown in \cite{Kri:rao:1961}. Characterization of the infinite  divisibility of $G^2$ has been studied in \cite{Vere-Jones:1997}, \cite{Bapat:1989}, \cite{Griffiths:1970} and \cite{Eisenbaum:Kaspi:2006}. In particular  \cite{Bapat:1989}
 provides a characterization based on $\Omega$, the correlation matrix
associated with $\rho(\bm{h})$. Specifically,  $\nu >0$
%$\nu \in \R_+$
if and only if there exists a matrix $S_n$ such that $S_n\Omega^{-1}S_n$ is   an $M$-matrix  \citep{plee:1977}, where $S_n$ is a signature matrix, $i.e$., a diagonal matrix of size $n$ with entries either $1$ or $-1$.
This  condition is satisfied, for  instance, by a stationary Gaussian random process $G$ defined on $A=\mathbb{R}$ with  an exponential correlation function.
The  $t$ process $Y^*_{\nu}$ inherits this feature with the additional restriction $\nu>2$. This implies that   $Y^*_{\nu}$  is well defined for $\nu=3,4,\ldots$  and  for $\nu>2$
under non-infinite divisibility of $G^2$.

{\bf Remark 2}:
The finite dimensional distribution of $Y^*_{\nu}$  is unknown  to the best of our knowledge, but in principle, it can be derived  by mixing  the multivariate  density  associated with
 $W^{-\frac{1}{2}}_{\nu}$
with the multivariate standard Gaussian density.
% The existence of such pdf is guaranteed for $\nu=3,4, \ldots$
%and, under infinity divisibility of the multivariate Gamma density, for all $\nu>2$.
The multivariate Gamma density $f_{\bm{W}_{\nu}}$  was first discussed by
\cite{krishnamoorthy1951}
 and its properties have been studied by different authors \citep{Royen:2004,marcus2014}.
 In the bivariate case, \cite{VJ:1967}  showed that  the bivariate Gamma distribution is  infinite divisible, $i.e.$  $\nu>0$ in (\ref{pairchi2}),
 irrespective  of the correlation function.
 Note that this  is consistent with the characterization given in  \cite{Bapat:1989} since,  given  an arbitrary  bivariate  correlation matrix $\Omega$,
  there exists a matrix $S_2$ such that
$S_2\Omega^{-1} S_2$  is a $M$-matrix.
 %because the inverse of a  bivariate correlation matrix is a  diagonally dominant $M$ matrix.
  In Theorem \ref{theo3} we provide the bivariate distribution of $Y^*_{\nu}$ %\textcolor{red}{to checkkk}

Note that,  both $W_{\nu}$ and $G$ in (\ref{kkkkk2}) are obtained through independent  copies of the   `parent' Gaussian process with correlation $\rho(\bm{h})$.
For this reason, henceforth, in some cases, we will call  $Y^*_{\nu}$ a standard $t$ process with underlying correlation $\rho(\bm{h})$.
%Additionally it is well known that $W_{\nu}(\bm{s})^{-\frac{1}{2}}=1$  for each $\bm{s}$. Then, the $G$ is a special case of the $t$ process when $\nu \to \infty$.
%In principle the  correlation functions
%could be different.  However, if we would use  different
%correlation  parametric models,  inference on the  parameters would
%be  difficult with practically relevant sample sizes, as outlined in \cite{Palacios:Steel:2006}.

In what follows, we make use of  the Gauss  hypergeometric function defined by  \citep{Gradshteyn:Ryzhik:2007}:
\begin{equation}\label{hyy}
{}_2F_1(a,b,c;x)=\sum\limits_{k=0}^{\infty}\frac{(a)_k (b)_k}{(c)_k}\frac{x^k}{k!},\;\;\;
\end{equation}
with $(s)_{k}=  \Gamma(s+k)/\Gamma(s)$ for $k \in  \mathbb{N}\cup\{ 0\}$
being the Pochhammer symbol and  we  consider the  restrictions   $a>0$, $b>0$, $c>0$ and $x\geq 0$. If  $c>a+b$
 the radius of convergence of (\ref{hyy}) is  $0\leq x\leq1$ and, in particular  (\ref{hyy})   is convergent at $x=1$  trough the identity: % \citep{Gradshteyn:Ryzhik:2007}:
\begin{equation}\label{JJ}
{}_2F_1\left(a,b;c;1\right)=\frac{\Gamma(c)\Gamma(c-a-b)}{\Gamma(c-a)\Gamma(c-b)}.
\end{equation}
%Otherwise, the hypergeometric function (\ref{hyy}) is divergent when $x=1$.

We  also consider  the Appell hypergeometric  function of the fourth type  \citep{Gradshteyn:Ryzhik:2007} defined as:
\begin{equation*}
F_4(a,b;c,c';w,z)=\sum\limits_{k=0}^{\infty}\sum\limits_{m=0}^{\infty}\frac{(a)_{k+m}(b)_{k+m}w^kz^m}{k!m!(c)_k(c')_m},\;\;\;\;|\sqrt{w}|+|\sqrt{z}|<1.
\end{equation*}

The following Theorem gives an analytic expression for $\rho_{Y^*_{\nu}}(\bm{h})$
in terms of the Gauss hypergeometric function.

\begin{theo}\label{theo0}
Let $Y^*_{\nu}$ be  a standardized $t$ process with underlying correlation $\rho(\bm{h})$. Then:
\begin{equation}\label{CC}
\rho_{Y^*_{\nu}}(\bm{h})=\frac{(\nu-2)\Gamma^2\left(\frac{\nu-1}{2}\right)}{2\Gamma^2\left(\frac{\nu}{2}\right)}
    \left[{}_2F_1\left(\frac{1}{2},\frac{1}{2};\frac{\nu}{2};\rho^2(\bm{h})\right)\rho(\bm{h})\right].
\end{equation}
\end{theo}

The following Theorem depicts some features of the $t$ process.
%and other interesting properties of  its correlation function.
It turns out that nice properties such as stationarity, mean-square continuity and degrees of mean-square differentiability
%(under some specific conditions)
can be inherited from the `parent' Gaussian process $G$.
Further, the $t$ process  has  long-range dependence when the `parent' Gaussian process  has long-range dependence
and this can be achieved when the correlation has some specific features. For instance, the generalized Cauchy  \citep{GneitingS:2004,LiTe09}
and  Dagum \citep{berg2008} correlation models can lead to a Gaussian process with long range dependence.
Finally, an appealing  and intuitive feature is that the correlation of $Y^*_{\nu}$
approaches the correlation of $G$ when $\nu \to \infty$.
%The proofs of Theorems 2.1 and 2.2   can be found in the Appendix.

\begin{theo}\label{theoiii}
Let $Y^*_{\nu}$, $\nu>2$ be  a standardized $t$ process with underlying correlation $\rho(\bm{h})$. Then:
\begin{enumerate}
\item[a)] $Y^*_{\nu}$ is  also  weakly  stationary;
\item[b)]  $Y^*_{\nu}$ is mean-square continuous if and only if $G$ is mean-square continuous;
\item[c)]   Let $G$
 $m$-times mean-square differentiable, for  $m=0,1,\ldots$
 \begin{itemize}
 \item If   $\nu > 2(2m+1)$  then $Y^*_{\nu}$  is  $m$-times mean-square differentiable;
 \item If   $\nu \leq 2(2m+1)$  then  $Y^*_{\nu}$ is $(m-k)$-times mean-square differentiable
if
  $2(2(m-k)+1)< \nu \leq  2(2(m-k)+3)$,  for $k=1, \ldots, m$.
 \end{itemize}
%if $\nu\leq2(2m+1)$ then  $Y^*_{\nu}$  is   not mean-square differentiable;
%\item  $\rho_{Y^*_{\nu}}(\bm{h})=0 \Leftrightarrow \rho(\bm{h})=0$,
\item[d)]  $Y^*_{\nu}$ is  a long-range dependent process  if and only if $G$ is a  long-range dependent process
\item[e)] %$\rho_{Y^*_{\nu}}(\bm{h})\leq \rho(\bm{h})$ and
$\rho_{Y^*_{\nu}}(\bm{h})\leq \rho(\bm{h})$ and
$\lim\limits_{\nu \to \infty } \rho_{Y^*_{\nu}}(\bm{h})=\rho(\bm{h})$.
\end{enumerate}
\end{theo}

One implication of Theorem  (\ref{theoiii}) point c) is that  the process $Y^*_{\nu}$  inherits the mean square differentiability of $G$
under the condition $\nu > 2(2m+1)$. Otherwise, the mean square differentiability  depends on $\nu$.
%$Y^*_{\nu}$ depends on $\nu$. For instance,
 %if $2<\nu \leq 6$ then  $Y^*_{\nu}$  is not mean square differentiable  irrespective of the mean square differentiability   of $G$
 % and
  For instance,  if $G$ is one time mean square differentiable   then
 $Y^*_{\nu}$ can be zero or one time differentiable depending if    $\nu>6$ or not.

{\bf Remark 3}: A simplified  version of  the $t$ process in Equation (\ref{rut}), can be obtained assuming  $W_{\nu}(\bm{s}_i)  \perp W_{\nu}(\bm{s}_j) , i\neq j$.
Under this assumption, $Y^*_{\nu}$ is still a process with $t$ marginal distribution
but, in this case, the geometrical properties are not inherited from the `parent' Gaussian process $G$.
In particular,  it can be shown that
the resulting correlation function   exhibits  a discontinuity at the origin and, as a consequence, the process is  not mean-square continuous.
A  not   mean-square continuous  version of the $t$
process   in Equation (\ref{rut}),
can be obtained by introducing a nugget effect, i.e., a discontinuity of $\rho_{Y^*_{\nu}}(\bm{h})$
at the origin.
This can be easily  achieved   by replacing
$\rho(\bm{h})$ in (\ref{CC})  with
$\rho^*(\bm{h})=1$ if $\bm{h}=\bm{0}$ and
$\rho^*(\bm{h})=(1-\tau^2)\rho(\bm{h})$ otherwise,
where $0\leq\tau^2<1$ represents the underlying  nugget effect.
%Then, it can be easily  shown that
%$\rho_{Y^{*}_{\nu}}(\bm{h})=1$ if $\bm{h}=\bm{0}$
 %and
%$\rho_{Y^*_{\nu}}(\bm{h})=(1-(\tau^2+a(\tau)))\rho_{Y^*_{\nu}}(\bm{h})$ otherwise for some $a(\tau)\geq0$ such that
 %$a(0)=0$ and
%$0\leq\tau^2+a(\tau)<1$. %(\textcolor{red}{better ?})

 %\textcolor{red}{we need to comment different nugget effect in Palacio and Steel 2006}

Since the $t$ process inherits some of the geometrical properties of the `parent' Gaussian process,
the choice of the covariance function  is crucial.
Two flexible isotropic models that allow  parametrizing in a continuous  fashion  the mean square differentiability of a Gaussian process and its  sample paths
 are as follows:
 % the  Mat{\'e}rn   \citep{Matern:1960}
%and the Generalized Wendland  correlation models \citep{Gneiting:2002b,Bevilacqua_et_al:2018}.

%In this  subsection we  illustrate
%some geometric features of the random processes $G$ and $W$  by means of two examples of   parametric  correlation function  $\rho(h)$, namely
\begin{enumerate}
	\item the  Mat{\'e}rn correlation function \citep{Matern:1960}
	
	\begin{equation}\label{mt}
	%\label{Matern}
	{\cal  M}_{\alpha,\psi}(\bm{h})=
	\frac{2^{1-\psi}}{\Gamma(\psi)} \left ({||\bm{h}||}/{\alpha}
	\right )^{\psi} {\cal K}_{\psi} \left({||\bm{h}||}/{\alpha} \right ),
	%\qquad ||\bm{h}|| \ge 0.
	\end{equation}
	
	where ${\cal K}_{\psi}$ is a modified Bessel function of the second kind of
	order $\psi$.
	 Here, $\alpha>0$  and $\psi>0$ guarantee the positive definiteness of the model in any dimension.
	
	%\item the   Generalized Wendland correlation function \citep{Gneiting:2002b},  defined  as:
%		\begin{equation} \label{WG2*}
%	{\cal  GW}_{\alpha,\psi,\delta}(\bm{h}):=
%	\begin{cases}
%	K \left( 1- \left(\frac{||\bm{h}||}{\alpha} \right )^2\right)^{\psi+\delta}
%	    \left[{}_2F_1\left(\frac{\delta}{2},\frac{\delta+1}{2};b;1- \left(\frac{||\bm{h}||}{\alpha} \right )^2 )\right)\right]
%	& ||\bm{h}||< \alpha\\
%	0 & \text{otherwise} \end{cases},
%	\end{equation}
%	with $K=\frac{\Gamma(\psi)\Gamma(\psi+b)}{\Gamma(2\psi)\Gamma(b)2^{\delta+1}}$ and $b=\psi+\delta+1$.
%	Here  $\alpha\geq0$, $\psi > 0$,   and $\delta \ge  (d+1)/2+ \psi$
%		 guarantee the positive definiteness of the model in $\R^d$.

	\item the   Generalized Wendland correlation function \citep{Gneiting:2002b},  defined for $\psi>0$ as:
	\begin{equation} \label{WG2*}
	{\cal  GW}_{\alpha,\psi,\delta}(\bm{h}):=
	\begin{cases}\frac{ \int_{||\bm{h}||/\alpha}^{1} u\left(u^2- (||\bm{h}||/\alpha)^2\right)^{\psi-1} \left ( 1- u \right )^{\delta} d u}{{\operatorname{B}(2\psi,\delta+1)}}
	& ||\bm{h}||< \alpha\\
	0 & \text{otherwise} \end{cases},
	\end{equation}
	and for $\psi=0$  as:
	
		\begin{equation}\label{ask}
	{\cal  GW}_{\alpha,0,\delta}(\bm{h}):=
	\begin{cases}\left ( 1- ||\bm{h}|| / \alpha\right )^{\delta}
	& ||\bm{h}||< \alpha\\
	0 & \text{otherwise} \end{cases}.
	\end{equation}
	Here $\operatorname{B}(\cdot,\cdot)$ is the Beta function and $\alpha>0$, % $\psi > 0$,
	  $\delta \ge  (d+1)/2+ \psi$
	 guarantee the positive definiteness of the model in $\R^d$.

\end{enumerate}
%Both  represent  flexible parametric  models
%for the spatial correlation  and
%the parameter $\psi$ allows a  continuous parametrization  for the smoothness of the associated  Gaussian random process \citep{Stein:1999,Bevilacqua:2018}.

In particular for a
positive integer $k$, the sample paths  of a Gaussian process are $k$ times differentiable
if and only if  $\psi >k$ in the Mat{\'e}rn case  \citep{Stein:1999}
and if and only if  $\psi >k-1/2$ in the Generalized Wendland case  \citep{Bevilacqua_et_al:2018}.
Additionally, the  Generalized Wendland correlation is compactly supported, an interesting feature from computational point of view \citep{Furrer:2006},
which is inherited  by the $t$ process
since $\rho(\bm{h})=0$ implies   $\rho_{Y^*_{\nu}}(\bm{h})=0$.
%,  the sparsity of the correlation matrix associated to the parent Gaussian process
%is inherited by the $t$ correlation matrix.

In order to illustrate some geometric features of the $t$ process,
we first compare the correlation functions of the Gaussian and $t$ processes
using an underlying   Mat{\'e}rn model.
In Figure \ref{fig:fcovt} (left  part) we compare $\rho_{Y^*_{\nu}}(\bm{h})$ when $\nu=3, 7$
with the correlation of the `parent' Gaussian process
$\rho(\bm{h})={\cal M}_{1.5,\alpha^*}( \bm{h})$ where $\alpha^*$ is chosen such that the practical range is $0.2$. It is apparent that when increasing the
degrees of freedom $\rho_{Y^*_{\nu}}(\bm{h})$ approaches $\rho(\bm{h})$ and that the smoothness at the origin of $\rho_{Y^*_{\nu}}(\bm{h})$ is inherited by  the smoothness of the   Gaussian correlation $\rho(\bm{h})$ when $\nu=7$ and  if $\nu=3$ then $\rho_{Y^*_{\nu}}(\bm{h})$ is not differentiable at the origin. %, as depicted in Theorem \ref{theoiii}.
On
the right side of Figure (\ref{fig:fcovt}) we compare  a kernel
nonparametric density estimation  of a  realization of $G$  and  a realization of $Y^*_{7}$ (approximately 10000 location sites in the unit square)
 using $\rho(\bm{h})={\cal M}_{1.5,\alpha^*}( \bm{h})$.

In  Figure \ref{fig:f2} (a) and (b), we compare,  from left to
right,  two realizations of $G$  with $\rho(\bm{h})={\cal M}_{0.5,
\alpha^*}(\bm{h})$ and $\rho(\bm{h})={\cal M}_{1.5, \alpha^*}(\bm{h})$
where $\alpha^*$ is chosen such that the practical range is $0.2$.
In this case, the sample paths of $G$ are zero and one  times
differentiable. From the bottom
part of Figure \ref{fig:f2} (c) and (d) it can be appreciated   that this
feature is inherited by the
associated  realizations of $Y^*_{7}$.

\begin{figure}[h!]
\begin{tabular}{cc}
\includegraphics[width=7.3cm, height=6.8cm]{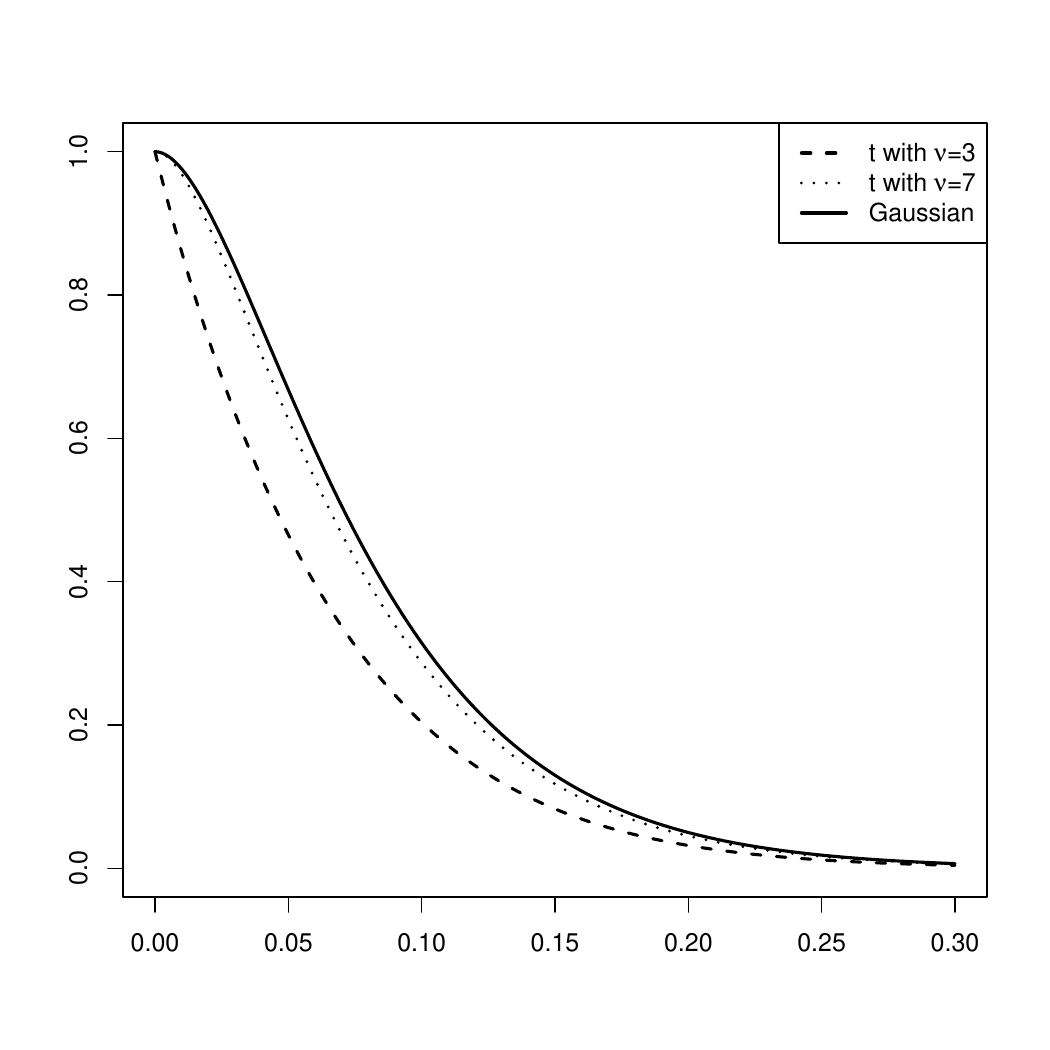} & \includegraphics[width=7.3cm, height=6.8cm]{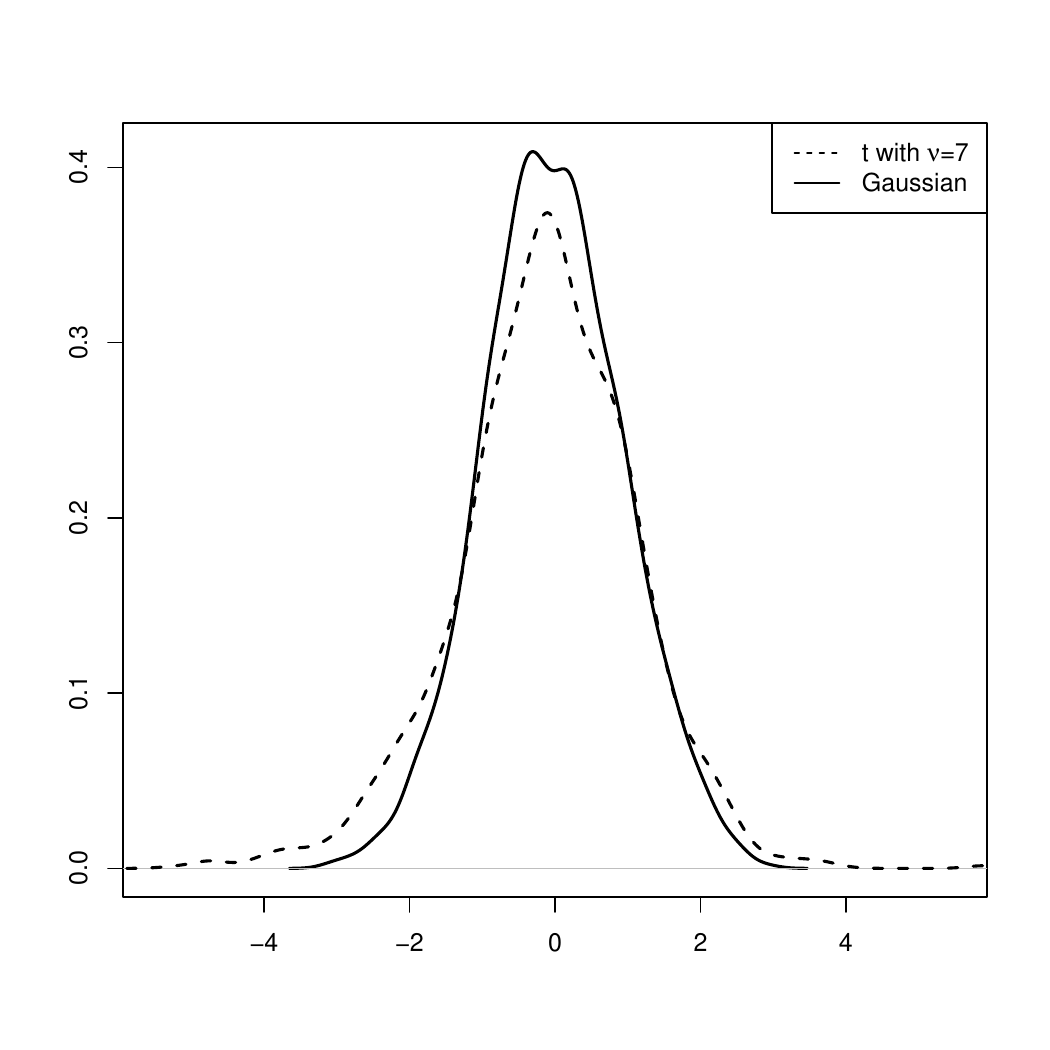}  \\
\end{tabular}
\caption{Left part:
 comparison of $\rho_{Y^*_{\nu}}(\bm{h})$, $\nu=3, 7$
with  the correlation $\rho(\bm{h})$ of the `parent' Gaussian process $G$
when $\rho(\bm{h})={\cal M}_{1.5,\alpha^*}( \bm{h})$  with $\alpha^*$ such that the practical range is $0.2$.
Right part: a comparison of  a nonparametric kernel density estimation of  realizations from $G$ and
from the $t$ process $Y^*_{7}$. \label{fig:fcovt}}
\end{figure}

\begin{figure}[h!]
\begin{tabular}{cc}
\includegraphics[width=7.1cm, height=6.4cm]{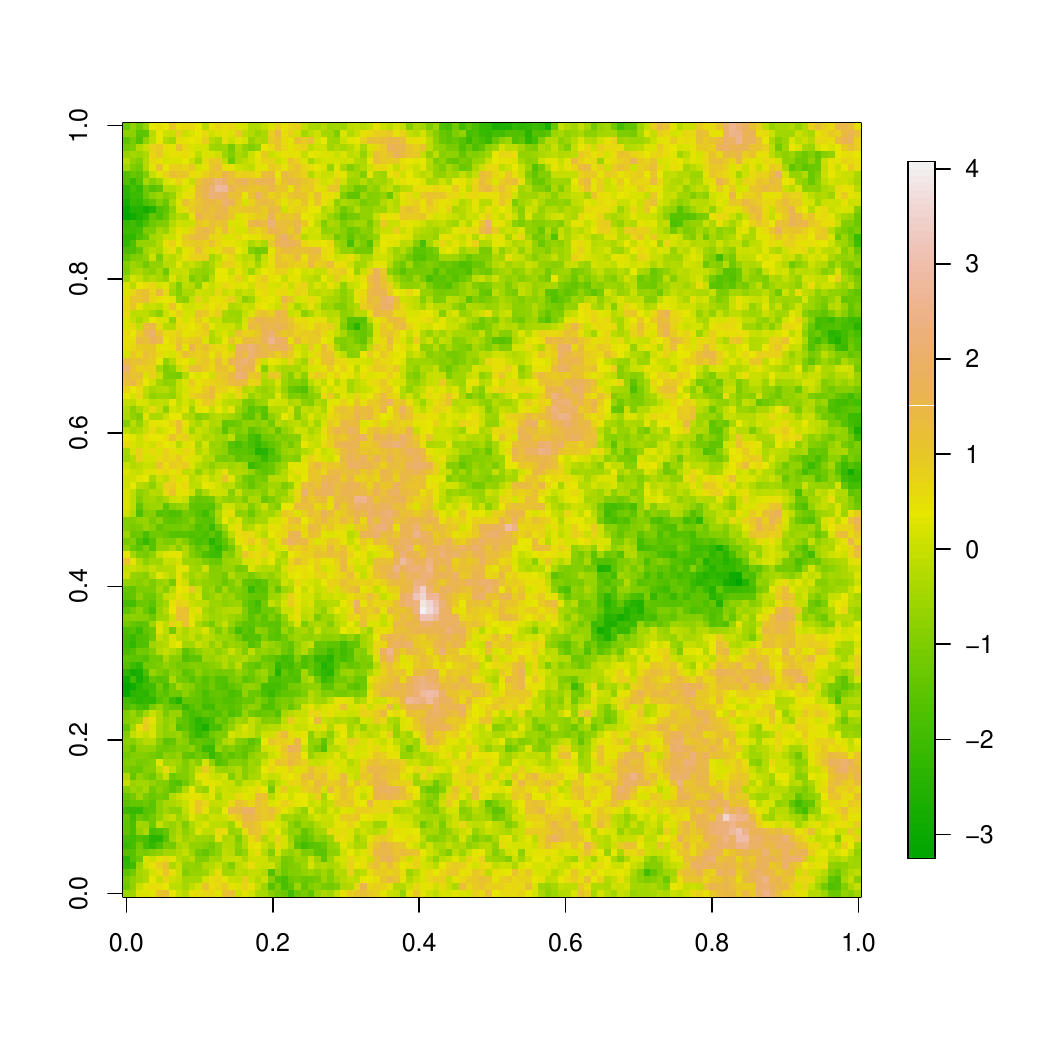} & \includegraphics[width=7.1cm, height=6.4cm]{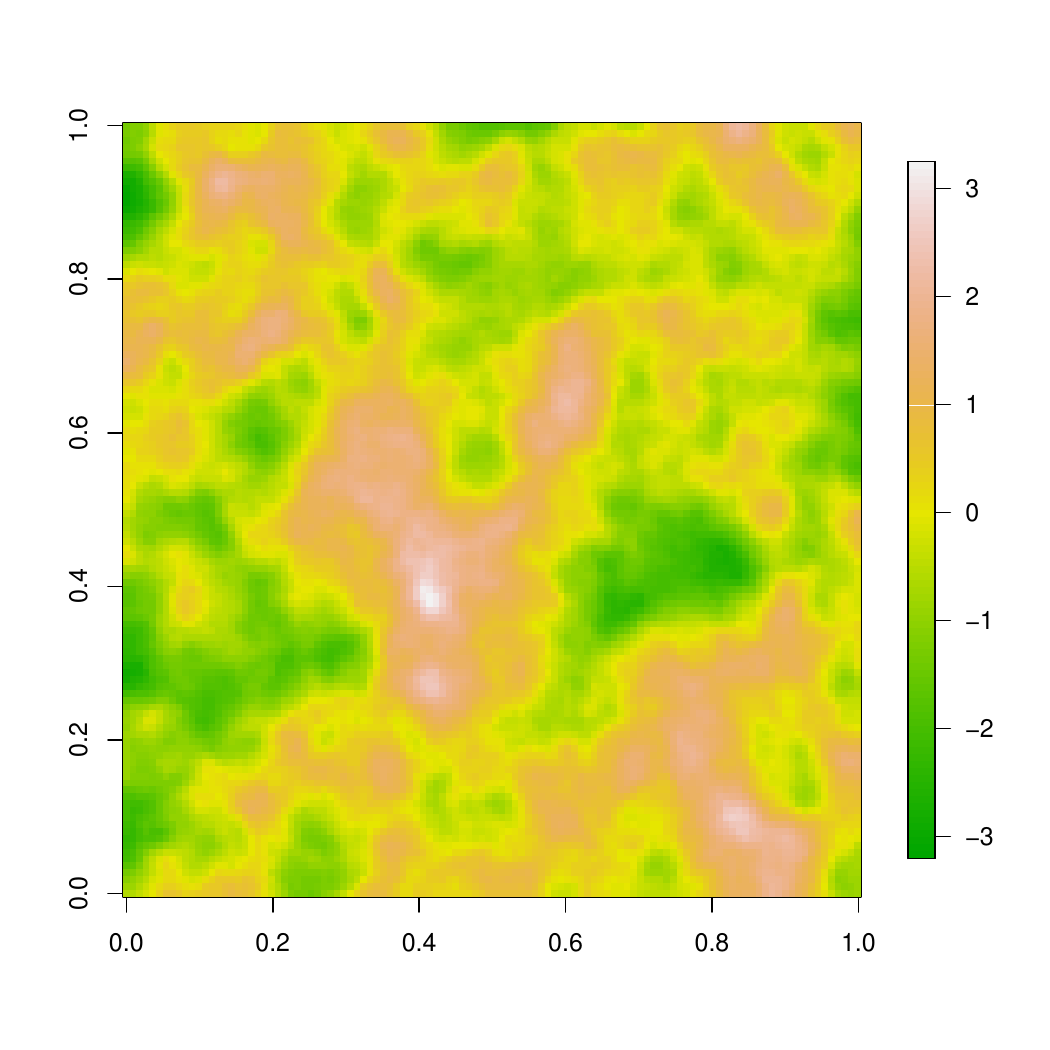}  \\
(a)&(b)\\
\includegraphics[width=7.1cm, height=6.4cm]{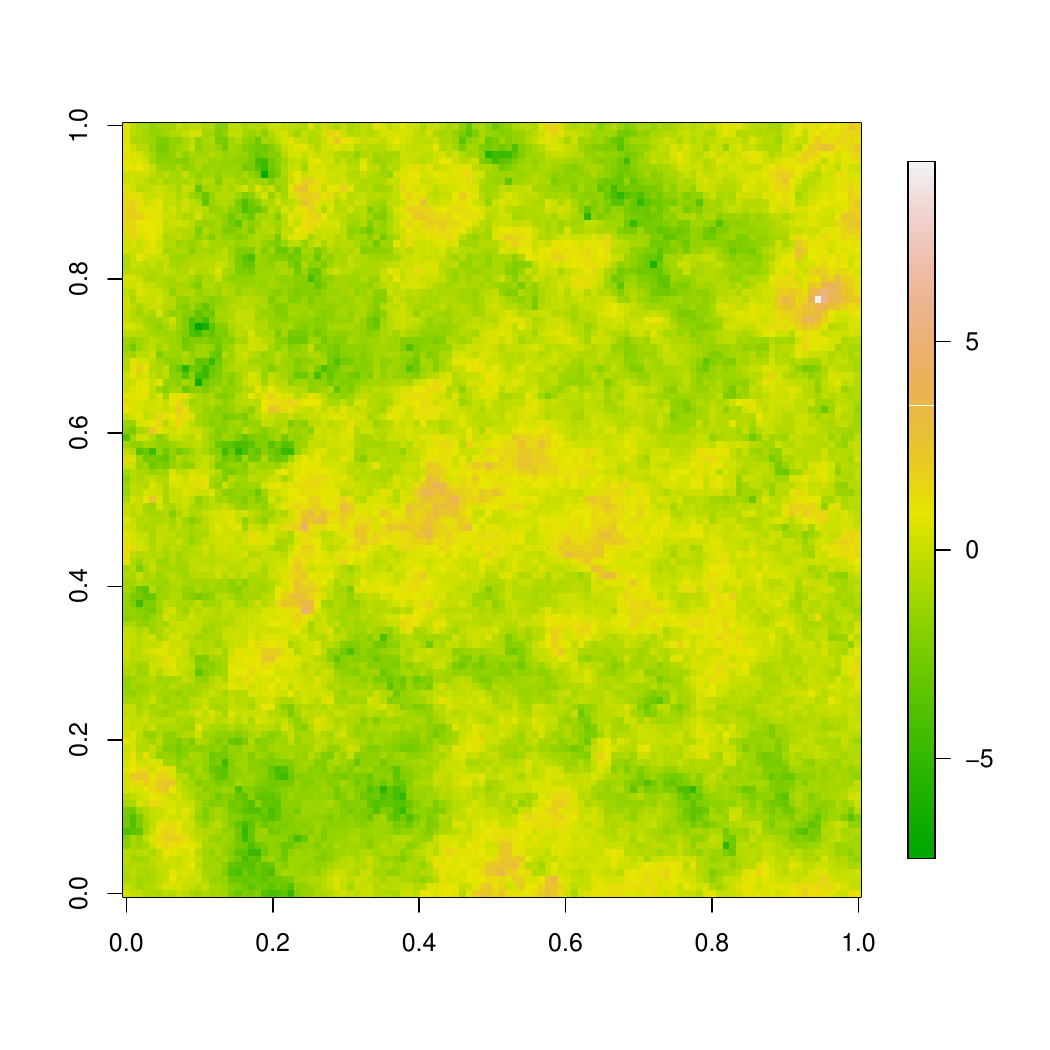} & \includegraphics[width=7.1cm, height=6.4cm]{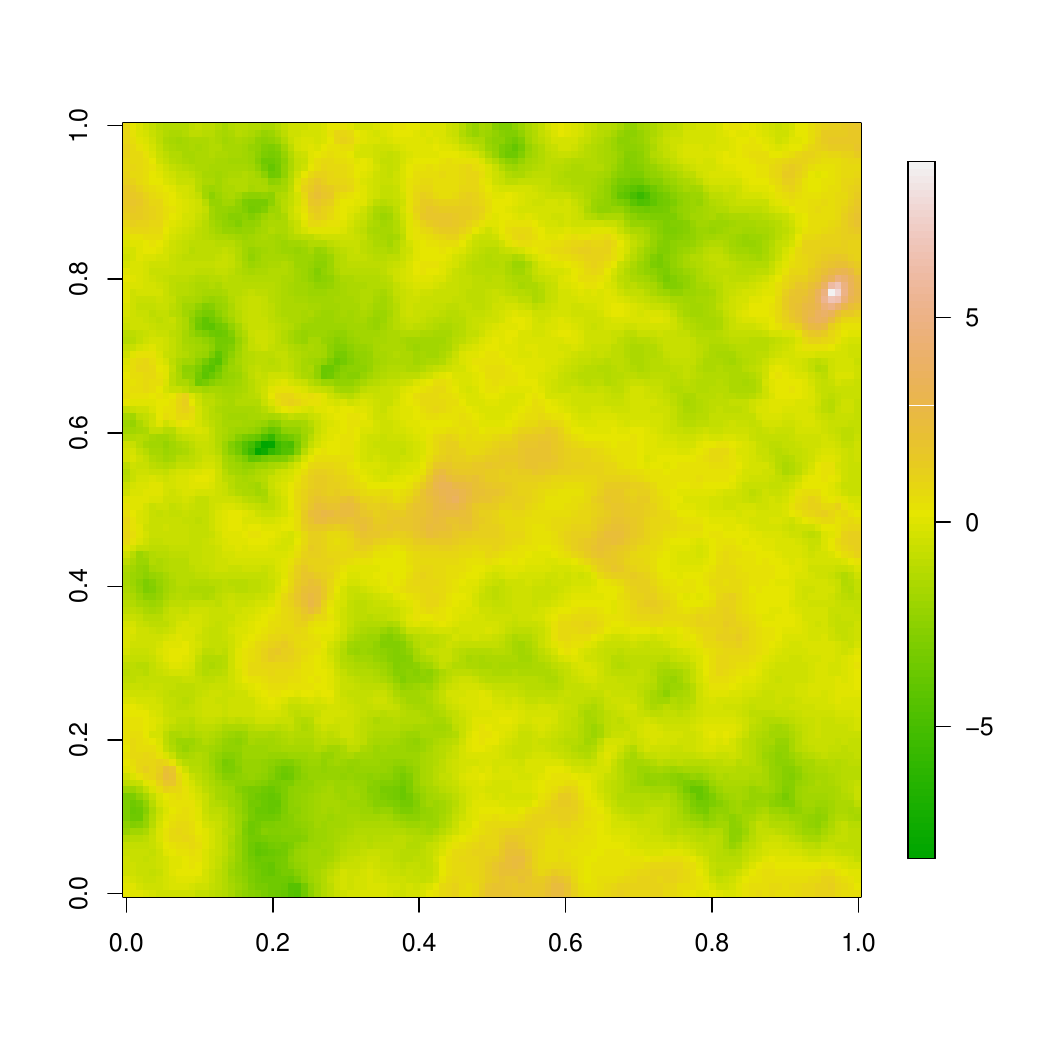}  \\
(c)&(d)\\
\end{tabular}
\caption{
Upper part: two realizations of  the `parent'  Gaussian process $G$ on $[0,1]^2$ with  (a) $\rho(\bm{h})={\cal M}_{0.5, \alpha^*}(\bm{h})$
and  (b) $\rho(\bm{h})={\cal M}_{1.5, \alpha^*}(\bm{h})$ (from left to right)
with $\alpha^*$ such that the practical range is approximatively $0.2$.
 Bottom part: (c) and (d)  associated realizations of the $t$ process $Y^*_{7}$. } \label{fig:f2}
\end{figure}

We now consider the  bivariate random vector associated with
$Y^*_{\nu}$ defined by:
\begin{equation*}
\bm{Y}^*_{\nu;ij}=\bm{W}^{-\frac{1}{2}}_{\nu;ij}\circ\bm{G}_{ij}
\end{equation*}
where $\circ$ denotes the Schur product vector.
 The following Theorem gives the pdf of $\bm{Y}^*_{\nu;ij}$ in terms of    the Appell function $F_4$.
 It can be viewed as a generalization of the generalized bivariate $t$ distribution proposed in  \cite{miller1968}.
%The proof  has been deferred to the Appendix.

%If we assume that $Z$ is a Gaussian RF, we will obtain the
%bivariate density function Student t for the process $Y$. We
%reject this restriction because of the computational complexity
%that exists if we suppose that $Z$ is a skew gaussian RF.

\begin{theo}\label{theo3}

Let $Y^*_{\nu}$, $\nu>2$ be   a standard $t$ process with underlying correlation $\rho(\bm{h})$.
Then:
\begin{footnotesize}
\begin{align}\label{pairt1}
f_{\bm{Y}^*_{\nu;ij}}(y_{i},y_j)&=\frac{\nu^{\nu}l_{ij}^{-\frac{(\nu+1)}{2}}\Gamma^2\left(\frac{\nu+1}{2}\right)}{\pi\Gamma^2\left(\frac{\nu}{2}\right)(1-\rho^2(\bm{h}))^{-(\nu+1)/2}}
F_4\left(\frac{\nu+1}{2},\frac{\nu+1}{2},\frac{1}{2},\frac{\nu}{2};\frac{\rho^2(\bm{h})y_i^2y_j^2}{l_{ij}},\frac{\nu^2\rho^2(\bm{h})}{l_{ij}}\right)\nonumber\\
&+\frac{\rho(\bm{h})y_iy_j\nu^{\nu+2}l_{ij}^{-\frac{\nu}{2}-1}}{2\pi(1-\rho^2(\bm{h}))^{-\frac{(\nu+1)}{2}}}
F_4\left(\frac{\nu}{2}+1,\frac{\nu}{2}+1,\frac{3}{2},\frac{\nu}{2};\frac{\rho^2(\bm{h})y_i^2y_j^2}{l_{ij}},
\frac{\nu^2\rho^2(\bm{h})}{l_{ij}}\right)
%\frac{2^{-2\nu}\nu^{\nu}[(y_i^2+\nu)(y_j^2+\nu)]^{-(\nu+1)/2}}{\sqrt{\pi}\Gamma^2\left(\frac{\nu}{2}\right)(1-\rho^2(\bm{h}))^{-(\nu+1)/2}}
%\sum\limits_{k=0}^{\infty}\frac{\Gamma^2\left(2\alpha_k\right)}{k!\left(\frac{\nu}{2}\right)_k\Gamma\left(2\alpha_k+\frac{1}{2}\right)}\left(\frac{\nu^2\rho^2(\bm{h})}{16(y_i^2+\nu)(y_j^2+\nu)}\right)^k\nonumber\\
%&\times&{}_2F_1\left(\alpha_k,\alpha_k;2\alpha_k+\frac{1}{2};1-\frac{\rho^2(\bm{h})y_i^2y_j^2}{(y_i^2+\nu)(y_j^2+\nu)}\right)\qquad if \quad y_i y_j >0
\end{align}
\end{footnotesize}
where $l_{ij}=[(y_i^2+\nu)(y_j^2+\nu)]$.
\end{theo}

\begin{figure}[h!]
\begin{tabular}{cc}
\includegraphics[width=7.3cm, height=6.5cm]{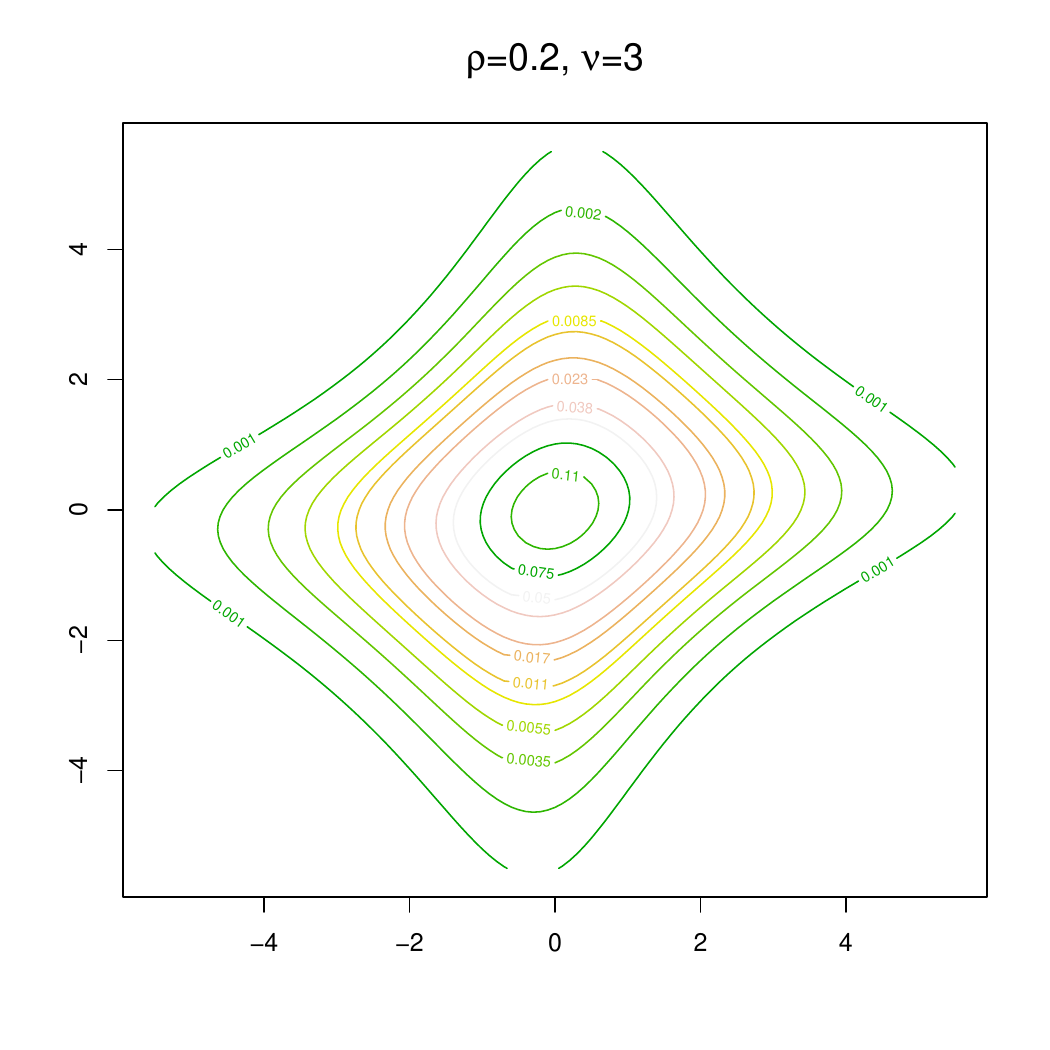} & \includegraphics[width=7.3cm, height=6.5cm]{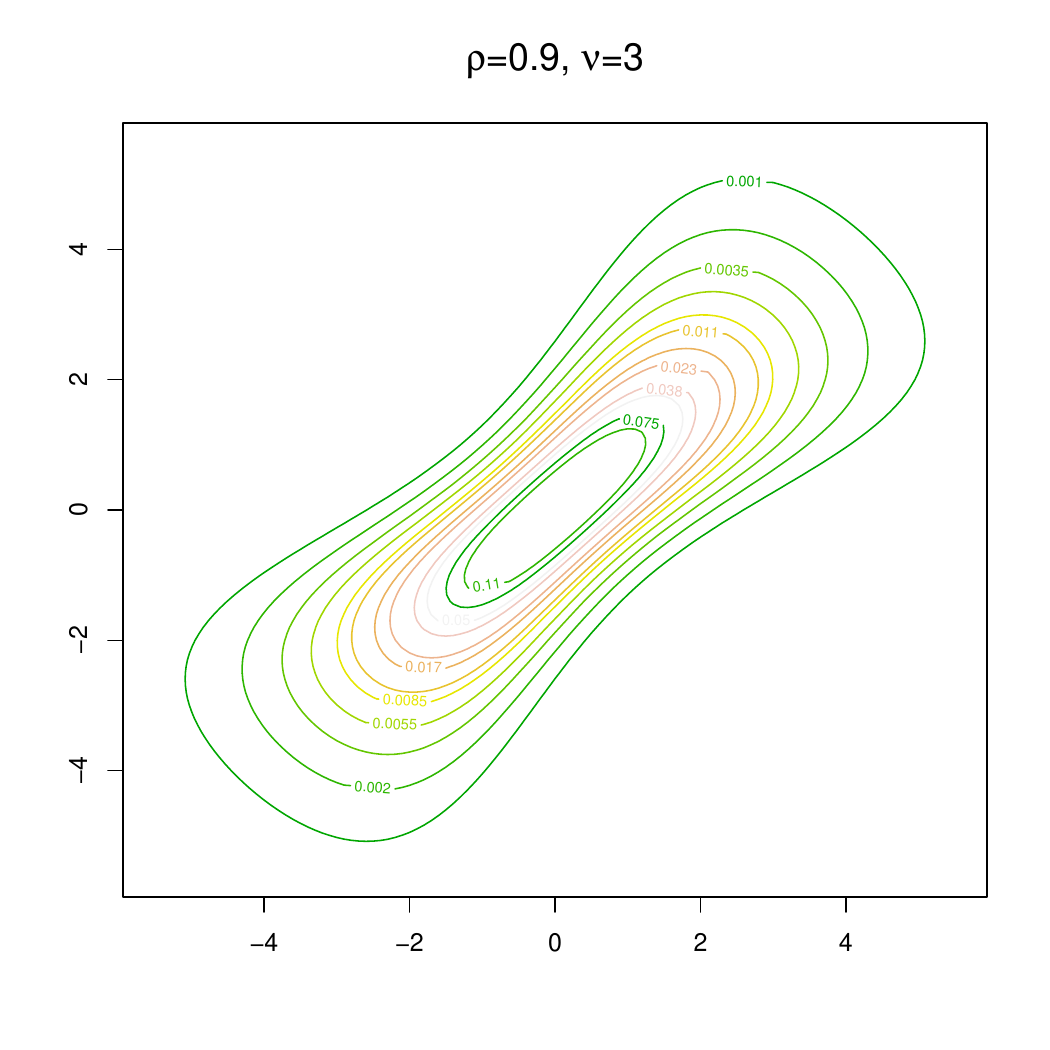} \\%& \includegraphics[width=4.75cm, height=5.5cm]{333.pdf}  \\
\includegraphics[width=7.3cm, height=6.5cm]{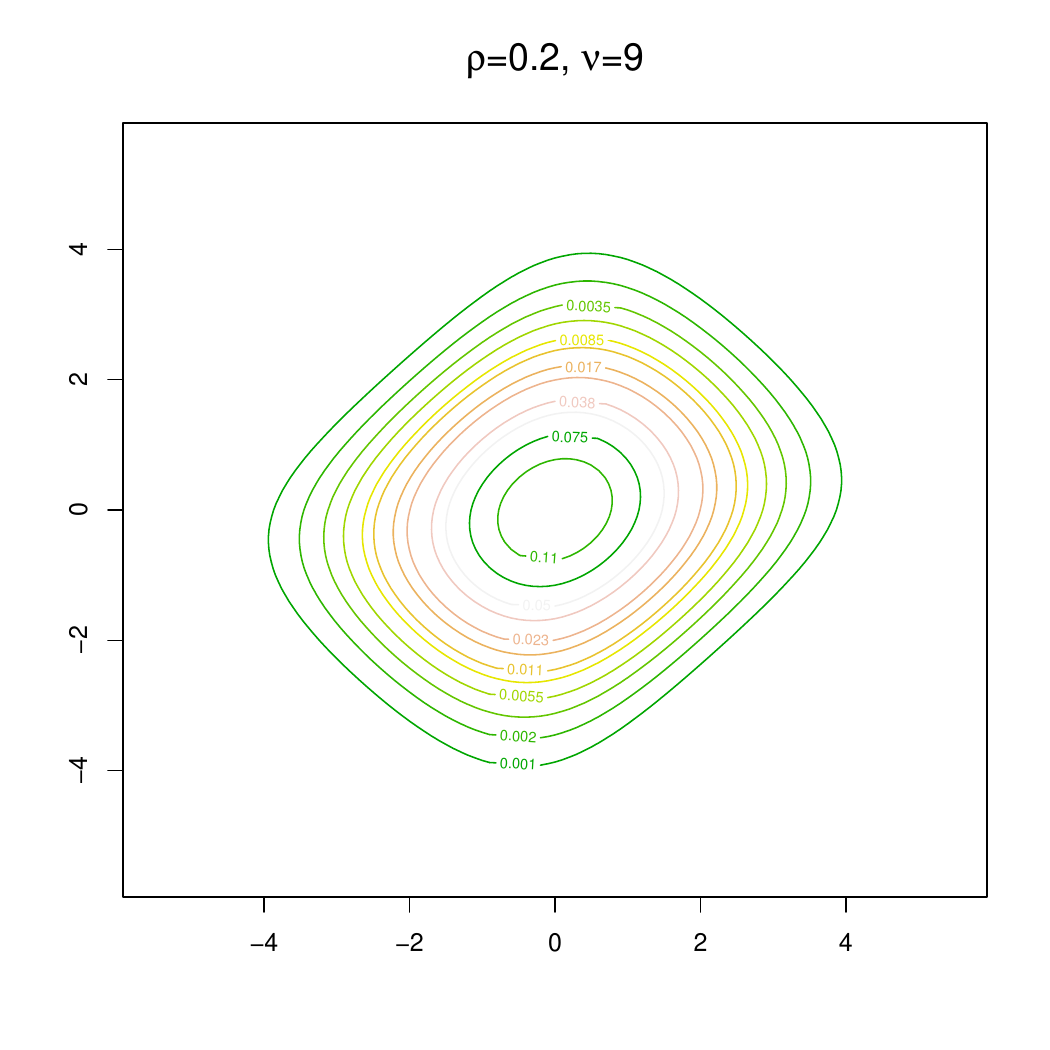} & \includegraphics[width=7.3cm, height=6.5cm]{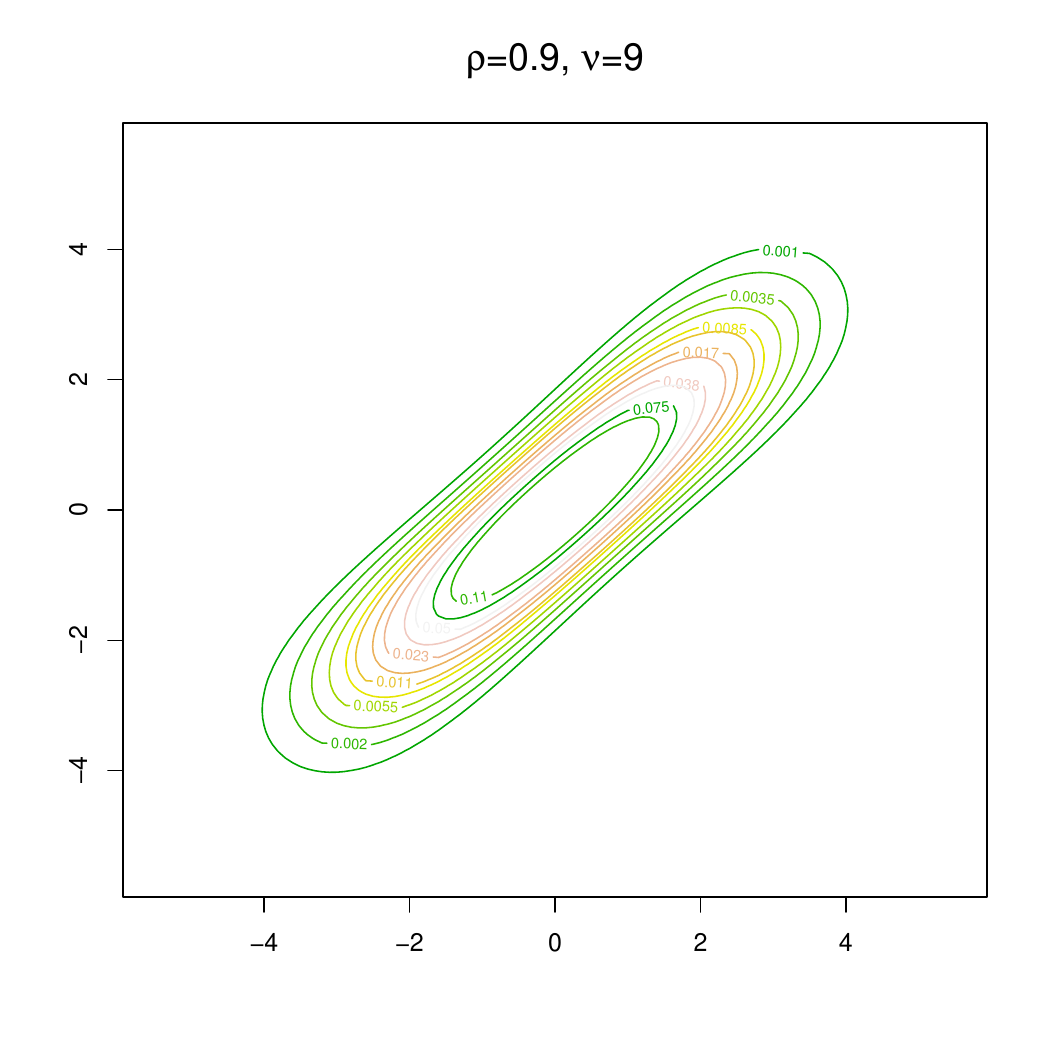} \\%& \includegraphics[width=4.75cm, height=5.5cm]{333.pdf}  \\
\end{tabular}
\caption{Contour plots of the bivariate $t$ distribution (\ref{pairt1}) when $\rho(\bm{h})=0.2, 0.9$ and $\nu=3,9$.} \label{cont}
\end{figure}

{\bf Remark 4}:
Note that $f_{\bm{Y}^*_{\nu;ij}}(y_{i},y_j)$
is  defined for  $\nu>2$  irrespectively  of the correlation function since it is obtained from a  bivariate Gamma distribution (see Remark 2).
Moreover, when $\rho(\bm{h})=0$, according to (\ref{apell4}) and using the identity ${}_2F_1(a,b;c';0)=1$,
we obtain $F_4(a,b;c,c';0,0)=1$,  and as a consequence,
%and using  $F_4(a,b;c,c';0,0)=1$ according to
%(\ref{apell4}), because
%\begin{footnotesize}
%\begin{eqnarray*}
%F_4(a,b;c,c';0,0)&=&\sum\limits_{k=0}^{\infty}\frac{(a)_{k}(b)_{k}0^k}{k!(c')_k}{}_2F_1(a+k,b+k;c;0)\nonumber\\
%&=&\sum\limits_{k=0}^{\infty}\frac{(a)_{k}(b)_{k}0^k}{k!(c')_k}={}_2F_1(a,b;c';0)=1
%\end{eqnarray*}
%\end{footnotesize}
$f_{\bm{Y}^*_{\nu;ij}}(y_{i},y_j)$ can be written as the product of
two independent $t$ random variables with $\nu$ degrees of
freedom. Thus,   zero pairwise  correlation implies pairwise independence,  as in the Gaussian case.
Figure (\ref{cont}) show the contour plots of (\ref{pairt1}) when $\nu=3, 9$ and $\rho(\bm{h})=0.2, 0.9$.
It turns out that the bivariate $t$ distribution is not elliptical and when increasing  $\nu$ the contour plots tend towards
an elliptical form. %\textcolor{red}{  (More comments? Can we show that it converge to a bivariate Gaussian distribution?)}
Finally, the bivariate density of the process $Y_{\nu}$ is easily obtained from   (\ref{kkkkk2}):
\begin{equation}\label{azz}
f_{\bm{Y}_{\nu;ij}}(y_{i},y_j)=\frac{1}{\sigma^2}f_{\bm{Y}^*_{\nu;ij}}\left(\frac{y_{i}-\mu_i}{\sigma},\frac{y_j-\mu_j}{\sigma}\right).
\end{equation}

%Note that when $\rho(\bm{h})=0$ and using the properties
%${}_2F_1(a,b;c;1)=\frac{\Gamma(c)\Gamma(c-a-b)}{\Gamma(c-a)\Gamma(c-b)}$
%and the gamma duplication formula
%$\frac{\Gamma(\nu+1+2k)}{\Gamma(\nu/2+1+k)}=\frac{2^{2k+\nu}\Gamma((\nu+1)/2+k)}{\sqrt{\pi}}$
%then
%\begin{footnotesize}
%$$f_{\bm{Y}^*_{ij}}(\bm{y}_{ij})=\frac{\nu^{\nu}[(y_i^2+\nu)(y_j^2+\nu)]^{-(\nu+1)/2}\Gamma^2\left(\frac{\nu+1}{2}\right)}{\pi\Gamma^2\left(\frac{\nu}{2}\right)}\sum\limits_{k=0}^{\infty}\frac{\left(\frac{\nu+1}{2}\right)^2_k}{k!\left(\frac{\nu}{2}\right)_k}(0)^k$$
%\end{footnotesize}
%with
%$\sum\limits_{k=0}^{\infty}\frac{\left(\frac{\nu+1}{2}\right)^2_k}{k!\left(\frac{\nu}{2}\right)_k}(0)^k={}_2F_1\left(\frac{\nu+1}{2},\frac{\nu+1}{2};\frac{\nu}{2};0\right)=1$.
%Therefore, the pdf in Theorem (\ref{theo3}) becomes the product of
%two independent Student t random variables with $\nu$ degrees of
%freedom.

\section{A stochastic process with  skew-t marginal distribution}
\label{sec:3}

In this section we first review the skew-Gaussian process proposed in \cite{Zhang:El-Shaarawi:2010}. For this process, we provide
an explicit expression for the finite dimensional distribution
generalizing previous results in \cite{ALEGRIA2017}.
Then, using this skew-Gaussian process,  we propose a  generalization of the $t$ process $Y_{\nu}$ obtaining
a new process with marginal distribution of the skew-t type \citep{Azzalin:Capitanio:2014}.

Following \cite{Zhang:El-Shaarawi:2010}   a general construction for a process
with asymmetric marginal distribution is
given by:
\begin{equation}\label{repskew}
U_{\eta}(\bm{s}):=g(\bm{s})+\eta|X_1(\bm{s})|+\omega
X_2(\bm{s}),\;\;\;\bm{s}\in A\subset \mathbb{R}^d
\end{equation}
where  $\eta\in\mathbb{R}$,  $\omega>0$ and $X_i$ $i=1,2$ are two independents copies of a process $X=\{X(\bm{s}), \bm{s} \in A\}$
with symmetric marginals.
 % Here $g(\bm{s})$ is spatially varying mean that can be specified as  $g(\bm{s})=X(\bm{s})^T\bbeta$
% where  $X(\bm{s}) \in  \R^k$
%is a vector of  covariates, $\bbeta  \in  \R^k$.
The parameters $\eta$ and $\omega$ allow  modeling the  asymmetry and variance of
the process simultaneously.

\cite{Zhang:El-Shaarawi:2010}  studied  the second-order properties of $U_{\eta}$ when $X\equiv G$.
In this case,  $U_{\eta}$ has  skew Gaussian marginal distributions \citep{Azzalin:Capitanio:2014} with pdf given by:
\begin{equation}\label{skewmar}
f_{U_{\eta}(\bm{s})}(u)=\frac{2}{(\eta^2+\omega^2)^{1/2}}\phi\left(\frac{(u-g(\bm{s}))}{(\eta^2+\omega^2)^{1/2}}\right)\Phi\left(\frac{\eta(u-g(\bm{s}))}{\omega(\eta^2+\omega^2)^{1/2}}\right)
\end{equation}
with $\E(U_{\eta}(\bm{s}))=g(\bm{s})+\eta(2/\pi)^{1/2}$,
$Var(U_{\eta}(\bm{s}))=\omega^2+\eta^2(1-2/\pi)$ and with
correlation function given:
\begin{footnotesize}
\begin{equation}\label{covskew}
\rho_{U_{\eta}}(\bm{h})=\frac{2\eta^2}{\pi\omega^2+\eta^2(\pi-2)}\left((1-\rho^2(\bm{h}))^{1/2}+\rho(\bm{h})\arcsin(\rho(\bm{h}))-1\right)+\frac{\omega^2\rho(\bm{h})}{\omega^2+\eta^2(1-2/\pi)}.
\end{equation}
\end{footnotesize}

The following theorem generalizes the results in \cite{ALEGRIA2017} and gives an explicit closed-form expression for the pdf of  the random vector $\bm{U}_{\eta}$.
%The proof can be found in the Appendix.
 %Here we suppose  $\eta\neq 0$ (the case with
%zero skewness is reduced to the Gaussian scenario)
% and let
%$\bm{s}_1,\bm{s}_2,\ldots,\bm{s}_n$ arbitrary sites in
%$\mathbb{R}^d$ with correlation matrix $\Omega$.
%$$\Omega=\left(%
%\begin{array}{cc}
%  1 & \rho_{C_{ij}}(\bm{h};\cdot) \\
% \rho_{C_{ji}}(\bm{h};\cdot) & 1 \\
%\end{array}%
%\right)$$

\begin{theo}\label{theo1}
Let $U_{\eta}(\bm{s})=g(\bm{s})+\eta|X_1(\bm{s})|+\omega
X_2(\bm{s})$ where $X_i$ $i=1,2$ are two independent  copies of $G$ the `parent' Gaussian process.
%  weakly-stationary zero mean and unit variance
%process with  correlation function $\rho(\bm{h})$.
%Given $(\bm{s}_1,\ldots, \bm{s}_n)$, $n\geq1$ locations sites, the finite dimensional distribution of $U_{\eta}$
%The pdf of $\bm{U}=(U(\bm{s}_1),\ldots,U(\bm{s}_n))^T$ for the
%multivariate skew-Gaussian RF
Then:
\begin{equation}\label{genfdskew}
f_{\bm{U}_{\eta}}(\bm{u})=2\sum\limits_{l=1}^{2^{n-1}}\phi_n(\bm{u}-\bm{\alpha};\bm{A}_l)\Phi_n(\bm{c}_l;\bm{0},\bm{B}_l)
\end{equation}

%where $\phi_n(\bm{u}-\bm{\alpha};A_l)$ denotes the n-variate
%normal density function of mean $\bm{0}$ and covariance matrix
%$A_l$, $\Phi_n(c_l,B_l)$ denotes the n-variate normal cumulative
%function of mean $0$ with covariance matrix $B_l$ and
where
\begin{align*}
  \bm{A}_l &= \omega^2\Omega+\eta^2\Omega_l \\
  \bm{c}_l &=\eta\Omega_l(\omega^2\Omega+\eta^2\Omega_l)^{-1}(\bm{u}-\bm{\alpha}) \\
  \bm{B}_l &=\Omega_l-\eta^2\Omega_l(\omega^2\Omega+\eta^2\Omega_l)^{-1}\Omega_l\\
  \bm{\alpha}&=[g(\bm{s}_i)]_{i=1}^n
\end{align*}
%for some $\Omega_l$
and the $\Omega_l$'s are correlation matrices that depend on the correlation matrix $\Omega$.
\end{theo}
Some comments are in order. First, note that $f_{\bm{U}}$ can be viewed as a generalization of the multivariate skew-Gaussian distribution proposed in
\cite{azz1996}. Second,
using Theorem (\ref{theo1}), it can be easily shown that the  consistency conditions given in
 \cite{BM2018}  are satisfied.
Third, it is apparent that  likelihood-based methods for the skew-Gaussian process are impractical from computational point of view even for a relatively small dataset.

To obtain a process with skew-$t$ marginal distributions \citep{Azzalin:Capitanio:2014},
we replace the process $G$ in (\ref{kkkkk2})  with the process $U_{\eta}$.
Specifically, we consider a process
$S_{\nu,\eta}=\{S_{\nu,\eta}(\bm{s}), \bm{s} \in A\}$  defined  as
\begin{equation}\label{kkkkkskew}
S_{\nu,\eta}(\bm{s}):=\mu(\bm{s})+\sigma
W_{\nu}(\bm{s})^{-\frac{1}{2}}U_{\eta}(\bm{s})
\end{equation}
%This type of construction is similar to (\ref{kkkkk})  but now we replace a  zero mean  unit variance Gaussian process with a  standard Skew-Gaussian RF.
%As in the previous Section, we focus on  $S^*_{\nu,\eta}(\bm{s})$ %(the standardized version of $G_{\nu,\eta}(\bm{s}))$
%\begin{equation}\label{opo2}
%S^*_{\nu,\eta}(\bm{s}):=(G_{\nu,\eta}(\bm{s})-\mu(\bm{s}))/\sigma= R_{\nu}(\bm{s})U_{\eta}(\bm{s}).
%\end{equation}
%and we study  its second order properties.
%The  associated properties of   $G_{\nu,\eta}$   are easily obtained by means of  transformation of location and scale.
where
$W_{\nu}$ and $U_{\eta}$ are supposed to be independent.
In (\ref{repskew})  we assume $g(\bm{s})=0$ and
$\eta^2+\omega^2=1$. The pdf of  the  marginal distribution of $S^*_{\nu,\eta}$
is given by: % \citep{Azzalin:Capitanio:2014}:
\begin{equation}\label{fduskewt}
f_{S^*_{\nu,\eta}(\bm{s})}(g)=2f_{Y^{*}_{\nu}(\bm{s})}(g;\nu)F_{Y^{*}_{\nu}(\bm{s})}\left(\eta
g\sqrt{\frac{\nu+1}{\nu+g^2}};\nu+1\right)
\end{equation}
%&=&\int\limits_{\mathbb{R}_+}2\phi(g\sqrt{w})\Phi(\eta
%g\sqrt{w})\sqrt{w}f_W(w)dw\nonumber\\
%&=&\frac{2}{\Gamma\left(\frac{\nu}{2}\right)\sqrt{\pi
%\nu}}\left(1+\frac{g^2}{\nu}\right)^{-\frac{1}{2}(\nu+1)}\int\limits_{\mathbb{R}_+}e^uu^{\frac{1}{2}(\nu-1)}\Phi\left(\eta
%g\sqrt{\frac{2u}{\nu+g^2}}\right) du\nonumber\\
%&=&2t(g;\nu)T\left(\eta g\sqrt{\frac{\nu+1}{\nu+g^2}};\nu+1\right)
with
$\E(S^*_{\nu,\eta}(\bm{s}))=\frac{\sqrt{\nu}\Gamma\left(\frac{\nu-1}{2}\right)\eta}{\sqrt{\pi}\Gamma\left(\frac{\nu}{2}\right)}$,
and
$Var(S^*_{\nu,\eta}(\bm{s}))=\left[\frac{\nu}{\nu-2}-\frac{\nu\Gamma^2\left(\frac{\nu-1}{2}\right)\eta^2}
{\pi\Gamma^2\left(\frac{\nu}{2}\right)}\right]$.

 If
$\eta=0$, (\ref{fduskewt}) reduces to a marginal  $t$ density
given in (\ref{ut}) and  if $\nu\to \infty$, (\ref{fduskewt})
converges to a skew-normal distribution. Moreover,
coupling (\ref{CC})   and (\ref{covskew})   the correlation function of the skew-$t$
process is given by:
%\begin{theo}\label{theo0pp}
%Let $G_i$
%$i=1,\ldots,\nu$, $\nu>2$, $X_1$, $X_2$
 %independent copies of a
%zero mean unit variance  weakly stationary Gaussian process with correlation $\rho(\bm{h})$.
%Let $U_{\eta}(\bm{s})=\eta|X_1(\bm{s})|+
%X_2(\bm{s})$
%and  $W_{\nu}(\bm{s})=\sum_{i=1}^\nu G_i(\bm{s})^2/\nu$.
%Then the correlation function of $S^*_{\nu,\eta}(\bm{s})= W^{\frac{1}{2}}_{\nu}(\bm{s})U_{\eta}(\bm{s})$ is given by
\begin{footnotesize}
\begin{align}\label{CC99}
\rho_{S^*_{\nu,\eta}}(\bm{h})&=a(\nu,\eta)\left[{}_2F_1\left(\frac{1}{2},\frac{1}{2};\frac{\nu}{2};\rho^2(\bm{h})\right)
\left\{(1+\eta^2(1-\frac{2}{\pi}))\rho_{U_{\eta}}(\bm{h})+\frac{2\eta^2}{\pi}\right\}-\frac{2\eta^2}{\pi}\right],
\end{align}
\end{footnotesize}
%\end{theo}
where $a(\nu,\eta)=\frac{\pi(\nu-2)\Gamma^2\left(\frac{\nu-1}{2}\right)}{2\left[\pi\Gamma^2\left(\frac{\nu}{2}\right)(1+\eta^2)
-\eta^2(\nu-2)\Gamma^2\left(\frac{\nu-1}{2}\right)\right]}$. Note that $\rho_{S^*_{\nu,\eta}}(\bm{h})=\rho_{S^*_{\nu,-\eta}}(\bm{h})$ that is, as  in the skew-Gaussian process $U_{\eta}$,   the correlation is invariant with respect to positive or
negative asymmetry  and using similar arguments of  Theorem \ref{theoiii} point e), it   can be shown that $\lim\limits_{\nu \to
\infty } \rho_{S^{*}_{\nu,\eta}}(\bm{h})=\rho_{U_{\eta}}(\bm{h})$.

Finally, following the steps of the proof of    Theorem \ref{theoiii},
it can be shown that properties  $a)$, $b)$, $c)$ and $d)$  in  Theorem \ref{theoiii} are true for the skew-t process $S^*_{\nu,\eta}$.

Figure  \ref{fig:fcovt222}, left part,
compares $\rho_{S^*_{7,0.9}}(\bm{h})$  and $\rho_{S^*_{7,0}}(\bm{h})=\rho_{Y^*_{7}}(\bm{h})$
with  the underlying correlation $\rho(\bm{h})={\cal  GW}_{0.3,1,5}(\bm{h})$.
The right part  shows a realization of $S^*_{7,0.9}$.

\begin{figure}[h!]
\begin{tabular}{cc}
   \includegraphics[width=7.1cm, height=6.8cm]{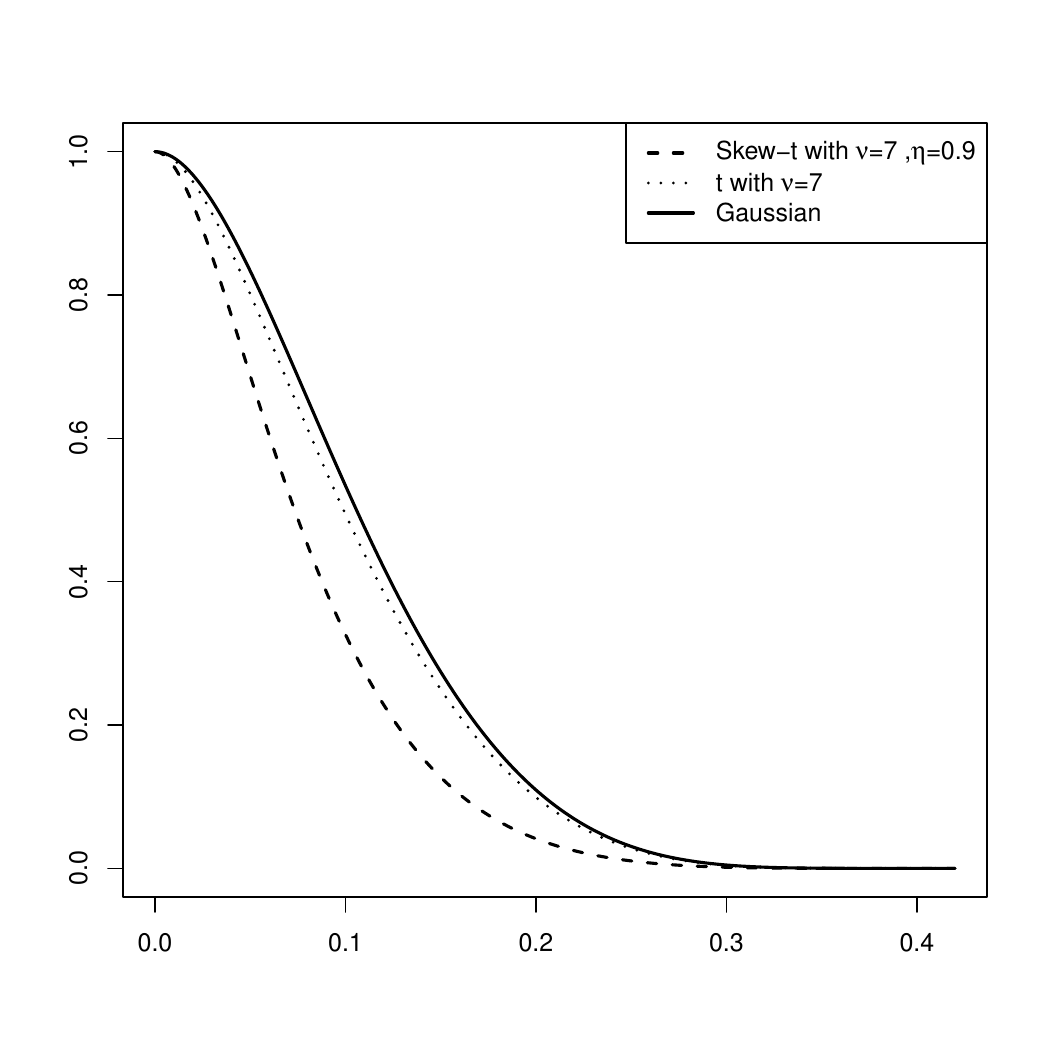}& \includegraphics[width=7.1cm, height=6.8cm]{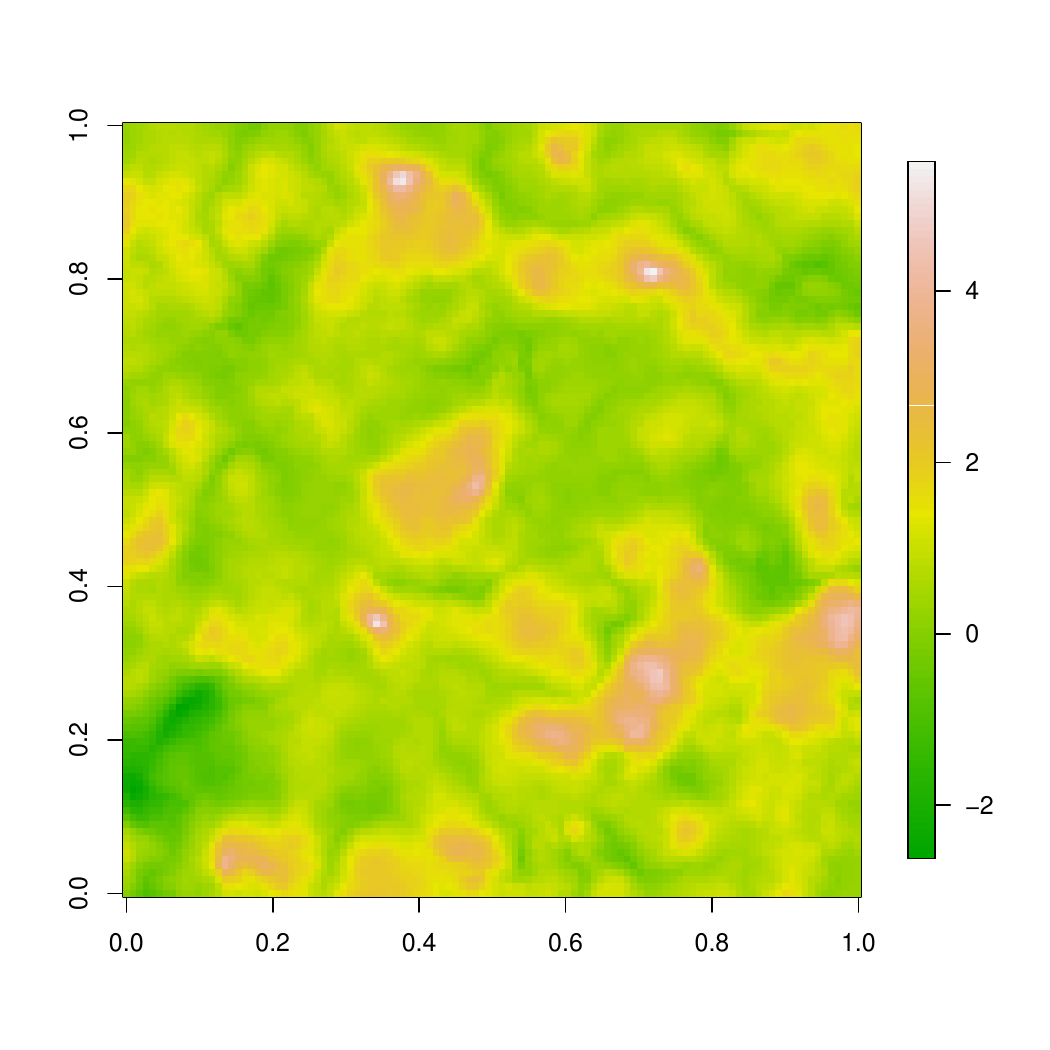}   \\
\end{tabular}
\caption{From left to right:
a) comparison between $\rho_{S^*_{7,0.9}}(\bm{h})$, $\rho_{S^*_{7,0}}(\bm{h})=\rho_{Y^*_{7}}(\bm{h})$
and the underlying correlation $\rho(\bm{h})={\cal  GW}_{0.3,1,5}(\bm{h})$;
b)
 a  realization from   $S^*_{7,0.9}(\bm{s})$. \label{fig:fcovt222}}
\end{figure}

\section{Numerical examples}\label{sec:4}

%The main goals of this section are twofold: on the one hand,
In this Section
we analyze the performance of the $wpl$ method   when estimating the
$t$ process  assuming   $\nu$ known or unknown. Following Remark 1  in Section \ref{sec:2}, we consider the cases when $\nu>2$ and $\nu=3,4,\ldots$.
In the latter case, we give a practical solution
for fixing the degrees of freedom parameter to a positive integer value
through a two-step estimation.
%In our simulation study we consider estimation of  $t$ processes observed on a subset of  the real line and of  the spatial and spatio-temporal
%euclidean spaces.

We also compare the performance of the $wpl$ using the bivariate $t$ distribution  (\ref{pairt1}) with
a misspecified Gaussian
standard and weighted pairwise likelihood. % and $wpl$  by assuming    $\rho_{Y^*_{\nu}}(\bm{h})$ as correlation function.
Finally,  we compare
the performance of the optimal linear predictor of  the $t$ process  using (\ref{CC})
 versus the optimal predictor of  the Gaussian process.

\subsection{Weighted pairwise likelihood estimation}\label{opop}
Let   $(y_1,\ldots,y_n)^T$
 be a  realization of the $t$ random process $Y_{\nu}$ defined
in equation (\ref{kkkkk2}) observed
   at distinct  spatial locations $ \bm{s}_1,\ldots, \bm{s}_n$, $ \bm{s}_i  \in A$ and
  let $\bm{\theta}=(\bm{\beta}^T, \nu, \sigma^2, \bm{\alpha}^T)$ be the vector of unknown parameters where $\bm{\alpha}$ is the vector parameter associated with the   underlying correlation model.
The method of $wpl$, \citep{Lindsay:1988,Varin:Reid:Firth:2011}  combines  the bivariate distributions of  all possible distinct pairs of observations.
 The  pairwise likelihood function
 is given by
\begin{equation}
\label{pln}
\operatorname{pl}(\bm{\theta}):=\sum_{i=1}^{n-1} \sum_{j=i+1}^{n}\log (f_{\bm{Y}_{\nu;ij}}(y_{i},y_j;\bm{\theta})) {c_{ij}}
\end{equation}
where
$f_{\bm{Y}_{\nu;ij}}(y_{i},y_j;\bm{\theta})$ is the bivariate
density  in (\ref{azz})  and
$c_{ij}$ is a nonnegative suitable weight.  The choice of  cut-off weights, namely

\begin{equation} \label{Wei}
    c_{ij}=
    \begin{cases}1
    &\|\bm{s}_i-\bm{s}_j\| \le d_{ij}\\
    0 & \text{otherwise} \end{cases},
    \end{equation}
 for a positive value of $d_{ij}$, can be  motivated by its simplicity and by observing that the dependence  between observations that are distant  is weak.  Therefore, the use of all  pairs may  skew the information confined in pairs of near observations
\citep{Bevilacqua:Gaetan:2015,Joe:Lee:2009}.
The maximum $wpl$  estimator is given by
\begin{equation*}
\widehat{\bm{\theta}}:=\operatorname{argmax}_{\bm{\theta}}\, \operatorname{pl}(\bm{\theta})
\end{equation*}
and, arguing as in  \citet{Bevilacqua:Gaetan:Mateu:Porcu:2012} and \citet{Bevilacqua:Gaetan:2015},
under some mixing conditions of the $t$ process,
it can be shown that, under increasing domain asymptotics, $\widehat{\bm{\theta}}$ is consistent and asymptotically Gaussian
with  the asymptotic covariance matrix   given by $\mathcal{G}^{-1}_n(\bm{\theta})$
the inverse of the Godambe information
$\mathcal{G}_n(\bm{\theta}):=\mathcal{H}_n(\bm{\theta})\mathcal{J}_n(\bm{\theta})^{-1}\mathcal{H}_n(\bm{\theta}),
$
where
$\mathcal{H}_n(\bm{\theta}):=\E[-\nabla^2 \operatorname{pl}(\bm{\theta})]$ and $\mathcal{J}_n(\bm{\theta}):={\mbox{Var}}[\nabla \operatorname{pl}(\bm{\theta})]$.
Standard error estimation can be obtained considering the square root diagonal  elements of $\mathcal{G}^{-1}_n(\widehat{\bm{\theta}})$.
Moreover, model selection can be performed by considering  two information criterion, defined as
\begin{equation*}
\mbox{PLIC}:= - 2\operatorname{pl}(\hat{\bm{\theta}})  + 2\mathrm{tr}(
\mathcal{H}_n(\hat{\bm{\theta}})\mathcal{G}^{-1}_n(\hat{\bm{\theta}})), \quad \mbox{BLIC}:= - 2\operatorname{pl}(\hat{\bm{\theta}})  +log(n) \mathrm{tr}(
\mathcal{H}_n(\hat{\bm{\theta}})\mathcal{G}^{-1}_n(\hat{\bm{\theta}}))
\end{equation*}
which are composite likelihood version of
the Akaike information
criterion (AIC) and Bayesian information
criterion  (BIC)  respectively  \citep{Varin:Vidoni:2005,xao:2010}.
Note that, the computation of standard errors,  \mbox{PLIC} and  \mbox{BLIC} require  evaluation of the matrices $\mathcal{H}_n(\hat{\bm{\theta}})$ and $\mathcal{J}_n(\hat{\bm{\theta}})$. However, the
evaluation  of $\mathcal{J}_n(\hat{\bm{\theta}})$ is  computationally
unfeasible for large datasets and  in this case subsampling
techniques can be used in order to estimate $\mathcal{J}_n(\bm{\theta})$ as in \cite{Bevilacqua:Gaetan:Mateu:Porcu:2012} and \cite{Heagerty:Lele:1998}.
A straightforward  and  more robust alternative    is  parametric bootstrap estimation of $\mathcal{G}^{-1}_n(\bm{\theta})$
\citep{Bai:Kang:Song:2014}.
We adopt the second strategy in Section \ref{sec:5}.

\subsection{Performance of the weighted pairwise likelihood estimation}
Following \cite{Ciccio:Monti:2011} and \cite{AR22}
we consider  a reparametrization  for the $t$ process   by using the inverse of  degrees of freedom,
$\lambda=1/\nu$. % \textcolor{red}{We need to justify this parametrization. Reinaldo?}
In the standard i.i.d case this kind of parametrization has proven  effective for solving some problems associated with the singularity of the Fisher information matrix
associated to the original parametrization.
Here we consider two possible  scenarios i.e. a $t$ process observed on a  subset of $\R$ and $\R^2$. % and $\R^2\times \R$:
%In the first scenario
%for the $t$ and two-piece t processes.
%In the second scenario, we focus on $d=2$  and we estimate  all the parameters except for the parameter  $\lambda$ that we consider  known and fixed.
%Specifically:

\begin{enumerate}

\item

We consider points  $ s_i  \in A=[0,1]$, $i=1,\ldots,N$ and an exponential correlation function for the `parent' Gaussian process.
Then, according to Remark 1 in section \ref{sec:2}, the $t$ process is well-defined  for $0<\lambda <1/2$
 and in this specific case all the parameters can be jointly estimated.
We simulate,  using Cholesky decomposition, $500$ realizations of a $t$ process
observed on a regular transect $s_{1}=0$, $s_2=0.002, \ldots ,s_{501}=1$.
We consider two mean regression  parameters,   that is,
$\mu(s_i )=\beta_0+\beta_1 u( s_i )$
with
$\beta_0=0.5$, $\beta_1=-0.25$ where $u( s_i )$
is a realization from a $U(0,1)$. Then we  set $\lambda= 1/\nu$, $\nu=3, 6, 9$   and $\sigma^2=1$.

As correlation model we consider $\rho(h)={\cal M}_{\alpha,0.5}(h)=e^{-|h|/\alpha}$
 with $\alpha=0.1/3$ and in the $wpl$
estimation we consider a cut-off weight
function with $d_{ij}=0.002$.  Table \ref{tab:tab44} shows the  bias and mean square error
(MSE) associated with $\lambda$, $\beta_0$, $\beta_1$, $\alpha$ and $\sigma^2$.

%\textcolor{red}{Moreover.   Figure ..  contain the boxplots of the estimates of the $t$  process  as illustrative examples.}

\item
We consider points  $ \bm{s}_i  \in A=[0,1]^2$, $i=1,\ldots,N$. Specifically,
we simulate,  using Cholesky decomposition, $500$ realizations of a $t$ process
observed at   $N=500$ spatial  location sites  uniformly distributed in the unit square.
Regression,  variance and (inverse of) degrees of freedom  parameters have been set as in the first scenario.
As an isotropic parametric
correlation model,  %$\rho(\bh)={\cal M}_{\alpha,0.5}(\bm{h})=e^{-||\bh||/\alpha}$ and
$\rho(\bm{h})={\cal
GW}_{\alpha,0,4}(\bm{h})$ %=\left(1-||\bh||/\alpha\right)^4_+$
with $\alpha=0.2$ is considered. In the $wpl$ estimation we consider a
cut-off weight function with $d_{ij}=0.05$ and
for each simulation we estimate with $wpl$,
assuming the degrees of freedom are fixed and known.

We also  consider the  more realistic case
when the (inverse of) degrees of freedom are supposed to be unknown.
Recall that from Remark 1,  $\nu$ must be fixed to a positive integer  $\nu= 3,4, \ldots$ in order to guarantee the existence of the $t$ process.
A brute force approach considers different $wpl$ estimates using a fixed  $\lambda= 1/\nu$, $\nu=3, 4, \ldots$.
and then simply keeps the  estimate with the  best $PLIC$ or $BLIC$. % in terms of $wpl$ function.
We propose  a computationally easier approach by  considering a two-step method. In the first step,  we estimate all the parameters including $0<\lambda<1/2$
maximizing  the  $wpl$ function.
 %can be obtained % assuming independent data.
%In the second step a $wpl$ estimation is performed fixing the inverse of degree of parameters equal to
%the  rounded estimation (with respect to $\nu$) obtained at first step.
%A third approach also consider
 %a two-step method but in this case in the first step
%maximizing  the  $wpl$ function.
 %without any restriction. % on $0<\lambda <1/2$.
This  is possible since the bivariate $t$ distribution is well  defined for $0<\lambda <1/2$ (see Remark 4).
In the second step
$\nu$ is fixed equal to the rounded value of $1/\widehat{\lambda}_1$
where $\widehat{\lambda}_1$ is the estimation at first step.
(If at the  first step, the estimation of $1/\widehat{\lambda}_1$ is lower than $2.5$, then it is rounded to $3$).
%Note the estimation obtained at first is just a device in
 Table \ref{tab:tab411}  shows  the bias and MSE
associated with
$\beta_0$, $\beta_1$, $\alpha$ and $\sigma^2$
%regression,  variance and
%correlation  parameters
when estimating
with $wpl$, assuming (the inverse of) degrees of freedom    1) known and fixed, and
2) unknown and fixed using a two-step estimation
and Figure  \ref{tpsta} shows the boxplots of the  $wpl$ estimates for the case 1) and 2).

\end{enumerate}
%\item
%We consider points    $ (\bm{s}_i,t_k)  \in A=[0,1]^2\times[0.5,10]$, $i=1,\ldots,N$, $k=1,\ldots,T$.
%Specifically,
%we simulate,  using Cholesky decomposition, $500$ realizations of a $t$ process
%observed at   $N=80$ spatial  location sites uniformly distributed in the unit square and $T=20$ temporal instants $t_{1}=0.5$, $t_2=1$, $\ldots$, $t_{20}=10$.
%As spatially isotropic and temporally symmetric  correlation model  we consider
%a special case of the nonseparable class proposed in \cite{PbG2018}
%\begin{equation} \label{wendland22}
 %\rho(\bm{h},u) = \frac{1}{\gamma(u/\alpha_t)^{2.5}}  {\cal
%GW}_{\alpha_s\gamma(u/\alpha_t)^{0.5},0,4}(\bm{h})  % \left ( 1-\frac{||\bh||}{\alpha_s\gamma(u/\alpha_t)^{0.5}}\right )_{+}^{4}
 %\end{equation}
%where $\gamma(u)=(1+|u|)$, $\alpha_t>0$ is a temporal scale parameter and  $\alpha_s>0$ is a spatial compact support.
%Regression,  variance and (inverse of) degrees of freedom  parameters have been set as in the first two  settings, and additionally,
%we set $\alpha_s=0.3$ and $\alpha_t=0.5$.
 %In the $wpl$ estimation we consider a
%spacetime cut off weight function with $d_{ij}=0.05$ and $d_{lk}=0.5$.
 %Table \ref{tab:tab419}  shows the bias and MSE
%associated with $\beta_0$, $\beta_1$, $\alpha_s$, $\alpha_t$ and $\sigma^2$
%when estimating
%with $wpl$ assuming (the inverse of) degrees of freedom  1) known and fixed,
%2) unknown and fixed using the two-step estimation depicted  in the second scenario.
%Finally, Figure  \ref{tpstast} shows the boxplots of the  $wpl$ estimates for the case 1) and 2).
%\end{enumerate}

As a  general comment, the distribution of the estimates are quite
symmetric, numerically stable and with very few outliers for the three scenarios.
In Scenario 1, the MSE of  $\lambda=1/\nu$ slightly decreases when increasing $\nu$.
Moreover, in Table \ref{tab:tab411}, % and  \ref{tab:tab419},
 it can be appreciated that only the  estimation of  $\sigma^2$
is affected when considering  a two step estimation. Specifically, the MSE of $\sigma^2$
  slightly increases with respect  to the one-step estimation, $i.e.$, when the degrees of freedom are supposed to be known.

%Therefore, the precision of
%estimates of the regression parameters benefit of the increment of
%the shape parameter.

%With respect to the distribution of the weighted pairwise
%likelihood estimates, the figure \ref{tpsta}, as an example,
%depicts the boxplots of the estimates for a  $t$ random
%field with $\beta_0=0.5$, $\beta_1=-0.25$, $\eta=0.5$, using a
%Askey correlation model with $\alpha=0.15$ and $\sigma^2=1$.

%%%%%%%%%%%%%%%%%%%%%%%%%%%%%%%%%%%%
\begin{table}[!hbtp]
\begin{center}
\scalebox{0.65}{
  \begin{tabular}{|r|rr|rr|rr|}
  \hline \multicolumn{1}{|c|}{$\lambda$} & \multicolumn{2}{c|}{$1/3$} &
\multicolumn{2}{c|}{$1/6$} &\multicolumn{2}{c|}{$1/9$}\\
\hline
\cline{1-7}
  &  Bias & MSE& Bias & MSE & Bias & MSE \\
\hline

$\hat{\lambda}$&$-0.01022$&$0.00321$&$-0.01033$&$0.00215$&$-0.00924$&$0.00172$\\
$\hat{\beta}_0$&$-0.01466$&$0.06585$&$-0.00559$&$0.06154$&$0.00028$&$0.06532$\\
$\hat{\beta}_1$&$-0.00099$&$0.00062$&$-0.00232$&$0.00049$&$-0.00193$&$0.00049$\\
$\hat{\alpha}$&$-0.00213$&$0.0006$&$-0.00278$&$0.0007$&$-0.00189$&$0.0008$\\
$\hat{\sigma}^2$&$-0.01064$&$0.07493$&$-0.05502$&$0.06902$&$-0.02677$&$0.07052$\\
 \hline
\end{tabular}
}
\end{center}
\caption{
 Bias and MSE when estimating with $wpl$  the  $t$ process
with  $\lambda=1/\nu$, $\nu=3,6,9$  and exponential correlation function (Scenario 1).} \label{tab:tab44}
\end{table}

%%%%%%%%%%%%%%%%%%%%%%%%%%%%%%%%%%%%
\begin{table}[!hbtp]
\begin{center}
\scalebox{0.65}{
  \begin{tabular}{|r|rr|rr|rr|rr|rr|rr|}
  \hline \multicolumn{1}{|c|}{$\lambda$} & \multicolumn{4}{c|}{$1/3$} &
\multicolumn{4}{c|}{$1/6$} &\multicolumn{4}{c|}{$1/9$}\\
\hline \multicolumn{1}{|c|}{} & \multicolumn{2}{c|}{1} &
\multicolumn{2}{c|}{2} &\multicolumn{2}{c|}{1} & \multicolumn{2}{c|}{2} & \multicolumn{2}{c|}{1}& \multicolumn{2}{c|}{2}\\
\cline{1-13}
  &  Bias & MSE& Bias & MSE & Bias & MSE & Bias & MSE& Bias & MSE& Bias & MSE\\
\hline
$\hat{\beta}_0$& $-0.00024$&$0.01261$  & $-0.00037$&$0.01271$ &$0.00712$ &$0.01139$&$0.00712$&$0.01135$&$0.00638$ &$0.01153$&$0.00646$&$0.0115$\\
$\hat{\beta}_1$&$-0.00702$&$0.00874$ & $-0.00721$&$0.00881$  & $0.00419$&$0.00697$&$0.00419$&$0.00696$&$-0.00089$ &$0.00685$&$-0.00098$&$0.00694$\\
$\hat{\alpha}$&$-0.00233$&$0.00058$& $-0.00352$&$0.00058$ & $-0.00115$ &$0.00059$&$-0.00143$&$0.00061$ &$-0.00425$&$0.00056$ &$-0.0048$&$0.00058$\\
$\hat{\sigma}^2$&$-0.00971$&$0.01341$&  $0.01768$&$0.01745$  & $-0.01056$&$0.01142$&$-0.00461$&$0.01694$&$-0.01049$&$0.01148$ &$-0.00015$&$0.01657$\\
 \hline
\end{tabular}
}
\end{center}
\caption{Bias and MSE when estimating with $wpl$  the  $t$ process
   when
the (inverse of) degrees of freedom ($\lambda=1/\nu$, $\nu=3,6,9$) are: 1) fixed and known,
2) unknown and fixed through a two-step estimation (Scenario 2).
}
\label{tab:tab411}
\end{table}

%%%%%%%%%%%%%%%%%%%%%%%%%%%%%%%%%%%%
%\begin{table}[!hbtp]
%\begin{center}
%\scalebox{0.65}{
 % \begin{tabular}{|r|rr|rr|rr|rr|rr|rr|}
 % \hline \multicolumn{1}{|c|}{$\lambda$} & \multicolumn{4}{c|}{$1/3$} &
%\multicolumn{4}{c|}{$1/6$} &\multicolumn{4}{c|}{$1/9$}\\
%\hline \multicolumn{1}{|c|}{} & \multicolumn{2}{c|}{1} &
%\multicolumn{2}{c|}{2} &\multicolumn{2}{c|}{1} & \multicolumn{2}{c|}{2} & \multicolumn{2}{c|}{1}& \multicolumn{2}{c|}{2}\\
%\cline{1-13}
 % &  Bias & MSE& Bias & MSE & Bias & MSE & Bias & MSE& Bias & MSE& Bias & MSE\\
%\hline
%$\hat{\beta}_0$&$0.00354$&$0.00324$&$0.00351$&$0.00324$&$-0.00034$&$0.00326$&$-0.00027$&$0.00326$&$0.00140$&$0.00290$&$0.00153$&$0.00291$\\
%$\hat{\beta}_1$&$-0.00347$&$0.00338$&$-0.00352$&$0.00337$&$-0.00232$&$0.00285$&$-0.00232$&$0.00286$&$0.00058$&$0.00249$&$0.00061$&$0.00248$\\
%$\hat{\alpha}_s$&$0.00176$&$0.00140$&$0.00112$&$0.00140$&$0.00034$&$0.00126$&$0.00006$&$0.00127$&$0.00088$&$0.00131$&$0.00036$&$0.00134$\\
%$\hat{\alpha}_t$&$-0.00214$&$0.00153$&$-0.00245$&$0.00153$&$-0.00509$&$0.00160$&$-0.00535$&$0.00160$&$-0.00367$&$0.00175$&$-0.00400$&$0.00177$\\
%$\hat{\sigma}^2$&$-0.00070$&$0.00457$&$0.00856$&$0.00579$&$0.00027$&$0.00373$&$0.00293$&$00.00596$&$0.00208$&$0.00371$&$0.00842$&$0.00588$\\
 %\hline
%\end{tabular}
%}
%\end{center}
%\caption{ Bias and MSE when estimating with $wpl$  the space-time $t$ process
 % when
%the (inverse of) degrees of freedom ($\lambda=1/\nu$, $\nu=3,6,9$) are: 1) fixed and known,
%2) unknown and fixed through a two-step estimation (Scenario 3):
%}
%\label{tab:tab419}
%\end{table}
%%%%%%%%%%%%%%%%%%%%%%%%%%%%%%%%%%%%

\begin{figure}[!hbtp]
\begin{tabular}{cccc}
\includegraphics[width=3.4cm, height=6.4cm]{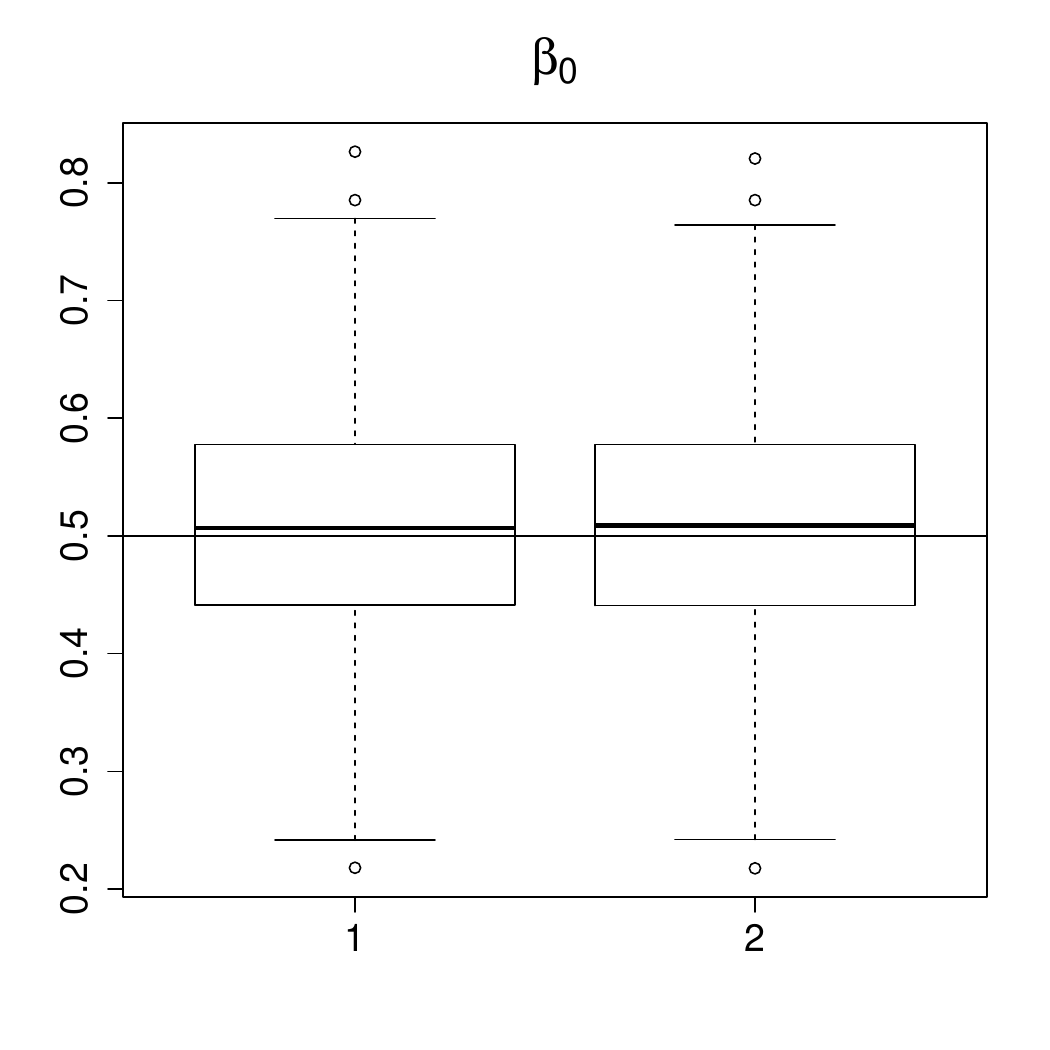} & \includegraphics[width=3.4cm, height=6.4cm]{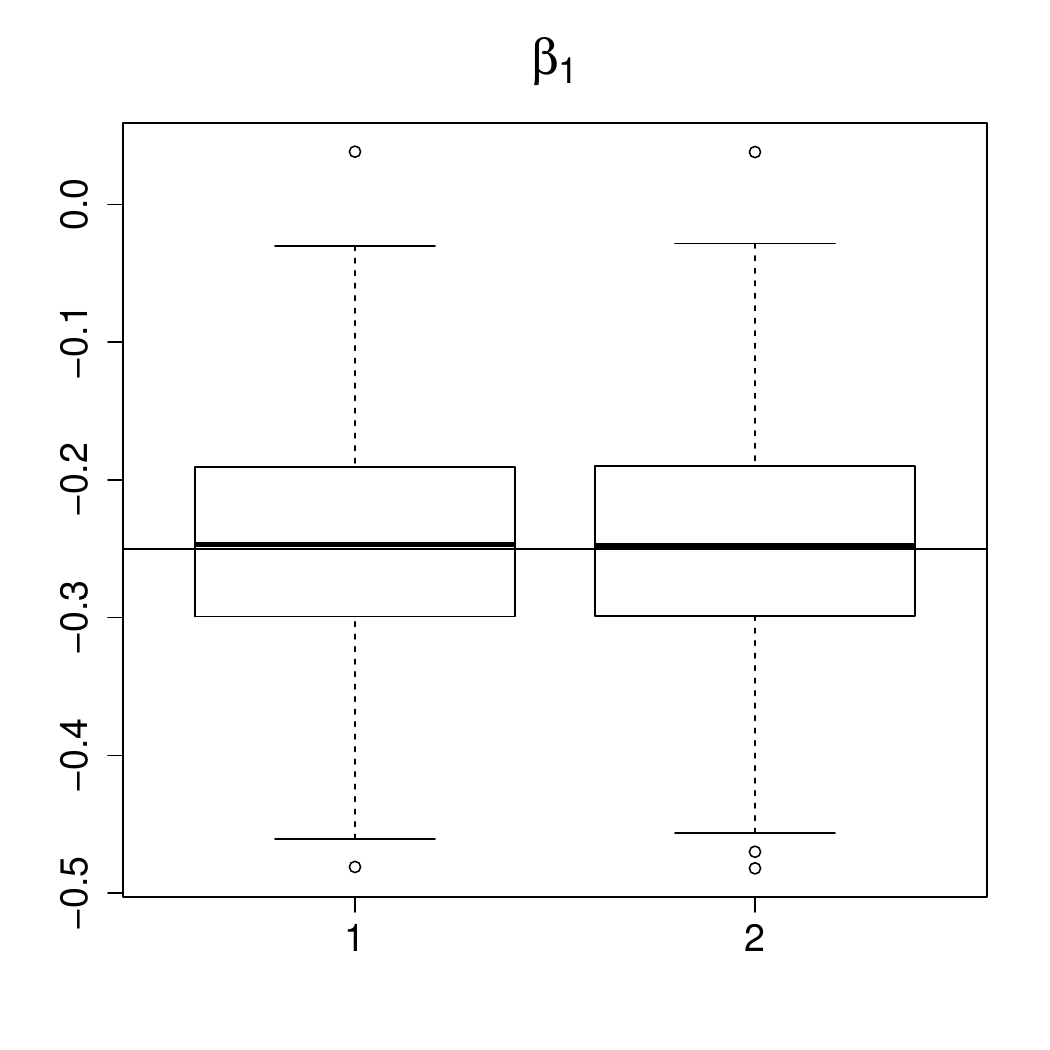} & \includegraphics[width=3.4cm, height=6.4cm]{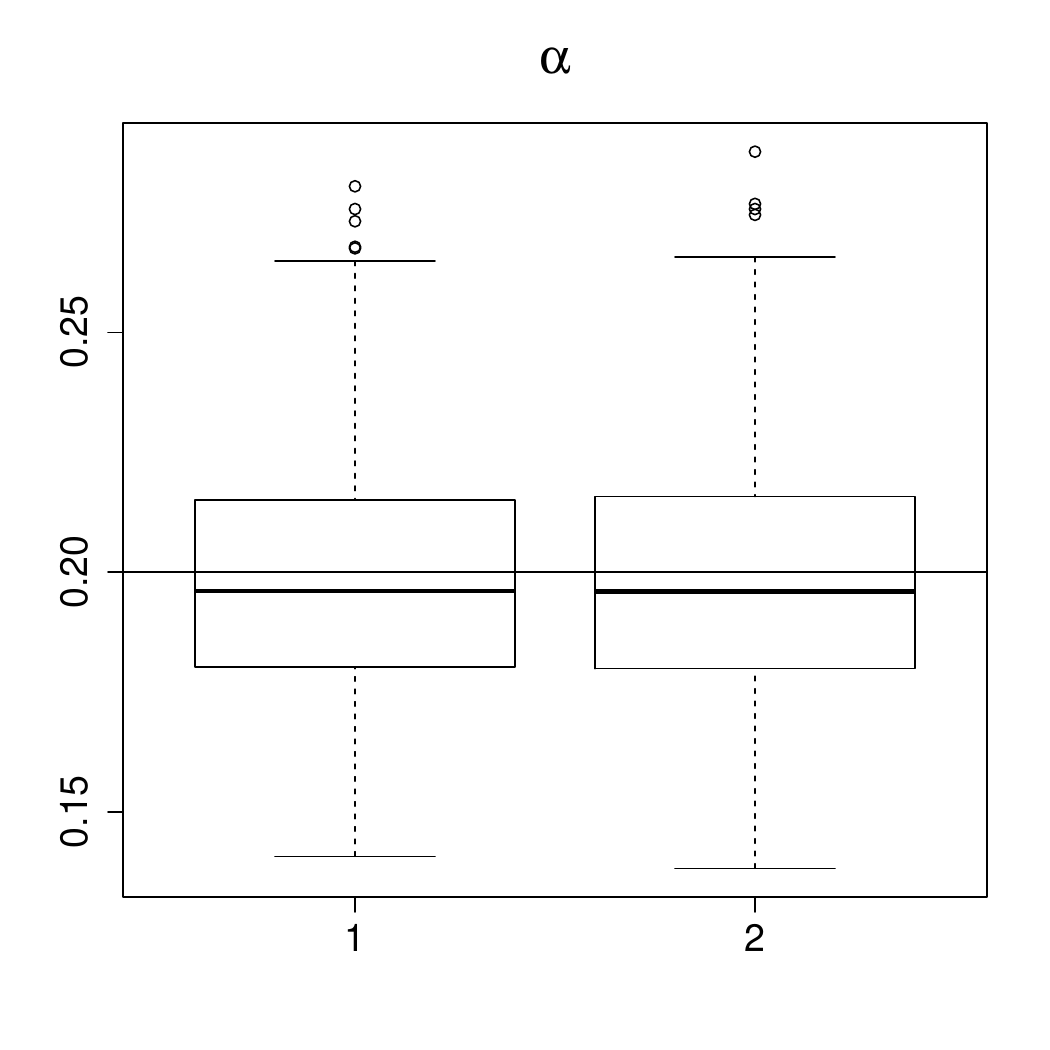}  & \includegraphics[width=3.4cm, height=6.4cm]{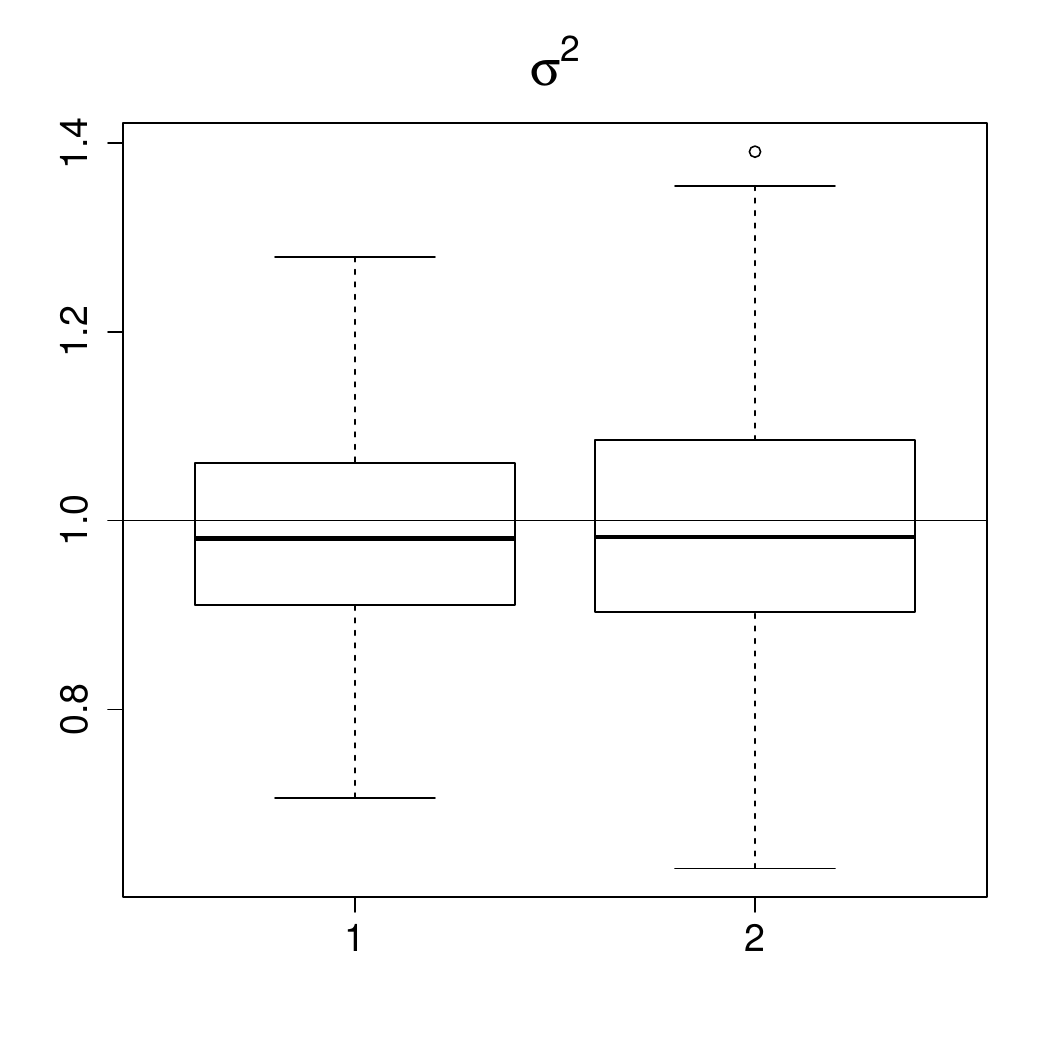}   \\
\end{tabular}
\caption{Boxplots of  $wpl$ estimates  for $\beta_0=0.5$, $\beta_1=-0.25$,
$\alpha=0.2$, $\sigma^2=1$ (from left to right) under Scenario 2
when estimating a $t$ process with $\lambda=1/\nu$, $\nu=6$  when
1) $\nu$ is assumed  known,
2)  $\nu$ is assumed  unknown and it is fixed  to a positive integer through a two-step estimation
.}\label{tpsta}
\end{figure}

%\begin{figure}[!hbtp]
%\begin{tabular}{ccccc}
%\includegraphics[width=2.6cm, height=6.4cm]{1aa.pdf} & \includegraphics[width=2.6cm, height=6.4cm]{1bb.pdf} & \includegraphics[width=2.6cm, height=6.4cm]{1cc.pdf}  & \includegraphics[width=2.6cm, %height=6.4cm]{1dd.pdf}  & \includegraphics[width=2.7cm, height=6.4cm]{1ee.pdf}  \\
%\end{tabular}
%\caption{Boxplots of  $wpl$ estimates  for $\beta_0=0.5$, $\beta_1=-0.25$,
%$\alpha_s=0.3$, $\alpha_t=0.5$ $\sigma^2=1$ (from left to right) under Scenario 3
%when estimating a space-time $t$ process with $\lambda=1/\nu$, $\nu=9$  when
%1) $\nu$ is assumed  known,
%2)  $\nu$ is assumed  unknown and it is fixed  to a positive integer through a two-step estimation.}\label{tpstast}
%\end{figure}

\subsection{Performance of the misspecified  (pairwise)  Gaussian  likelihood estimation}\label{pol}
Weighted pairwise likelihood estimation
requires the evaluation of  the bivariate distribution (\ref{pairt1}) $i.e.$ the computation of the  Appell  $F_4$ function.
Standard statistical software libraries for the computation of the $F_4$ function  are unavailable  to the best of our knowledge.
In our implementation we exploit the following relation with the Gaussian hypergeometric function \citep{Brychkov:Saad:2017}:
\begin{footnotesize}
\begin{equation}\label{apell4}
F_4(a,b;c,c';w,z)=\sum\limits_{k=0}^{\infty}\frac{(a)_{k}(b)_{k}z^k}{k!(c')_k}{}_2F_1(a+k,b+k;c;w),\;\;\;\;|\sqrt{w}|+|\sqrt{z}|<1,
\end{equation}
\end{footnotesize}
truncating the series when the $k$-th generic  element of the series is smaller than a fixed $\epsilon$
and where standard libraries for the computation of the ${}_2F_1$ function can be used \citep{hyp}.
Evaluation  of the $F_4$ function can be time consuming  depending of the speed of convergence
of (\ref{apell4}) and, as  a consequence, if the number of location sites is large the computation of the $wpl$ estimator
 can be computationally demanding.
%and in order to speed up the computation, the  implementation in the \texttt{GeoModels} package
%\citep{Bevilacqua:2018aa}
 %uses the OpenCL framework for parallel computing.

 An estimator   that require smaller computational burden  can be obtained
  by considering a misspecified $wpl$. Specifically,
 if  in the estimation procedure   we assume a Gaussian process with mean equal to $\mu(\bm{s})$,   variance equal to $\sigma^2\nu/(\nu-2)$ and correlation $\rho_{Y^*_{\nu}}(\bm{h})$,
 then a Gaussian $wpl$ only requires the computation of the Gaussian bivariate distribution
 and of the Gauss  hypergeometric function in (\ref{CC}).
 Note that the misspecified Gaussian process matches mean, variance and correlation function  of the $t$ process.
To avoid identifiability problems,
 we need a reparametrization of the variance $i.e$ $\sigma_*^2:=\sigma^2\nu/(\nu-2)$.
Then, maximization of the  Gaussian  $wpl$  function leads to the estimation of $\mu(\bm{s})$, $\sigma_*^2$,  $\nu$
and the parameters of the underlying correlation model $\rho(\bm{h})$.

To investigate the performance of this kind of estimator,
we consider $676$ points  on a  regular spatial grid   $ A=[0,1]^2$ that is
$(x_i,x_j)^T$ for $i,j=1,\ldots,21$
with
$x_1=0, x_2=0.04, \ldots x_{26}=1$
and
we simulate,  using Cholesky decomposition, $500$ realizations of a $t$ process
setting $\mu(\bm{s})=\mu=0$, $\sigma^2=1$,  $\nu=3, 6, 9$
and  underlying correlation function $\rho(\bm{h})={\cal
GW}_{\alpha,0,4}(\bm{h})$
with $\alpha=0.2$.
Then we estimate the parameters
$\mu$, $\sigma_*^2$, $\alpha$ (assuming $\nu$ known and fixed)
with  $wpl$ using  the  bivariate $t$ distribution (\ref{pairt1}) and with  both misspecified  Gaussian $wpl$  and standard likelihood.
 In the $wpl$ estimation we consider a
cut-off weight function with $d_{ij}=0.05$.

 Table \ref{tab:newtab}  shows  the bias and MSE
associated with
$\mu$, $\alpha$, and $\sigma^2$
for the three methods of estimation.
Note that, for comparison, the results of the variance parameter are reported in terms of the original parametrization. It is apparent that $wpl$ with bivariate $t$ distribution
show the best performance. In particular when $\lambda=1/3$ the gains in terms of efficiency are considerable. However,
 when increasing the degrees of freedom the gains tends to decrease and when $\nu=9$ the efficiencies of the three estimators are quite similar (see boxplots in
 Figure \ref{fig:FFr}).

\begin{table}[!hbtp]
\begin{center}
\scalebox{0.65}{
\begin{tabular}{cc|cc|cc|ll|}
\cline{3-8}
\multicolumn{1}{l}{}                                 & \multicolumn{1}{l|}{}             & \multicolumn{2}{c|}{$WPL_T$}         & \multicolumn{2}{c|}{$L_G$}           & \multicolumn{2}{c|}{$WPL_G$}                        \\ \hline
\multicolumn{1}{|c|}{$\lambda$}                      & \multicolumn{1}{l|}{$Parameters$} & \multicolumn{1}{c}{Bias} & MSE      & \multicolumn{1}{c}{Bias} & MSE      & \multicolumn{1}{c}{Bias} & \multicolumn{1}{c|}{MSE} \\ \hline
\multicolumn{1}{|c|}{\multirow{3}{*}{$\lambda=1/3$}} & $\mu$                             & $0.0045$& $0.0088$& $0.0036$& $0.0154$& $0.0036$& $0.0158$\\
\multicolumn{1}{|c|}{}                               & $\alpha$                          & $-0.0016$& $0.0003$ &$0.0084$& $0.0009$ &$0.0129$& $0.0013$ \\
\multicolumn{1}{|c|}{}                               & $\sigma^2$                        & $-0.0019$ &$0.0102$& $-0.0577$& $0.0477$ & $-0.0534$ &$0.0478$\\ \hline
\multicolumn{1}{|c|}{\multirow{3}{*}{$\lambda=1/6$}} & $\mu$                             & $-0.0057$& $0.0096$& $-0.0053$ &$0.0106$& $-0.0052$& $0.0110$\\
\multicolumn{1}{|c|}{}                               & $\alpha$                          & $-0.0020$ &$0.0003$ &$-0.0007$ &$0.0003$ &$-0.0014$ &$0.0003$\\
\multicolumn{1}{|c|}{}                               & $\sigma^2$                        & $-0.0105$& $0.0075$& $-0.0150$& $0.0095$& $-0.0168$& $0.0098$\\ \hline
\multicolumn{1}{|c|}{\multirow{3}{*}{$\lambda=1/9$}} & $\mu$                             & $-0.0029$& $0.0091$& $-0.0038$& $0.0094$& $-0.0040$& $0.0096$\\
\multicolumn{1}{|c|}{}                               & $\alpha$                          & $-0.0019$& $0.0003$& $-0.0019$& $0.0003$& $-0.0022$ &$0.0003$ \\
\multicolumn{1}{|c|}{}                               & $\sigma^2$                        & $-0.0075$& $0.0081$& $-0.0089$& $0.0088$ &$-0.0093$& $0.0088$\\ \hline
\end{tabular}
}
\end{center}
\caption{Bias and MSE associated with  $\mu$, $\alpha$ and $\sigma^2$ for  $wpl$ with bivariate $t$ distribution ($WPL_T$),
standard misspecified  Gaussian likelihood
($L_G$)
and
$wpl$ with bivariate misspecified  Gaussian distribution ($WPL_G$) when   $\lambda=1/\nu$, $\nu=3,6,9$.}
\label{tab:newtab}
\end{table}

\begin{figure}[h!]
\begin{tabular}{ccc}
\includegraphics[width=4.8cm, height=6.4cm]{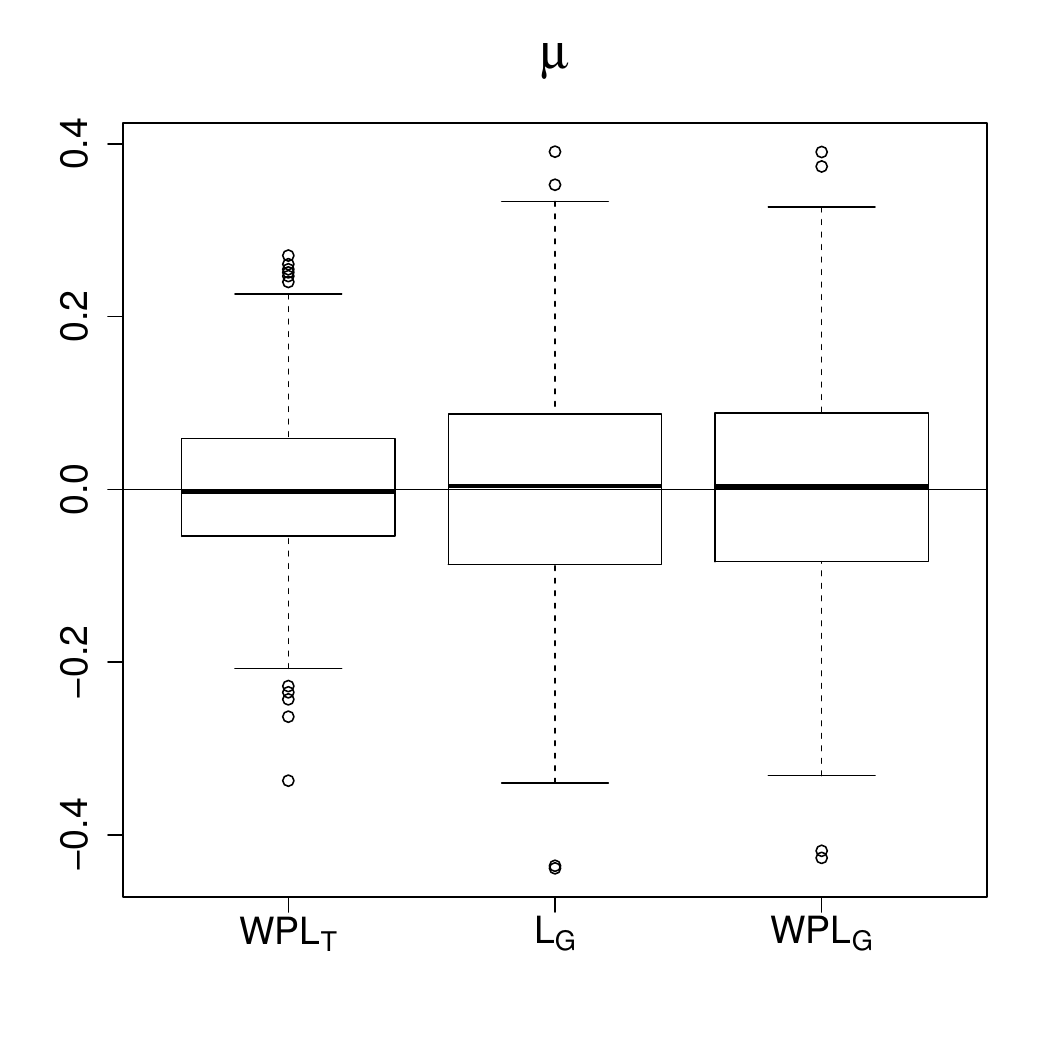} & \includegraphics[width=4.8cm, height=6.4cm]{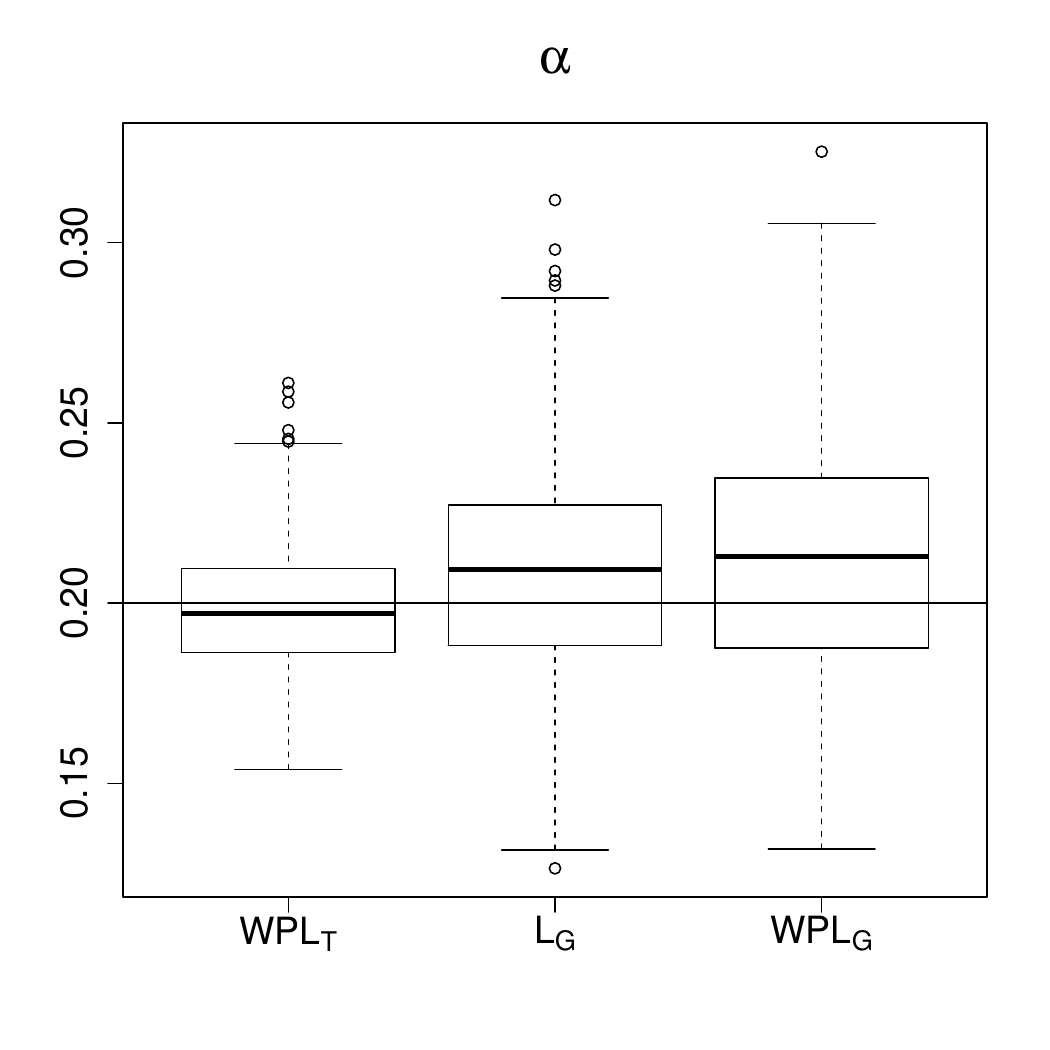} & \includegraphics[width=4.8cm, height=6.4cm]{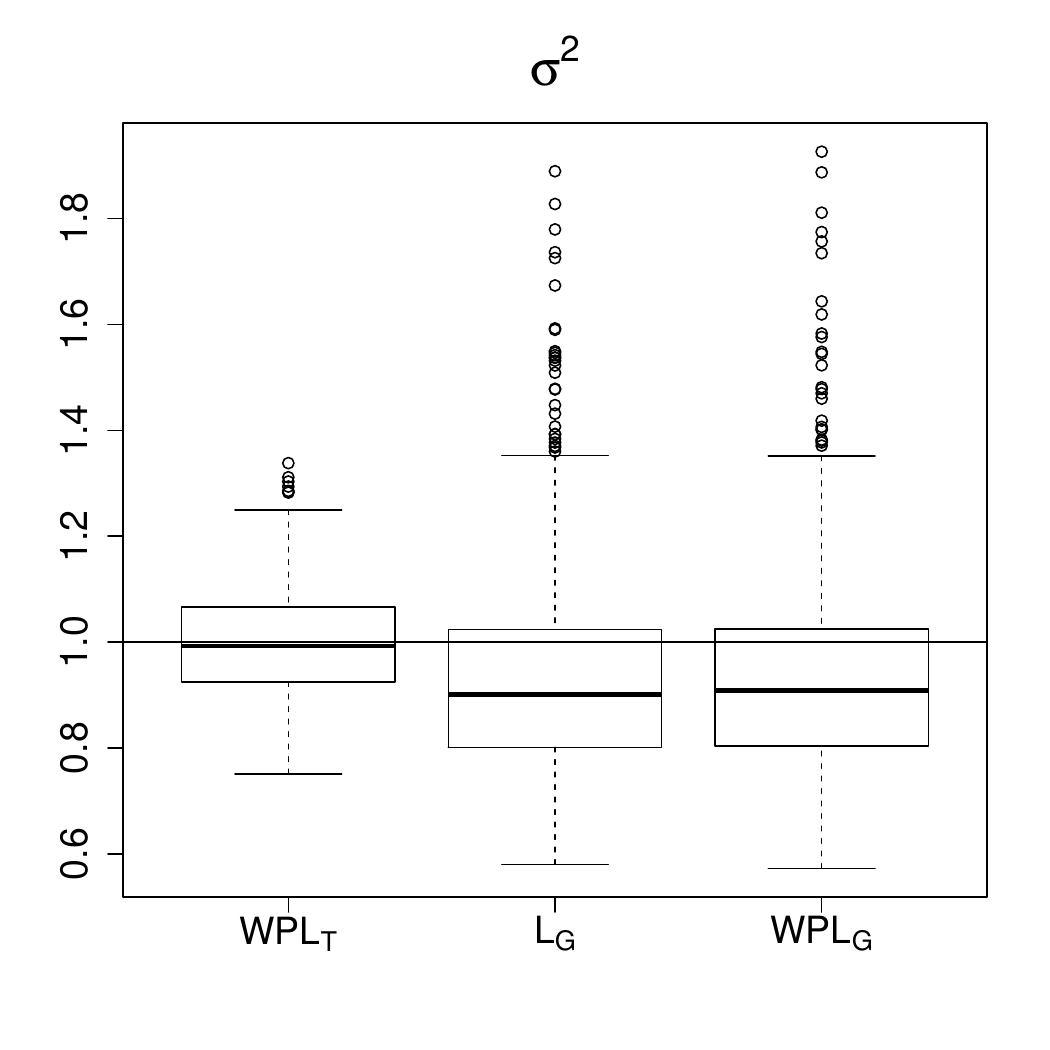}  \\
\includegraphics[width=4.8cm, height=6.4cm]{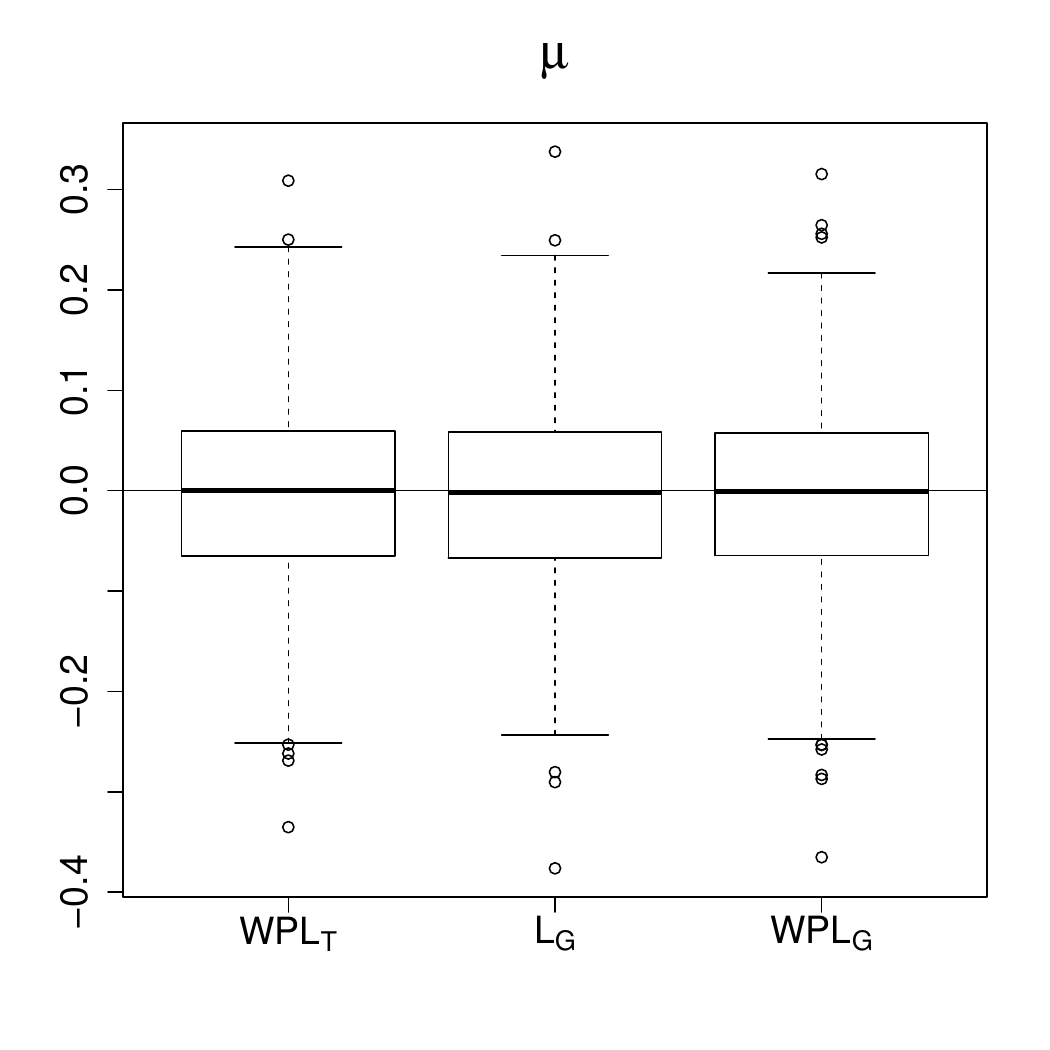} & \includegraphics[width=4.8cm, height=6.4cm]{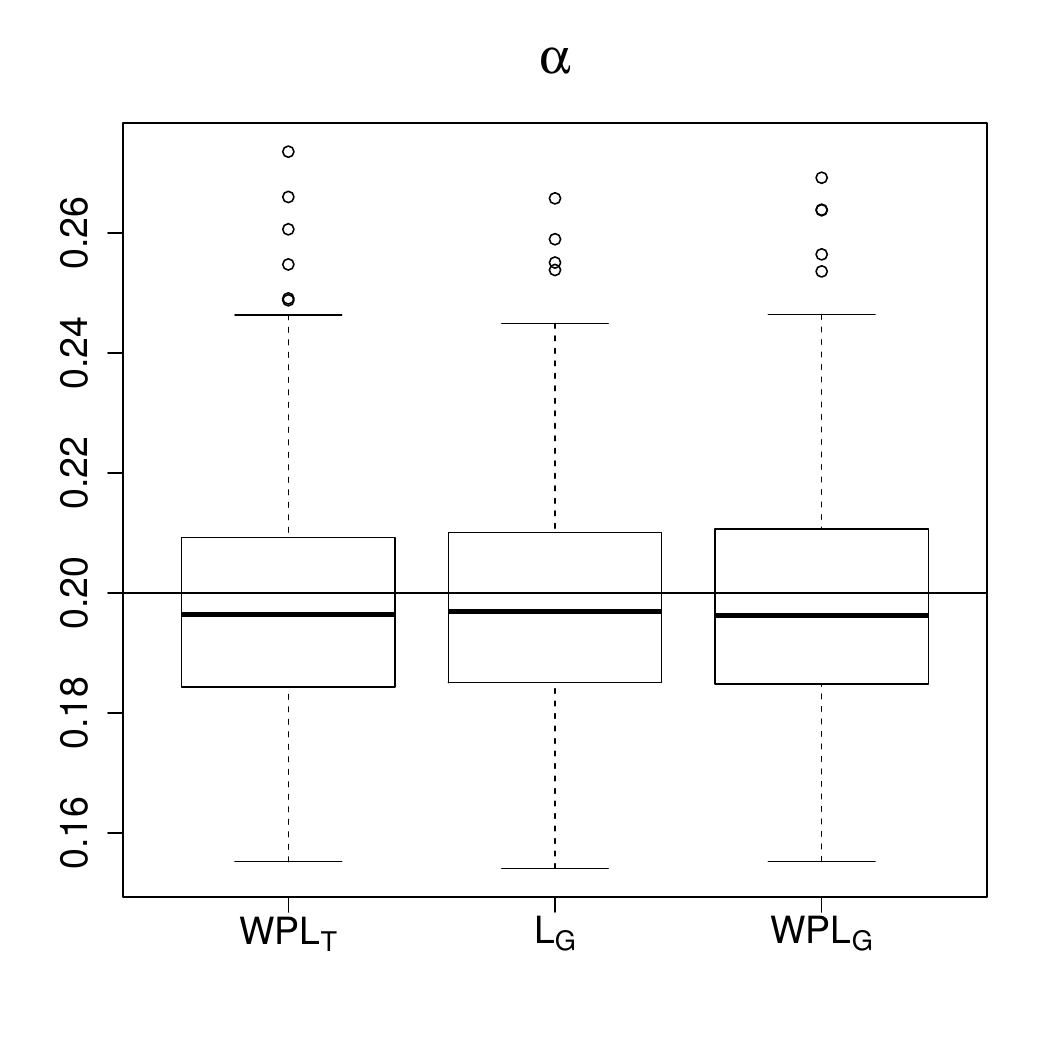}& \includegraphics[width=4.8cm, height=6.4cm]{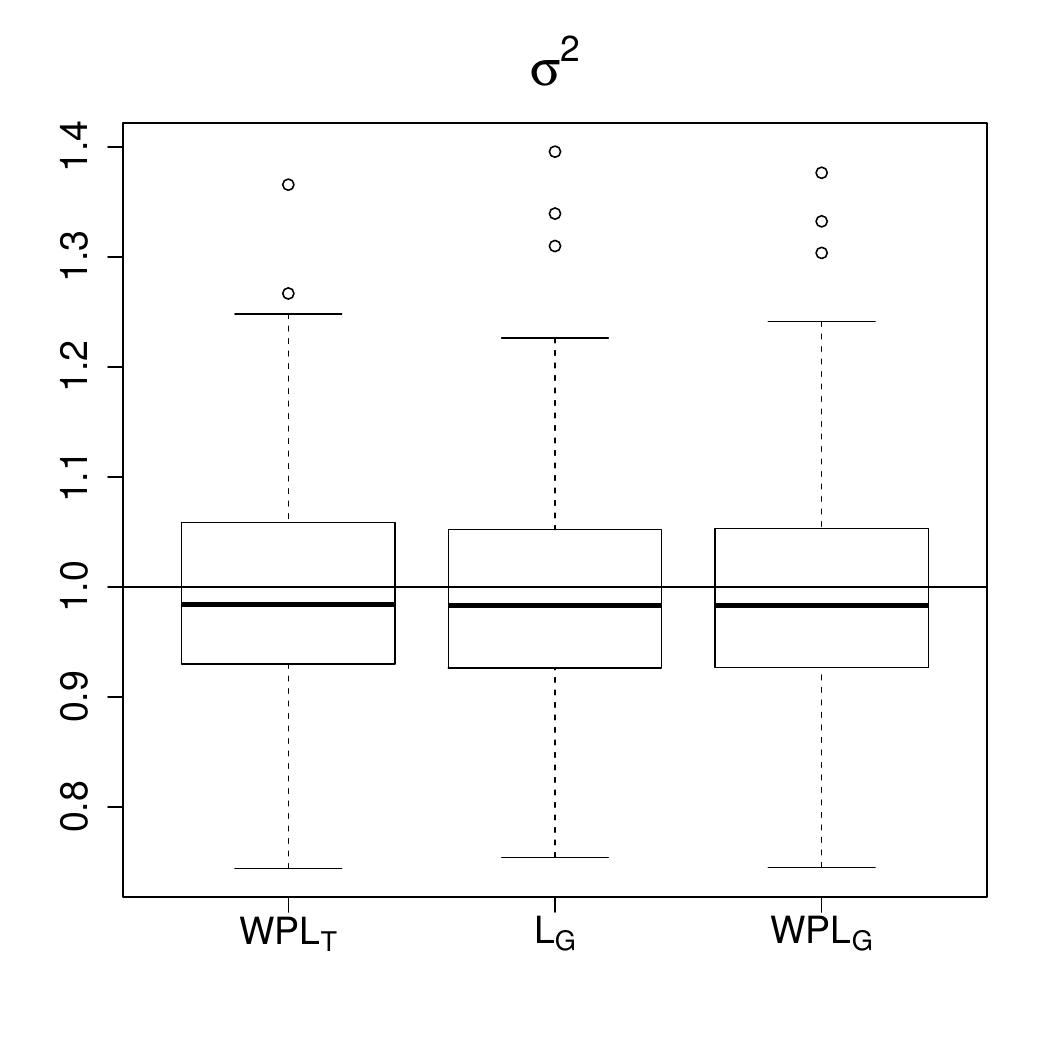}   \\
\end{tabular}
\caption{Upper part: boxplots of   $\mu$, $\alpha$, $\sigma^2$ for $wpl$ using $t$ bivariate distribution ($WPL_T$), standard misspecified Gaussian  likelihood ($L_G$)
and $wpl$ using Gaussian  misspecified bivariate  distribution ($WPL_G$) when $\nu=3$.  Bottom part: The same boxplots with $\nu=9$.  } \label{fig:FFr}
\end{figure}

\subsection{$t$ optimal linear prediction versus Gaussian optimal prediction}\label{piop}

One of the primary goals of geostatistical modeling is to make predictions at spatial
locations without observations.
The optimal predictor for  the $t$ process, with respect to the
mean squared error criterion, is nonlinear and difficult to
evaluate explicitly since it requires the knowledge of the finite
dimensional distribution.

% In principle, the optimal predictor can be  computed as  the the posterior mean  of the  predictive distribution that  can be approximated using Monte Carlo samples
 %as in  \cite{Palacios:Steel:2006} and  \cite{ZAREIFARD201316}.
% However, from a computational point of view,
%Monte Carlo samples are difficult to produce efficiently and such
%a method can be unfeasible for large data sets.
%\citep{Zhang:El-Shaarawi:2010}.
A %more 
practical and less efficient 
solution can be
obtained using the optimal linear prediction. Assuming known mean, correlation and the degrees of freedom  of the $t$ process,
the predictor  at an unknown location $\bm{s}_0$ is given by:

\begin{equation}\label{pyt}
\widehat{y(\bm{s}_0)}=\mu(\bm{s}_0)+ \bm{c}^T_{\nu}R_{\nu}^{-1}(\bm{Y}-\bm{\mu})
\end{equation}
where $\bm{\mu}=(\mu(\bm{s}_1),\ldots,\mu(\bm{s}_n))^T$,  $\bm{c}_{\nu}=[\rho_{Y^*_{\nu}}(\bm{s}_0-\bm{s}_j)]_{i=1}^n$ and
$R_{\nu}=[\rho_{Y^*_{\nu}}(\bm{s}_i-\bm{s}_j)]_{i,j=1}^n$, and the  associated variance is given by:
\begin{equation}
Var(\widehat{y(\bm{s}_0)})=\sigma^2_{*} (1-\bm{c}^T_{\nu}R_{\nu}^{-1}\bm{c}_{\nu}).
\end{equation}

 As an Associate Editor pointed out,
 this is equivalent to perform optimal Gaussian  prediction with  covariance  function equal to $\sigma^2_{*} \rho_{Y^*_{\nu}}(\bm{h})$.
%From this perspective $\nu>2$ in not restricted to be an integer value.
Similarly, using (\ref{CC99}), the optimal linear predictor of the skew-t process can be obtained.

%=\frac{\sigma^2 \nu}{\nu-2}(1-\bm{c}^T_{\nu}R_{\nu}^{-1}\bm{c}_{\nu}r)
We investigate   the performance of $(\ref{pyt})$ when compared
with the Gaussian optimal predictor, assuming  $\rho(\bm{h})={\cal  GW}_{0.2,\psi,4}(\bm{h})$, $\psi=0, 1,2$  as underlying  correlation function, using cross-validation.
With this goal in mind, we simulate   $1000$ realizations  from  a  $t$ process  $Y^*_{\nu}$ with $\nu=3, 7, 11$ and a Gaussian process
under the  settings of Section 4.3
% simulate  $500$ realizations % (i.e  $2601$   location sites)
% of a standardized $t$ process with   $\nu=3,6,9,12,15$ and underlying correlation $\rho(\bh)={\cal
%GW}_{0.2,0,4}(\bm{||h||})$
%on the square $[0,0.02,\ldots,1]^2$ (i.e  $N=2601$   location sites).
%We simulate for each process $500$ realizations
and
for each realization, we consider  $80\%$ of the data for
prediction and leave $20\%$ as validation dataset.
%As methods of estimation and prediction  we consider a) $wpl$ using the $t$ bivariate distribution
%and prediction using (\ref{pyt});
%b) standard likelihood  assuming Gaussianity  with correlation $\rho_{Y^*_{\nu}}(\bm{h})$ and prediction using (\ref{pyt});
%c) standard likelihood assuming Gaussianity  with correlation $\rho(\bm{h})$ and prediction using (\ref{pyt}) assuming $\nu \to \infty$
 %(using Theorem \ref{theoiii} point (e));

%The use of a Generalized Wendland  model induces sparsity in the covariance matrix of the $t$ process and this allows
%speeding up the computation of the  the optimal linear
%predictor  using specific  sparse matrix algorithms \citep{Furrer:2006,Bevilacqua_et_al:2018}.% as in the Gaussian case.

For each model and for each realization, we
compute the root mean square errors (RMSEs)
%and mean absolute
%errors (MAEs)
 that  is:
\begin{footnotesize}
\begin{equation*}
RMSE_l=\left(\frac{1}{n_l}\sum_{i=1}^{n_l} \left(\widehat{y(\bm{s}_{i,l})}-y(\bm{s}_{i,l}))\right)^2\right)^{\frac{1}{2}},
%\;\; MAE_l= \frac{1}{n_l}\sum_{i=1}^{n_l} |\widehat{y(\bm{s}_{i,l})}-y(\bm{s}_{i,l})|,\;\; l=1,\ldots,500
\end{equation*}
\end{footnotesize}
where $y(\bm{s}_{i,l})$, $i=1,\ldots,n_l$ are the observation in the $l$-th validation set  and $n_l$ is the associated cardinality ($n_l=135$ in our example).
%Similarly we compute  $RMSE_l$ and $MAE_l$, $l=1,\ldots,500$ in the Gaussian case.

Finally we compute   the empirical mean of the $500$
RMSEs %and MAEs
when the prediction is performed with the optimal linear predictor    (\ref{pyt})  $i.e.$ using $\rho_{Y^*_{\nu}}(\bm{h})$ with  $\nu=3, 7, 11$
and the optimal Gaussian predictor  with $\rho(\bm{h})$. Note that, from Theorem 2.2 (e), the prediction using $\rho(\bm{h})$  can be viewed as the
prediction using $\rho_{Y^*_{\nu}}(\bm{h})$ when  $\nu \to \infty$.

In Table \ref{tab:tabpred} we report the simulation results in terms of  relative  efficiency
 that is, for a given process and a given $\psi=0, 1, 2$,
  the ratio between
the mean RMSE  of the best predictor and the mean   RMSE  associated with a competitive predictor. This implies that relative efficiency prediction  is lower than 1 and it is equal to $1$ in the best case.
From  Table \ref{tab:tabpred}, it can be appreciated
that under the $t$ process $Y^*_{\nu}$,
 the prediction with $\rho_{Y^*_{\nu}}(\bm{h})$, for $\nu=3,7,11$,
performs  overall better than the optimal Gaussian prediction using $\rho(\bm{h})$.
 As expected, the gain is more apparent when decreasing the degrees of freedom and increasing $\psi$.
 %we loose about 13% of the efficiency in the worst case which is an encouraging result.
 For instance,  if $\nu=3$ and $\psi=2$ the loss of   efficiency %when
 predicting with  $\rho(\bm{h})$ %with respect to the  prediction with $\rho_{Y^*_{3}}(\bm{h})$
 %is $0.81$ approximatively.
 is $19\%$ approximatively.
 It can also be noted that if $Y^*_{7}$, $Y^*_{11}$  or Gaussian are one or two mean squared differentiable ($\psi=1,2$), then   the  prediction using  $\rho_{Y^*_{3}}(\bm{h})$ can be very  inefficient.
 This is not surprising since from Theorem 2.2 (c), $Y^*_{3}$ is not mean square differentiable.
 Resuming, this numerical experiment study suggests  that
when predicting  data exhibiting   heavy tails
the use of the correlation function  $\rho_{Y^*_{\nu}}(\bm{h})$ should be preferred
to the use of % the correlation
$\rho(\bm{h})$.

\begin{table}[!hbtp]
\begin{center}
\scalebox{0.75}{
\begin{tabular}{cc|c|c|c|c|}
 \hline
 \multicolumn{1}{|c|}{${\cal  GW}_{0.2,\psi,4}(\bm{h})$}                                     &  Simulation from       &$Y^*_{3}$    & $Y^*_{7}$   & $Y^*_{11}$& Gaussian \\ \hline
 \multicolumn{1}{|c|}{}         & Prediction with            &   &    &  &  \\ \hline
\multicolumn{1}{|c|}{\multirow{4}{*}{$\psi=0$}} & $\rho_{Y^*_{3}}(\bm{h})$    & $1$ & $0.9896$ & $0.9870$ & $  0.9838 $ \\ \cline{2-6}
\multicolumn{1}{|c|}{}                                   & $\rho_{Y^*_{7}}(\bm{h})$    & $0.9904$ & $1$ & $0.9998$ & $0.9987$ \\ \cline{2-6}
\multicolumn{1}{|c|}{}                                   & $\rho_{Y^*_{11}}(\bm{h})$ & $0.9877$ & $0.9998$ & 1 & $0.9996$ \\ \cline{2-6}
\multicolumn{1}{|c|}{}                                   & $\rho(\bm{h})$& $0.9836$ & $0.9988$ & $0.9996$ & $1$ \\ \hline
\multicolumn{1}{|c|}{\multirow{4}{*}{$\psi=1$}} & $\rho_{Y^*_{3}(\bm{h})}$    & $1$ & $0.9399$ & $0.9180$ & $0.8789$ \\ \cline{2-6}
\multicolumn{1}{|c|}{}                                   & $\rho_{Y^*_{7}}(\bm{h})$    & $0.9503$ & $1$ & $0.9984$ & $0.9895$ \\ \cline{2-6}
\multicolumn{1}{|c|}{}                                   & $\rho_{Y^*_{11}}(\bm{h})$ & $0.9327$ & $0.9984$ & $1$ & $0.9966$ \\ \cline{2-6}
\multicolumn{1}{|c|}{}                                   & $\rho(\bm{h}) $& $0.9095$ & $0.9947$ & $0.9969$ & $1$ \\ \hline
\multicolumn{1}{|c|}{\multirow{4}{*}{$\psi=2$}} & $\rho_{Y^*_{3}}(\bm{h})$    & $1$ & $0.8827$ & $0.8263$ & $0.7109$ \\ \cline{2-6}
\multicolumn{1}{|c|}{}                                   & $\rho_{Y^*_{7}}(\bm{h})$    & $0.9223$ & $1$ & $0.9936$ & $0.9511$ \\ \cline{2-6}
\multicolumn{1}{|c|}{}                                   & $\rho_{Y^*_{11}}(\bm{h})$ & $0.8834$ & $0.9947$ & $1$ & $0.9817$ \\ \cline{2-6}
\multicolumn{1}{|c|}{}                                   & $\rho(\bm{h})$ & $0.8169$ & $0.9643$ & $0.9848$ & $1$ \\ \hline
\end{tabular}
}
\end{center}

\caption{Relative mean RMSEs  prediction efficiency %(RMSE) %and mean absolute error(MAE)
over 500 runs
for a $t$ process  with  $\nu=3, 7, 11$ and a Gaussian process
when predicting using $\rho_{Y^*_{\nu}}(\bm{h})$ for  $\nu=3, 7, 11$  and  $\rho(\bm{h})$.
The underlying correlation  model  is  $\rho(\bm{h})={\cal  GW}_{0.2,\psi,4}(\bm{h})$
with $\psi=0,1,2.$
} \label{tab:tabpred}
\end{table}

\section{Application to Maximum Temperature Data}\label{sec:5}

In this section, we apply the proposed $t$ process to  a data set of maximum temperature data observed in Australia.
Specifically, we consider a subset of a global  data set  of merged maximum daily temperature measurements from the Global Surface Summary of Day data (GSOD)  with European Climate Assessment  \&Dataset (ECA\&D) data in July 2011.
 The dataset  is described in detail in \cite{qqq} and it is available in  the R package \texttt{meteo}. The subset we consider is depicted in Figure
 \ref{australia} (a) and
    consists of the  maximum temperature
observed on July 5 in  $446$ location sites, $y(\bm{s}_i)$, $i=1,\ldots,446$,
in the region
with longitude
$[110,154]$  and latitude $[-39,-12]$.

\begin{figure}[h!]
\begin{tabular}{cc}
  \includegraphics[width=7.3cm, height=6.8cm]{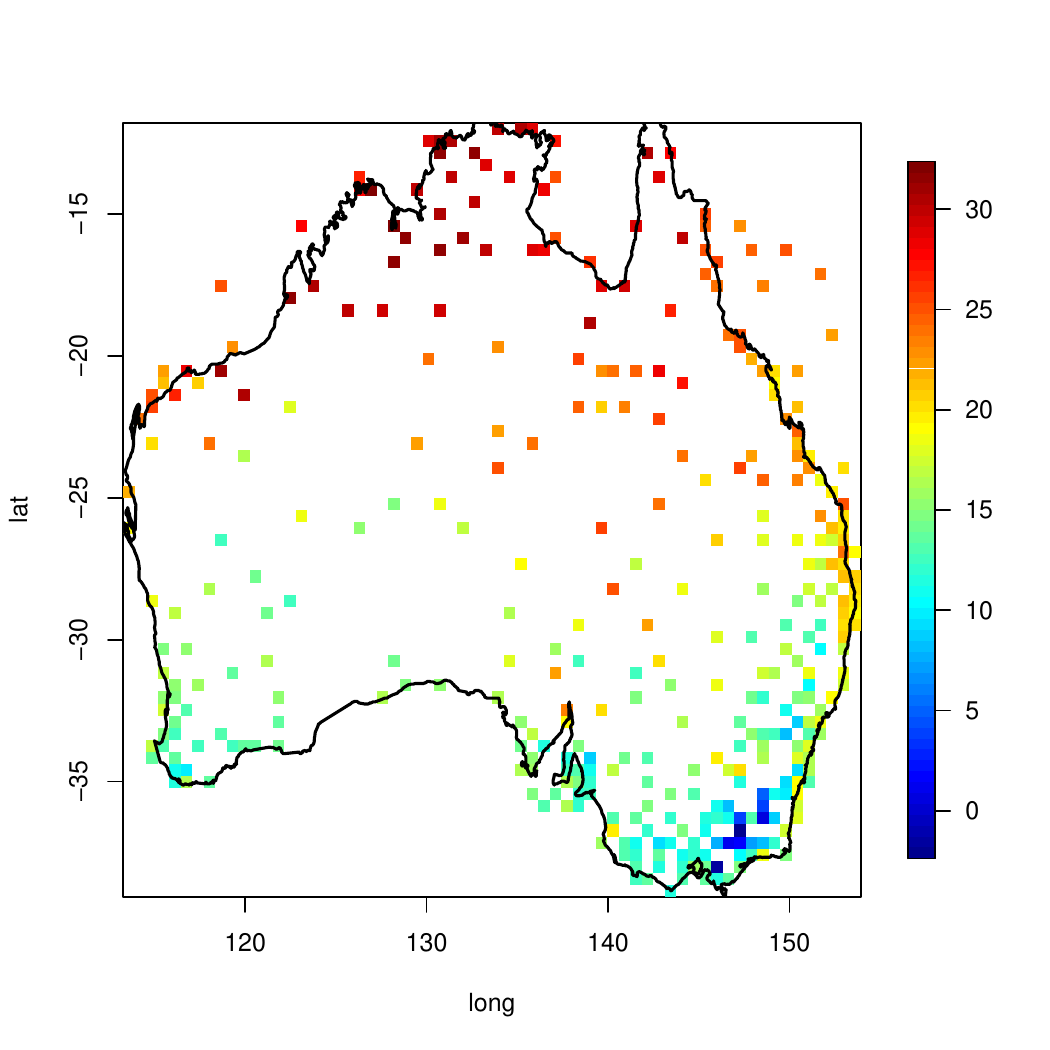}& \includegraphics[width=7.3cm, height=6.8cm]{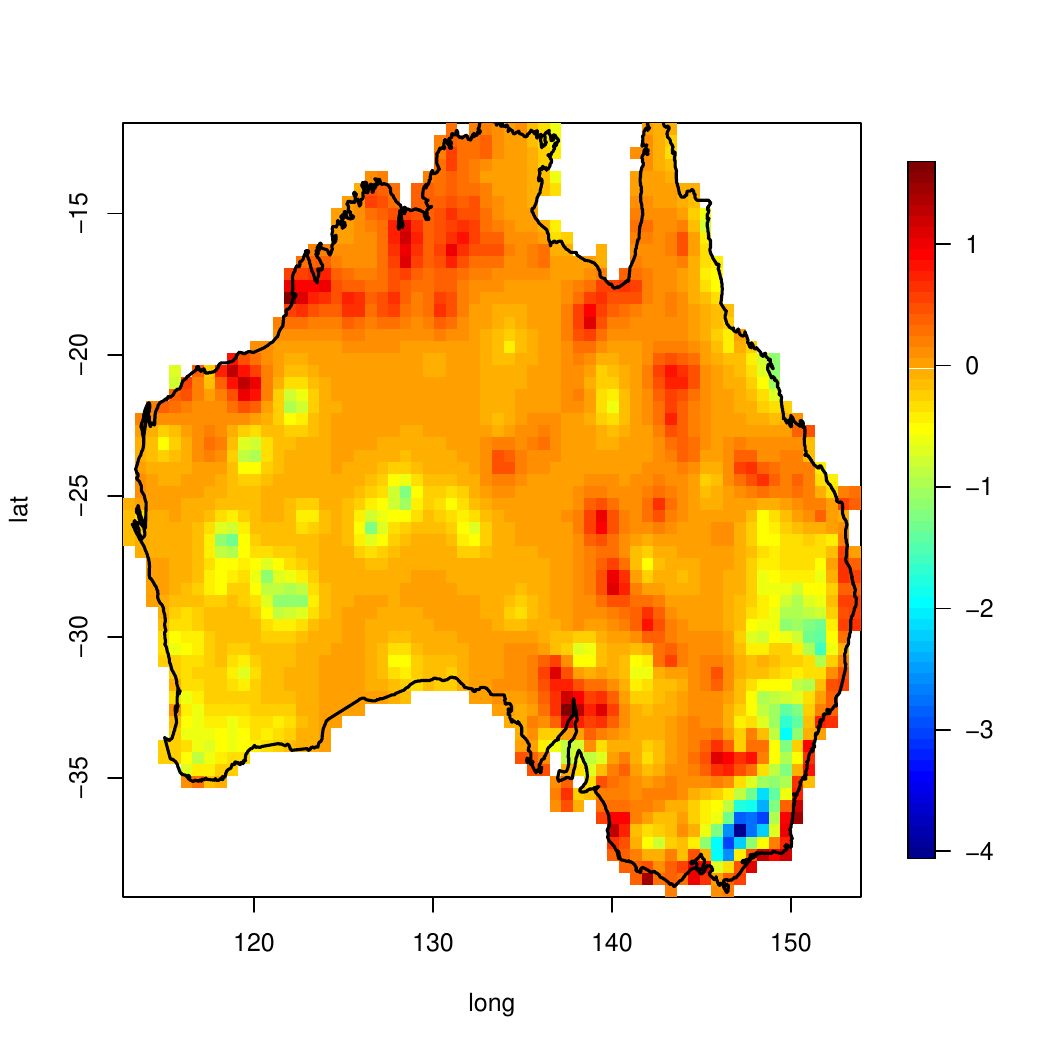}\\
  (a)&(b)
  \end{tabular}
  \caption{From left to right: a) spatial locations of   maximum temperature in Australia  in July 2011
  and b) prediction of residuals of the estimated $t$ process.}\label{australia}
\end{figure}

Spatial coordinates are given in  longitude and  latitude    expressed as decimal degrees and we consider the proposed $t$ process defined on the  planet Earth sphere approximation
 $\mathbb{S}^2=\{    \bm{s}\in \R^{3}: ||\bm{s}||=6371 \}$.
The first process  we use to model this dataset is a $t$ process:
\begin{equation}\label{kkkkk28}
Y_{\nu}(\bm{s})=\beta_0+\beta_1X(\bm{s})+\sigma Y^*_{\nu}(\bm{s}), \;\;\;  \bm{s} \in \mathbb{S}^2
\end{equation}
where $Y^*_{\nu}$ is a standard  $t$ process. %, and we estimate it  using $wpl$ estimation as depicted in Section $4.1$.
Here,
$X(\bm{s})$
is a covariate  called \emph{geometric temperature}  which represents
the geometric position of a particular location on Earth and the  day of the year
 \citep{qqq}.
As a  comparison, we also consider a  Gaussian process:
\begin{equation}\label{kkkkk29}
Y(\bm{s})=\beta_0+\beta_1X(\bm{s})+\sigma G(\bm{s}),  \quad  \bm{s} \in \mathbb{S}^2
\end{equation}
where $G$ is a standard Gaussian process.
 We assume that the underlying geodesically isotropic  correlation function \citep{gneiting2013,porcubev}
is of the Mat{\'e}rn and  generalized Wendland type.
A preliminary estimation of the $t$  and Gaussian processes,
 including the smoothness parameters, highlights a
multimodality of the (pairwise) likelihood surface, for both correlation models
and a not mean-square differentiability of the  process.
 For this reason, we fix the smoothness parameters and we consider the underlying correlation models
${\cal  M}_{\alpha,0.5}(d_{GC})=e^{-d_{GC}/\alpha} $ and ${\cal  GW}_{\alpha,0,5}(d_{GC})=(1-d_{GC}/\alpha)_+^5$
where, given  two spherical points
$\bm{s}_i=(\mbox{lon}_i,\mbox{lat}_i)$    and $\bm{s}_j=(\mbox{lon}_j,\mbox{lat}_j)$,
$d_{GC}(\bm{s}_i,\bm{s}_j) = 6371\theta_{ij},$ is the great circle distance.
Here $\theta_{ij}=\arccos\{ \sin a_i \sin a_j  + \cos a_i  \cos a_j  \cos(b_i - b_j)\}$
is the  great circle distance on the unit sphere with  $a_i=(\mbox{lat}_i)\pi/180$, $a_j=(\mbox{lat}_j)\pi/180$, $b_i=(\mbox{lon}_i)\pi/180$, $b_j=(\mbox{lon}_j)\pi/180$.

%\textcolor{red}{Under these specific choices for the correlation model, Theorem 2.2 (point c),
%guarantees that both Gaussian and $t$ processes are not mean-square differentiable.}

%\begin{equation}\label{pp4}
%	\rho_{\alpha,4}(d_{GC})= {\cal  GW}_{\alpha,0,\delta}(d_{GC})=
%	\begin{cases}\left ( 1- d_{GC} / \alpha \right )^{4}
%	& d_{GC}< \alpha \\
%	0 & \text{otherwise} \end{cases}.
%	\end{equation}
%	and $X$ is the geometric temperature.

%Preliminary estimations  with pairwise likelihood and maximum likelihood, including the smoothness parameter $\psi$, highlight a
%multimodality of the likelihood surface and an estimates of the smoothness parameter lower than $0.5$.
%For this reason we fix the smoothness parameter  $\psi=0.2, 0.4, 0.5$.
%Note that, this choice for the smoothness parameter is consistent with the validity of the Mat{\'e}rn  coupled with the great-circle distance  \citep{gneiting2013}.

For the $t$ process the parameters
were estimated using $wpl$ using the bivariate $t$ distribution with the two-step method described in Section \ref{sec:4} and using the  weight function
(\ref{Wei}) with $d_{ij}=150$ Km.
It turns out that the estimation at the  first step  leads to fix $\nu=4$ in the second step, irrespective of the correlation model.
We also consider a Gaussian misspecified standard  likelihood  and $wpl$  estimation as described in Section \ref{pol}
i.e. we estimate using the Gaussian process (\ref{kkkkk29}) with the  $t$ correlation model \ref{CC}  fixing
 $\nu=4$.

In addition,  we compute the standard error estimation, $\mbox{PLIC}$ and $\mbox{BLIC}$ values
through parametric bootstrap estimation of the inverse of the Godambe information matrix \citep{Bai:Kang:Song:2014}.
For  standard maximum  likelihood
we compute the standard errors as the square root of diagonal elements of the inverse of Fisher Information matrix \citep{Mardia:Marshall:1984}.
The results are summarized in Table  \ref{tab:est}. Note that the  regression parameters estimates are quite similar
for the $t$ and Gaussian  processes, irrespective of the  correlation model.
Furthermore,   we note that the standard Gaussian  process assigns lower spatial dependence  and stronger variance
compared to the other cases.
Finally,  for each correlation model,  both  the (pairwise) likelihood information criterion $\mbox{PLIC}$ and $\mbox{BLIC}$
select for the pairwise case the $t$ model  and for the standard case  the Gaussian model with $t$ correlation function.
\begin{table}[htbp]
\begin{center}
\scalebox{0.65}{
\begin{tabular}{|c|c|c|c|c|c|c|c|c|c|c|}
  \hline

\multicolumn{11}{|c|}{Mat\'ern}  \\
  \hline
&Method& $\hat{\beta}_0$ & $\hat{\beta}_1$&$\hat{\alpha}$&$\hat{\sigma}^2$&PLIC&BLIC& RMSE&MAE&CRPS\\
\hline
\multirow{2}{*}{$t$}&\multirow{2}{*}{Pairwise}&$6.652$&$0.994 $&$58.582$&$7.589 $&$23906$&$24689$&$2.775$&$2.138$&$1.807$\\
&&$( 1.192 )$&$(0.107)$&$(13.492)$&$(1.451)$&&&&&\\
\hline
\multirow{2}{*}{Missp-Gaussian}&\multirow{2}{*}{Pairwise}&$5.656$&$1.064$&$72.435$&$7.062$&$24680$&$26172$&$2.743$&$2.112$&$1.809$\\
&&$( 1.514)$&$(0.130)$&$(18.303)$&$( 2.024 )$&&&&&\\
\hline
\multirow{2}{*}{Missp-Gaussian}&\multirow{2}{*}{Standard}&$ 5.673 $&$ 1.050$&$ 63.437 $&$ 5.385 $& $2228$ &$2244 $&$2.755$&$ 2.119$& $1.813$\\
&&$( 0.660)$&$(0.045)$&$(10.41)$&$(0.432)$&&&&&\\
\hline
\multirow{2}{*}{Gaussian}&\multirow{2}{*}{Standard}&$  5.508$&$1.060 $&$ 40.484 $&$10.762    $& $2240$&$2257$&$2.812$&$ 2.155$& $1.815$\\
&&$(0.558)$&$(0.039)$&$(5.346)$&$(0.813)$&&&&&\\
\hline
%&&$$&$$&$$&$$&&&&&\\
\multicolumn{11}{|c|}{Wendland}  \\
  \hline
&Method& $\hat{\beta}_0$ & $\hat{\beta}_1$&$\hat{\alpha}$&$\hat{\sigma}^2$&PLIC&BLIC& RMSE&MAE&CRPS\\
\hline
\multirow{2}{*}{$t$}&\multirow{2}{*}{Pairwise}&$6.6375$&$0.995$&$331.70$&$7.622 $&$ 23890  $&$24624 $&$2.796$&$2.155$&$1.807$\\
&&$(1.134)$&$(0.103)$&$(69.03)$&$(1.482)$&&&&&\\
\hline
\multirow{2}{*}{Missp-Gaussian}&\multirow{2}{*}{Pairwise}&$5.647$&$1.065$&$404.81$&$7.080$&$24657 $&$26100 $&$2.760$&$2.127$&$1.810$\\
&&$(1.446)$&$(0.124)$&$(92.99)$&$( 2.032 )$&&&&&\\
\hline
\multirow{2}{*}{Missp-Gaussian}&\multirow{2}{*}{Standard}&$ 5.579 $&$ 1.056$&$ 349.90 $&$ 5.485 $& $2232 $ &$2249$&$2.774$&$ 2.134$& $1.815$\\
&&$(0.618)$&$(0.042  )$&$(52.030)$&$(0.434)$&&&&&\\
\hline
\multirow{2}{*}{Gaussian}&\multirow{2}{*}{Standard}&$  5.400$&$1.066 $&$ 225.89 $&$11.081   $& $2250$&$2266$&$2.846$&$2.177$& $1.820$\\
&&$(0.530)$&$(0.038)$&$(27.72)$&$(0.835)$&&&&&\\

\hline
%\multirow{2}{*}{Gaussian}&\multirow{2}{*}{Pairwise}&$  5.622  $&$1.067 $&$260.894$&$14.275  $& $ 24335 $&$25077 $&$2.8136$&$ 2.1633 $&$1.8313$\\
%&&$(1.264 )$&$( 0.113)$&$(53.208 )$&$(2.076  )$&&&&&\\
%\hline
\end{tabular}
}
\end{center}
\caption{ Estimates for the $t$ process using $wpl$ with bivariate $t$ distribution and misspecified (full and weighted pairwise) Gaussian likelihood
   and for the Gaussian process using standard likelihood
  with associated standard  error (in parenthesis) and PLIC and BLIC values,
when estimating the Australian maximum temperature dataset using two correlation models: ${\cal  M}_{\alpha,0.5}$ and ${\cal  GW}_{\alpha,0,5}$.
Last three columns:  associated empirical mean of RMSEs, MAEs  and CRPSs.} \label{tab:est}
\end{table}

Given the estimation of the mean regression  and variance parameters of the $t$ process,
 the estimated residuals
%\begin{equation}\label{def:weibeull2}
\begin{equation*}
\hat{ {Y}}^*_{4}(\bm{s}_i)=\frac{y(\bm{s}_i)-(\hat{\beta}_0+\hat{\beta}_1X(\bm{s}_i))}{(\hat{\sigma}^2)^{\frac{1}{2}}} \quad i=1,\ldots N
\end{equation*}
%\end{equation}
can be viewed as a realization of  the process $Y^*_{4}$.
Similarly we can compute the Gaussian residuals.
For the $t$ process we use the $wpl$ estimates  obtained with the bivariate $t$ distribution.
Both residuals  can be useful in order
to check  the  model assumptions, in particular the marginal and dependence assumptions.
In the top part of Figure \ref{fig:resa} a $qq$-plot  of the residuals of the Gaussian  and $t$ processes (from left to right) is depicted for the Mat{\'e}rn  case.
It can be appreciated that the $t$ model overall  fits better with respect the Gaussian model even if  it seems to  fail  to model properly the  right   tail behavior.
Moreover, the  graphical comparison between the  empirical and fitted semivariogram
of the residuals (bottom part
of Figure \ref{fig:resa}) highlights an apparent  better fitting  of the $t$ model.

\begin{figure}[h!]
\begin{tabular}{cc}
\includegraphics[width=7.1cm, height=6.4cm]{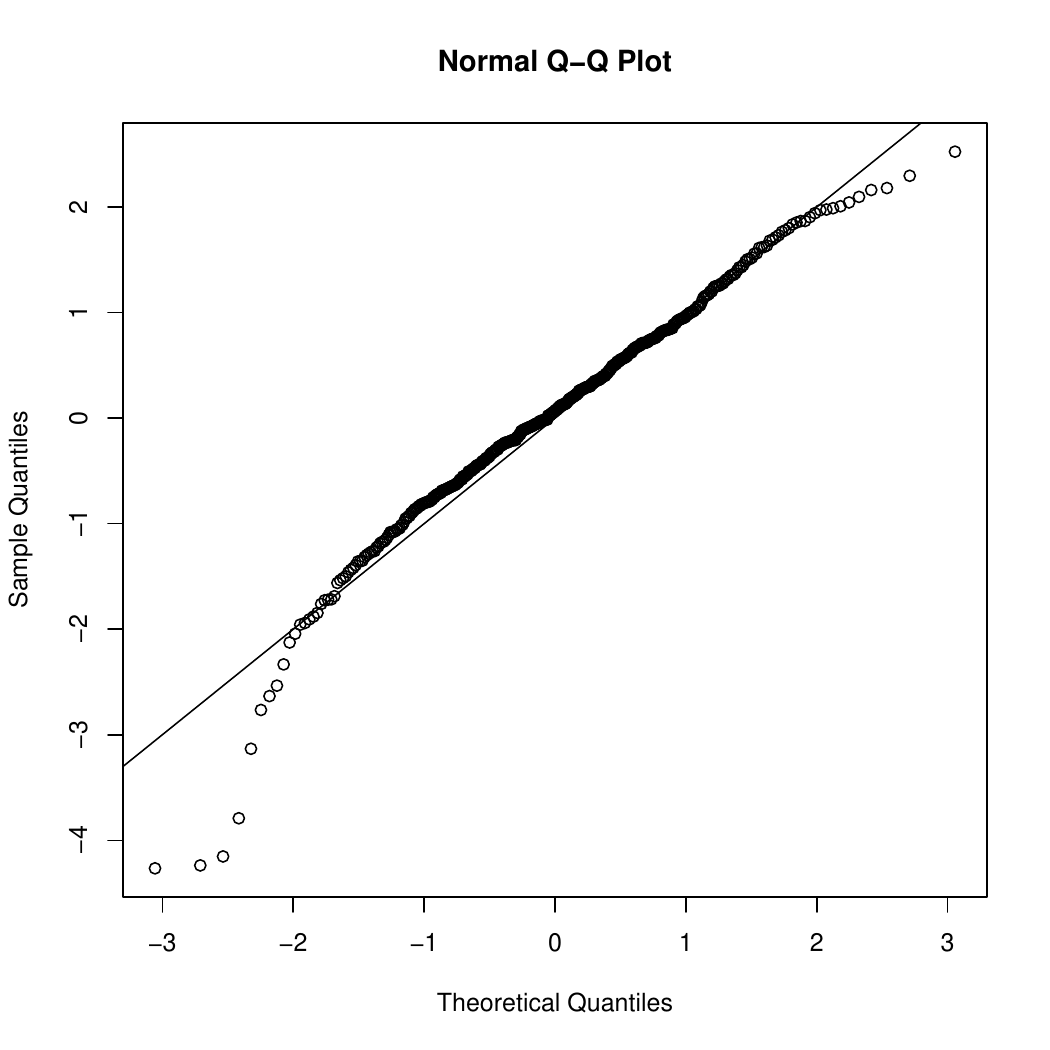} & \includegraphics[width=7.1cm, height=6.4cm]{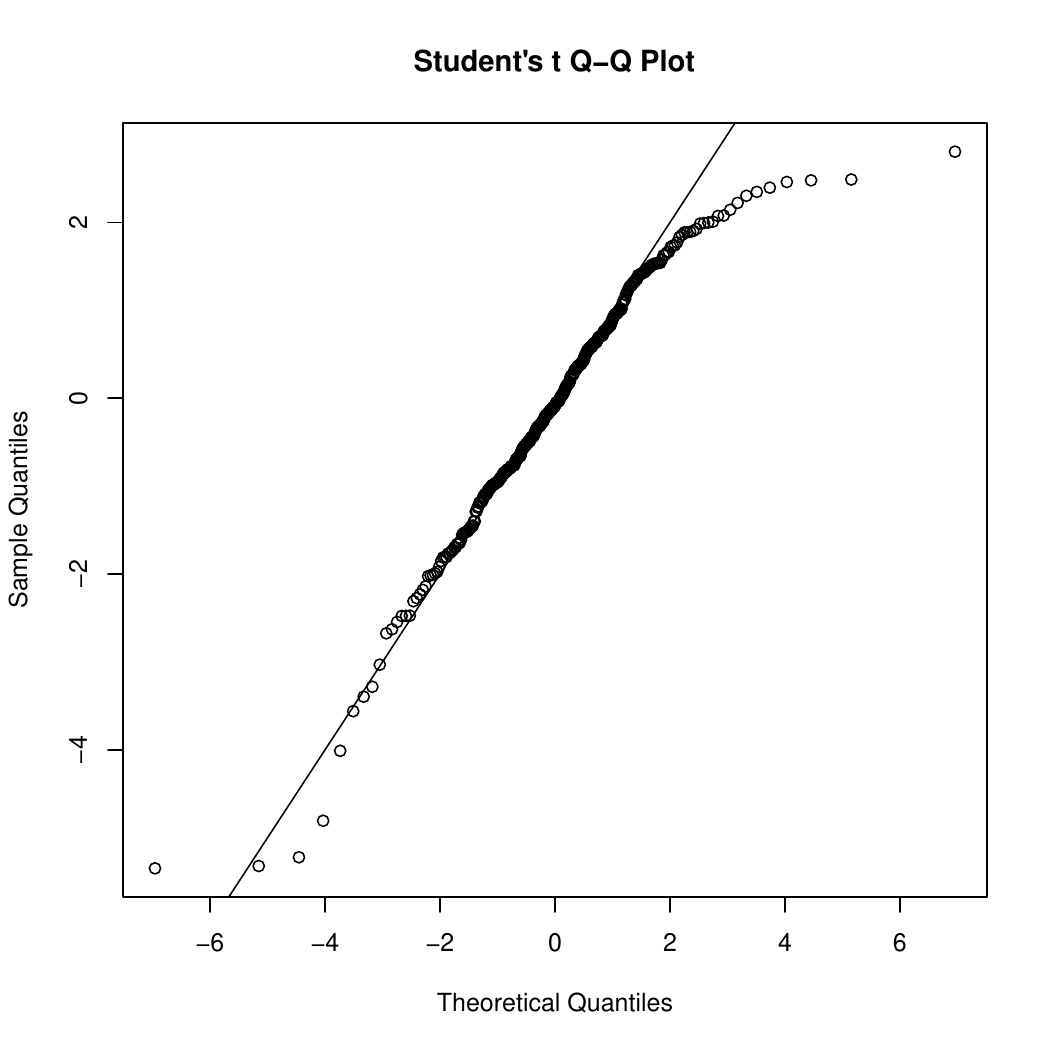}  \\
(a)&(b)\\
\includegraphics[width=7.1cm, height=6.4cm]{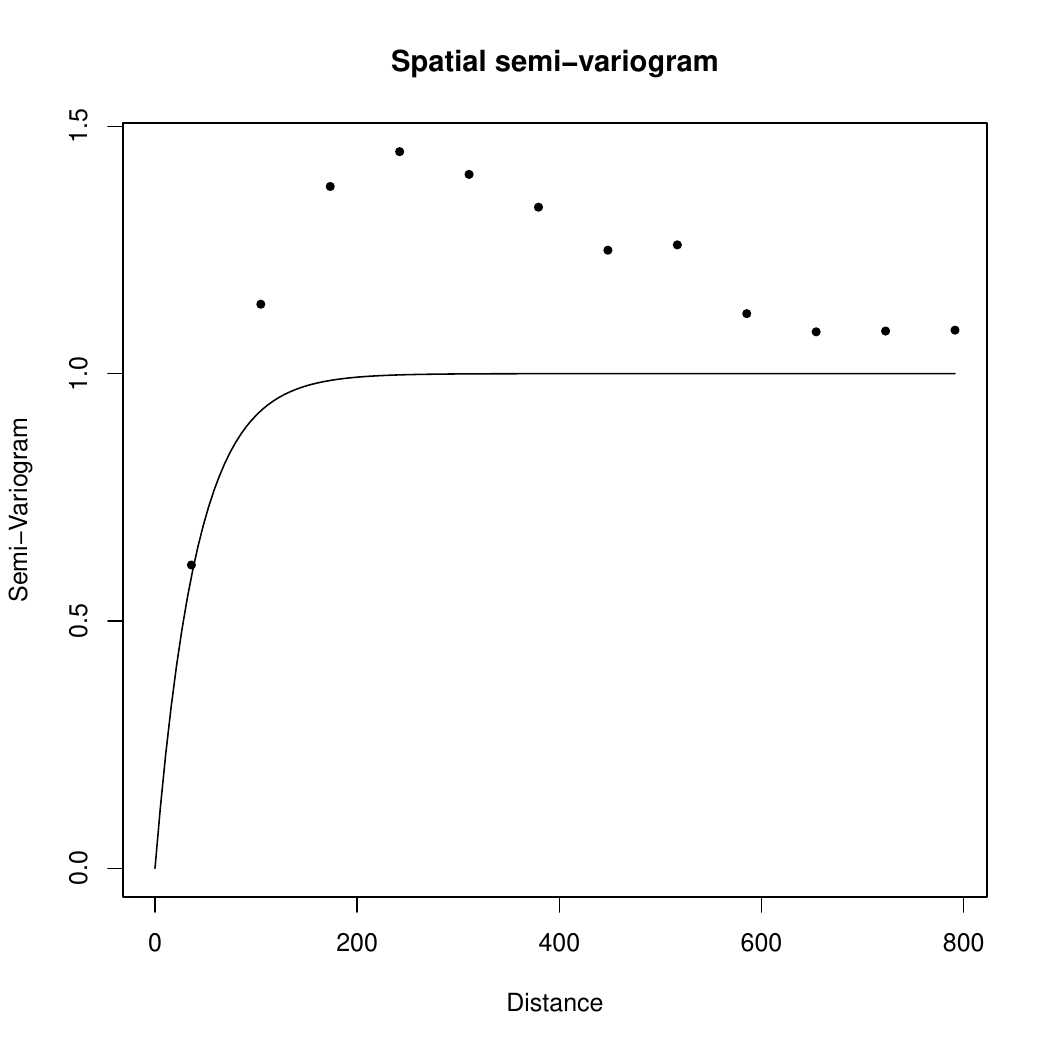} & \includegraphics[width=7.1cm, height=6.4cm]{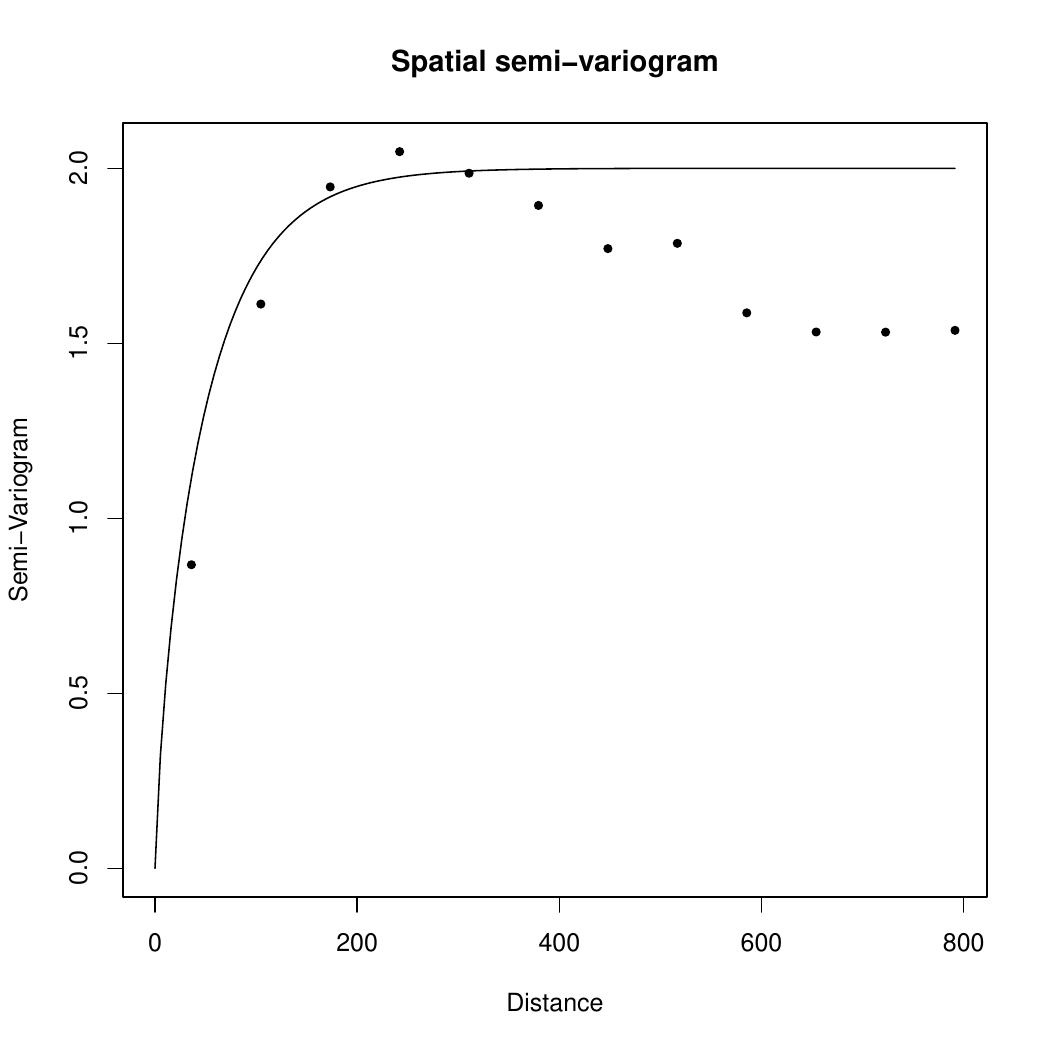}  \\
(c)&(d)\\
\end{tabular}
\caption{
Upper part: Q-Q plot of the residuals versus the estimated quantiles in the Gaussian and $t$ models ((a) and (b) respectively).
Bottom part: Empirical semi-variogram (dotted points) of the residuals versus the  estimated  semivariogram
(solid line) in the Gaussian and $t$ models ((c) and (d) respectively). Distances are expressed in Km.} \label{fig:resa}
\end{figure}

We want to further evaluate the predictive performances of Gaussian and $t$ processes using
RMSE and MAE as in Section \ref{piop}. Specifically, we use the following resampling approach: we randomly choose 80\%
of the data to predict and we use the estimates  previously obtained  in order to compute RMSE and MAE
 values at the remaining 20\% of the spatial locations.
  We repeat  the approach for 2000 times and
record all RMSEs and MAEs.
Specifically, for each $j-th$ left-out  sample   $(y_j^L(\bm{s}_1), \ldots,\ldots, y_j^L(\bm{s}_K))$, we compute

\begin{equation*}
\overline{\mathrm{RMSE}}_j=\left[ \frac{1}{K} \sum_{i= 1}^K\left( y_j^L(\bm{s}_i)-\widehat{Y}_j^L(\bm{s}_i)\right)^2\right]^{\frac{1}{2}}
\end{equation*}
and
\begin{equation*}
\overline{\mathrm{MAE}}_j= \frac{1}{K} \sum_{i= 1}^K| y_j^L(\bm{s}_i)-\widehat{Y}_j^L(\bm{s}_i)|,
\end{equation*}
where  $\widehat{Y}_j^L(\bm{s}_i)$ is the optimal or the  best linear optimal prediction for the Gaussian and $t$ processes respectively.
 Finally, we compute  the overall mean for both  Gaussian and $t$ processes and for both correlation models,  that is $\mathrm{RMSE}=\sum_{j=1}^{2000}\overline{\mathrm{RMSE}}_j/2000$ and
 $\mathrm{MAE}=\sum_{j=1}^{2000}\overline{\mathrm{MAE}}_j/2000$.

Additionally, to evaluate the marginal  predictive distribution performance, we also consider, for each sample, the continuous ranked
probability score (CRPS)
\citep{Gneiting:Raftery:2007}. For a single predictive cumulative distribution function $F$ and a verifying observation $y$, it is defined as:
\begin{equation*}
\mathrm{CRPS}(F,y)=\int\limits_{-\infty}^{\infty}(F(t)-\bm{1}_{[y,\infty]}(t))^2dt.
\end{equation*}
Specifically, for each $j-th$ left-out  sample, we consider the averaged
$\mathrm{CRPS}$ for  the Gaussian and $t$ distributions as:
\begin{equation}\label{poi}
\overline{\mathrm{CRPS}}_j= \frac{1}{K} \sum_{i= 1}^K \mathrm{CRPS}(F,y_j^L(\bm{s}_i))\quad F=F_G, F_{Y_4},
\end{equation}
for $j=1,\ldots,2000$.
 In particular in the Gaussian case
 \begin{footnotesize}
\begin{equation}\label{poli}
CRPS(F_G,y_j^L(\bm{s}_i))=\sigma \left(\frac{y_j^L(\bm{s}_i)-\mu(\bm{s})}{\sigma}\right)\left[2F_{G^*(\bm{s})}\left(\frac{y_j^L(\bm{s}_i)-\mu(\bm{s})}{\sigma}\right)-1\right]+2\sigma f_{G^*(\bm{s})}\left(\frac{y_j^L(\bm{s}_i)-\mu(\bm{s})}{\sigma}\right)-\frac{\sigma}{\sqrt{\pi}},
\end{equation}
\end{footnotesize}
and in the $t$ case with $4$ degrees of freedom:
\begin{footnotesize}
\begin{align}\label{ppo}
CRPS(F_{Y_4},y_j^L(\bm{s}_i))&=\sigma \left(\frac{y_j^L(\bm{s}_i)-\mu(\bm{s})}{\sigma}\right)\left[2F_{Y^*_{4}(\bm{s})}\left(\frac{y_j^L(\bm{s}_i)-\mu(\bm{s})}{\sigma}\right)-1\right]\\
&+
2\left[\frac{4\sigma^2+(y_j^L(\bm{s}_i)-\mu(\bm{s}))^2}{3\sigma}\right]f_{Y^*_4(\bm{s})}\left(\frac{y_j^L(\bm{s}_i)-\mu(\bm{s})}{\sigma}\right)\nonumber-\frac{4\sigma B\left(\frac{1}{2},\frac{7}{2}\right)}{3 B^2\left(\frac{1}{2},2\right)},
%CRPS(F_{Y_4},y)&=\sigma \left(\frac{y-\mu}{\sigma}\right)\left[2F_{Y^*_{4}}\left(\frac{y-\mu}{\sigma}\right)-1\right]+
%2\left[\frac{4\sigma^2+(y-\mu)^2}{\sigma(4-1)}\right]f_{Y^*_4}\left(\frac{y-\mu}{\sigma}\right)\nonumber\\
%&-\frac{2\sigma\sqrt{4}B\left(\frac{1}{2},4-\frac{1}{2}\right)}{(4-1)B^2\left(\frac{1}{2},\frac{4}{2}\right)}.
\end{align}
\end{footnotesize}
where $\mu(\bm{s})=\beta_0+\beta_1X(\bm{s})$.
We compute the CRPS in (\ref{poli}) and  (\ref{ppo})  plugging-in the estimates of the pairwise and standard (misspecified) likelihood estimation methods of $\beta_0$, $\beta_1$ and $\sigma$ using  the R package
\texttt{scoringRules} \citep{SR2019}.
 Finally, we compute  the overall mean for both  Gaussian and $t$ processes and for both correlation models,  that is $\mathrm{CRPS}=\sum_{j=1}^{2000}\overline{\mathrm{CRPS}}_j/2000$.

Table \ref{tab:est} reports the estimated RMSE, MAE  and CRPS.
As a general remark, the $t$ process outperforms the Gaussian process
for the three measures of prediction performance irrespective of the method of estimation and for  both correlation models. We point out  that RMSE and MAE are computed using the optimal  predictor in the Gaussian case
and the linear optimal in the $t$ case. However, the RMSE and MAE results highlight a better performance for the $t$ process.
In addition, a better  RMSE and MAE  for the Mat\'ern correlation model  with respect to the Wendland  is apparent, irrespective of the type of process.
The proposed $t$ process also leads to a clear better performance of the CRPS with respect to the Gaussian case.
In this specific example, the use of the  misspecified Gaussian $wpl$ estimates leads to the best results in terms of RMSE and  MAE.
On the other hand, the best CRPS results are achieved by using the $wpl$ estimates using the  proposed $t$ bivariate distribution
%In addition, CRPS is not affected by the choice of the correlation model, as expected,  since only the marginal distribution is involved in this case.

Finally, one important goal in spatial modeling of temperature data is
to create a high resolution  map in a spatial region using the observed data. In Figure
\ref{australia} (b), we plot  a high resolution map of the predicted residuals using the $t$ process
with underlying Mat\'ern correlation model estimated with $wpl$.

\section{Concluding remarks}\label{sec:6}

We have introduced a new  stochastic process with $t$ marginal distributions  for regression and dependence analysis when addressing
spatial %or spatiotemporal data
with
 heavy tails.
Our proposal allows overcoming any problem of identifiability associated with previously proposed spatial models with $t$ marginals and, as a consequence, the  model parameters can be estimated with just one realization of the process.
Additionally, the  proposed $t$ process partially inherits some geometrical properties  of the `parent' Gaussian process,
an appealing feature from a data analysis point of view.
We have also proposed
a possible generalization,
obtaining a new process with the marginal distribution of the skew-$t$ type
using the skew-Gaussian process proposed in   \cite{Zhang:El-Shaarawi:2010}.

%One interesting benefit of our proposal is that $t$ or skew $t$ processes can be easily defined on
%For this process, a closed form expression  of the finite dimensional distribution is provided.

%One benefit of our proposal is that  the  proposed models  depends on a `parent' Gaussian process
%that can be defined  not necessarily on euclidean spaces.
%For instance the same processes can be  easily defined on   a sphere of arbitrary  radius. % $r$
%In this case parametric correlation models defined on the sphere (\cite{gneiting2013}, \cite{porcubev})
%should be considered.

In our  proposal, a possible  limitation is  the lack of amenable expressions of the associated
 multivariate distributions. This prevents an inference approach based on the  full likelihood
 and the  computation of the optimal predictor.
 In the first case, our simulation study shows that  an inferential approach based on $wpl$,
% weighted pairwise likelihood,
using the bivariate $t$ distribution given in Theorem \ref{theo3}, could be an effective solution for estimating the unknown parameters.
%The computation of the $wpl$ estimator can be computationally demanding  for large datasets since it involves the computation of the Appell function.
An alternative less efficient solution   that requires  smaller computational burden can be  obtained by considering a misspecified Gaussian $wpl$
  %weighted pairwise likelihood
  using the  correlation function of the  $t$ process.
In the prediction case, our numerical experiments show that the optimal linear predictor  of the $t$ process
performs better than  the optimal Gaussian predictor when working with spatial data with heavy tails.

Another possible  drawback
concerns the restriction of the  degrees of freedom of the $t$ process to $\nu=3, 4,\ldots $.
under non-infinite divisibility of the  associated Gamma process. This problem could be solved by considering a Gamma process obtained by mixing  the proposed Gamma process
with a  process with beta marginals and using the results in \cite{YEO1991239}; however
the mathematics involved with this approach are much more challenging.

The estimation of the skew-$t$ process has not been addressed in this paper since the bivariate distribution in this case is quite complicated.
In principle, after a suitable parametrization, Gaussian misspecified $wpl$ can be performed  using   (\ref{CC99})
to estimate the parameters of the skew-$t$ process. In this case an additional issue is the
behavior of the information matrix when $\eta=0$ \citep{AR22}.
Finally, a $t$ process  with asymmetric marginal distribution can also   be obtained by considering
some  specific transformations  of the proposed standard $t$ process as in \cite{AV2011}
or under the two-piece distribution framework \citep{AV2015}
 and this will be studied in future work.

\section*{Acknowledgements}
%The research work conducted by Moreno Bevilacqua  was supported in
%part  by FONDECYT grant 1160280 and by  Iniciativa Cient\'ifica Milenio - Minecon Nucleo Milenio MESCD, Chile.
Partial support was provided by FONDECYT grant 1160280, Chile
and  by Millennium
Science Initiative of the Ministry
of Economy, Development, and
Tourism, grant "Millenium
Nucleus Center for the
Discovery of Structures in
Complex Data"
for Moreno Bevilacqua and by Proyecto de Iniciaci\'on Interno
DIUBB 173408 2/I de la Universidad del B\'io-B\'io for Christian Caama\~no.

\newpage
\appendix

\section{Appendix} \label{sec:appendix}

\subsection{Proof  of Theorem \ref{theo0}}

\begin{proof}
 Set $R_{\nu}\equiv W^{-\frac{1}{2}}_{\nu}$.
Then the  correlation function of $Y^*_{\nu}$ is given by
\begin{equation}\label{covskewellip}
\rho_{Y^*_{\nu}}(\bm{h})=\left(\frac{\nu-2}{\nu}\right)
\left(\E(R_{\nu}(\bm{s}_i)R_{\nu}(\bm{s}_j))\rho(\bm{h})\right).
\end{equation}

To find a closed form for $\E(R_{\nu}(\bm{s}_i)R_{\nu}(\bm{s}_j)$, we need the bivariate distribution of $\bm{R}_{\nu;ij}$
that can be easily obtained from
density of the bivariate random vector $\bm{W}_{\nu;ij}$
 given by \citep{Bevilacqua:2018ab}:
\begin{footnotesize}
\begin{equation}\label{pairchi2}
f_{\bm{W}_{\nu;ij}}(w_{i},w_j)=\frac{2^{-\nu}\nu^{\nu}(w_iw_j)^{\nu/2-1}e^{-\frac{\nu(w_i+w_j)}{2(1-\rho^2(\bm{h}))}}}{\Gamma\left(\frac{\nu}{2}\right)(1-\rho^2(\bm{h}))^{\nu/2}}
\left(\frac{\nu\sqrt{\rho^2(\bm{h})w_iw_j}}{2(1-\rho^2(\bm{h}))}\right)^{1-\nu/2}I_{\nu/2-1}
\left( \frac{ \nu\sqrt{\rho^2(\bm{h})w_iw_j} } {(1-\rho^2(\bm{h}))}\right)
\end{equation}
\end{footnotesize}
where $I_{\alpha}(\cdot)$ denotes the modified Bessel function of the first kind of order $\alpha$.
\cite{VJ:1967} show the infinite divisibility of $\bm{W}_{\nu;ij}$.

Then, for each $\nu>2$, the bivariate distribution of $\bm{R}_{\nu;ij}$ is given
by:
\begin{footnotesize}
\begin{equation}\label{pairsquartgaminv2}
f_{\bm{R}_{\nu;ij}}(\bm{r}_{ij})=\frac{2^{-\nu+2}\nu^{\nu}(r_ir_j)^{-\nu-1}e^{-\frac{\nu}{2(1-\rho^2(\bm{h}))}\left(\frac{1}{r^2_i}+\frac{1}{r^2_j}\right)}}{\Gamma\left(\frac{\nu}{2}\right)(1-\rho^2(\bm{h}))^{\nu/2}}
\left(\frac{\nu\rho(\bm{h})}{2(1-\rho^2(\bm{h}))r_ir_j}\right)^{1-\frac{\nu}{2}}I_{\frac{\nu}{2}-1}\left(\frac{\nu\rho(\bm{h})}{(1-\rho^2(\bm{h}))r_ir_j}\right)
\end{equation}
\end{footnotesize}
Using the  identity
${}_0F_1(;b;x)=\Gamma(b)x^{(1-b)/2}I_{b-1}(2\sqrt{x})$
 and the  series expansion of hypergeometric function ${}_0F_1$ in (\ref{pairsquartgaminv2}) we have
\begin{footnotesize}
\begin{align}\label{cal}
\E(R^a(\bm{s}_i)R^b(\bm{s}_j))&=
\frac{2^{-\nu+2}\nu^\nu}{\Gamma^2\left(\frac{\nu}{2}\right)(1-\rho^2(\bm{h}))^{\nu/2}}
\int\limits_{\mathbb{R}_+^2}r_i^{-\nu+a-1}r_j^{-\nu+b-1}e^{-\frac{\nu}{2(1-\rho^2(\bm{h}))r_i^2}}e^{-\frac{\nu}{2(1-\rho^2(\bm{h}))r_j^2}}\nonumber\\
&\times{}_0F_1\left(\frac{\nu}{2};\frac{\nu^2\rho^2(\bm{h})}{4(1-\rho^2(\bm{h}))^2r^2_ir^2_j}\right)d\bm{r}_{ij}\nonumber\\
&=\frac{2^{-\nu+2}\nu^\nu}{\Gamma^2\left(\frac{\nu}{2}\right)(1-\rho^2(\bm{h}))^{\nu/2}}
\sum\limits_{k=0}^{\infty}\int\limits_{\mathbb{R}_+^2}r_i^{-\nu+a-2k-1}r_j^{-\nu+b-2k-1}e^{-\frac{\nu}{2(1-\rho^2(\bm{h}))r_i^2}}e^{-\frac{\nu}{2(1-\rho^2(\bm{h}))r_j^2}}\nonumber\\
&\times\frac{1}{k!\left(\frac{\nu}{2}\right)_k}
\left(\frac{\rho^2(\bm{h})\nu^2}{4(1-\rho^2(\bm{h}))}\right)^kd\bm{r}_{ij}\nonumber\\
&=\frac{2^{-\nu+2}\nu^\nu}{\Gamma^2\left(\frac{\nu}{2}\right)(1-\rho^2(\bm{h}))^{\nu/2}}
\sum\limits_{k=0}^{\infty}\frac{I(k)}{k!\left(\frac{\nu}{2}\right)_k}\left(\frac{\rho^2(\bm{h})\nu^2}{4(1-\rho^2(\bm{h}))}\right)^k
\end{align}
\end{footnotesize}
where, using Fubini's Theorem
\begin{footnotesize}
\begin{equation*}
I(k)=\int\limits_{\mathbb{R}_+}r_i^{-\nu+a-2k-1}e^{-\frac{\nu}{2(1-\rho^2(\bm{h}))r_i^2}}dr_i
\int\limits_{\mathbb{R}_+}r_j^{\nu+b-2k-1}e^{-\frac{\nu}{2(1-\rho^2(\bm{h}))r_j^2}}dr_j
\end{equation*}
\end{footnotesize}
Using the univariate density
$f_{R_{\nu}(\bm{s})}(r)=2\left(\frac{\nu}{2}\right)^{\nu/2}r^{-\nu-1}e^{-\frac{\nu}{2r^2}}/\Gamma\left(\frac{\nu}{2}\right)$,
we obtain
\begin{footnotesize}
\begin{align}\label{res1}
I(k)&=\Gamma\left(\frac{\nu-a}{2}+k\right)\Gamma\left(\frac{\nu-b}{2}+k\right)2^{\frac{\nu-a}{2}+k-1}2^{\frac{\nu-b}{2}+k-1}\left(\frac{(1-\rho^2(\bm{h}))}{\nu}\right)^{\frac{\nu-a}{2}+k}
\left(\frac{(1-\rho^2(\bm{h}))}{\nu}\right)^{\frac{\nu-b}{2}+k}
\end{align}
\end{footnotesize}
and combining equations (\ref{res1}) and (\ref{cal}), we obtain
\begin{footnotesize}
\begin{align*}
\E(R^a(\bm{s}_i)R^b(\bm{s}_j))&=
\frac{2^{-(a+b)/2}\nu^{(a+b)/2}(1-\rho^2(\bm{h}))^{(\nu-a-b)/2}\Gamma\left(\frac{\nu-a}{2}\right)\Gamma\left(\frac{\nu-b}{2}\right)}{\Gamma^2\left(\frac{\nu}{2}\right)}
\sum\limits^{\infty}_{k=0}\frac{\left(\frac{\nu-a}{2}\right)_k\left(\frac{\nu-b}{2}\right)_k}{k!\left(\frac{\nu}{2}\right)_k}\rho^{2k}(\bm{h})\\
&=\frac{2^{-(a+b)/2}\nu^{(a+b)/2}(1-\rho^2(\bm{h}))^{(\nu-a-b)/2}\Gamma\left(\frac{\nu-a}{2}\right)\Gamma\left(\frac{\nu-b}{2}\right)}{\Gamma^2\left(\frac{\nu}{2}\right)}
{}_2F_1\left(\frac{\nu-a}{2},\frac{\nu-b}{2};\frac{\nu}{2};\rho^2(\bm{h})\right)
\end{align*}
\end{footnotesize}
Then, using the Euler transformation, we obtain
\begin{footnotesize}
\begin{align}\label{esquartgaminv2}
\E(R_{\nu}^a(\bm{s}_i)R_{\nu}^b(\bm{s}_j))&=\frac{2^{-(a+b)/2}\nu^{(a+b)/2}}{\Gamma^2\left(\frac{\nu}{2}\right)}
\Gamma\left(\frac{\nu-a}{2}\right)\Gamma\left(\frac{\nu-b}{2}\right){}_2F_1\left(\frac{a}{2},\frac{b}{2};\frac{\nu}{2};\rho^2(\bm{h})\right)
\end{align}
\end{footnotesize}
for $\nu>a$ and $\nu>b$. Finally, setting $a=b=1$ in
(\ref{esquartgaminv2}) and using it in (\ref{covskewellip}) we
obtain (\ref{CC}).
\end{proof}

\subsection{Proof  of Theorem \ref{theoiii}}

\begin{proof}
If  $G$ is a  weakly stationary Gaussian process with correlation $\rho(\bm{h})$ then from (\ref{CC}) it is straightforward to see that
$Y^*_{\nu}$ is also weakly stationary.
Points b) and c) can be shown using the relations between the geometrical properties of a stationary process and the associated correlation.
Specifically, the mean-square continuity and
the $m$-times mean-square differentiability of $Y^*_{\nu}$  are equivalent to the continuity and $2m$-times
differentiability of $\rho_{Y^*_{\nu}}(\bm{h})$ at $\bm{h}=\bm{0}$, respectively \citep{Stein:1999}.
Recall that the correlation function of $Y^*_{\nu}$  is given
by:

\begin{equation}\label{dai2}
\rho_{Y^*_{\nu}}(\bm{h})=a(\nu)    \left[{}_2F_1\left(\frac{1}{2},\frac{1}{2};\frac{\nu}{2};\rho^2(\bm{h})\right)\rho(\bm{h})\right].
\end{equation}
with $a(\nu)=\frac{(\nu-2)\Gamma^2\left(\frac{\nu-1}{2}\right)}{2\Gamma^2\left(\frac{\nu}{2}\right)}$.
Using (\ref{JJ}) it can be easily seen  that  $\rho(\bm{\bm{0}})=1$ if and only if  $\rho_{Y^*_{\nu}}(\bm{0})=1$. Then $Y^*_{\nu}$
is mean-square continuous if and only if $G$ is mean-square continuous.

%$\rho(\bm{0})=1$, then    $\rho_{Y^*_{\nu}}(\bm{0})=1$ for each $\nu>2$.

%Since $g(x)$ is continous at $1$ then

For  the  mean square differentiability, let  $G$ $m-$times mean square differentiable.   Using  iteratively  the $n-$th derivative of the ${}_2F_1$ function
with respect to $x$:
\begin{equation}\label{dai}
{}_2F^{(n)}_1(a,b,c,x)=\frac{(a)_n(b)_n}{(c)_n} {}_2F_1(a+n,b+n,c+n,x),\qquad n=1,2,\ldots
\end{equation}
and  applying  the convergence  condition $c>b+a$ of identity   (\ref{JJ}),  it can be shown   that
  $\rho^{(2m)}_{Y^*_{\nu}}(\bm{h})|_{\bm{h}=\bm{0}}< \infty $ if $\nu>2(2m+1)$
 and, as a consequence,
$Y^*_{\nu}$ is $m-$times  mean square differentiable under this condition.
On the other hand if $\nu\leq 2(2m+1)$ then  $Y^*_{\nu}$ is $(m-k)$-times mean-square differentiable
if
  $2(2(m-k)+1)< \nu \leq  2(2(m-k)+3)$,  for $k=1, \ldots, m$.

For instance, let assume that $G$ is   $1-$times mean square differentiable.  This implies that $\rho^{(i)}(\bm{h})|_{\bm{h}=\bm{0}}< \infty $, $i=1,2$.
Applying (\ref{dai}) to (\ref{dai2}), the  second derivative of $\rho_{Y^*_{\nu}}(\bm{h})$  is given by:
\begin{footnotesize}
\begin{align*}
\rho^{(2)}_{Y^{*}_{\nu}}(\bm{h})&=\frac{a(\nu)}{\nu(\nu+2)}\bigg[(6+3\nu){}_2F_1\left(1.5,1.5;1+\frac{\nu}{2};\rho^{2}(\bm{h})\right)\rho(\bm{h})\{\rho^{(1)}(\bm{h})\}^{2}\\
&+9{}_2F_1\left(2.5,2.5;2+\frac{\nu}{2};\rho^{2}(\bm{h})\right)\rho^3(\bm{h})\{\rho^{(1)}(\bm{h})\}^{2}+\nu(2+\nu){}_2F_1\left(0.5,0.5;\frac{\nu}{2};\rho^{2}(\bm{h})\right)\rho^{(2)}(\bm{h})\\
&+(2+\nu){}_2F_1\left(1.5,1.5;1+\frac{\nu}{2};\rho^{2}(\bm{h})\right)\rho^2(\bm{h})\rho^{(2)}(\bm{h})\bigg]
\end{align*}
\end{footnotesize}
Then, applying  the convergence  condition of identity   (\ref{JJ}),  $\rho^{(2)}_{Y^{*}_{\nu}}(\bm{h})|_{\bm{h}=\bm{0}}< \infty $   if $2+\nu/2>5$ that is  $\nu>6$.  Therefore $Y^*_{\nu}$ is $1-$times  mean square differentiable
if $\nu>6$ and $0-$times  mean square differentiable if $2<\nu\leq 6$.

Point d) can be  shown recalling that a process $F$ is  long-range dependent
if the  correlation of $F$ is such that  $\int_{\R^n_{+}} |\rho_F(\bm{h})|d^n \bm{h}=\infty$  \citep{LiTe09}.
Direct inspection, using series expansion of the hypergeometric function,
shows that
$\int_{\R^n_{+}} |\rho_{Y^*_{\nu}}(\bm{h})|d^n \bm{h}=\infty$ if and only if
$\int_{\R^n_{+}} |\rho(\bm{h})|d^n \bm{h}=\infty$
 and, as a consequence, $Y^*_{\nu}$ has long-range dependence if and only if $G$  has long-range dependence.
%Expanding the hypegeoemetric function we have
%\begin{footnotesize}
%\begin{align}
%\int_{\R^n_{+}} |\rho_{Y^*_{\nu}}(\bm{h})|d^n \bm{h}&=a(\nu)
%\int_{\R^n_{+}} {}_2F_1\left(\frac{1}{2},\frac{1}{2};\frac{\nu}{2};\rho(\bm{h})^2\right)|\rho(\bm{h})|d^n \bm{h}\nonumber\\
%&=a(\nu)\sum\limits_{k=0}^{\infty}\frac{\left(\frac{1}{2}\right)^2_{k}}{k!\left(\frac{\nu}{2}\right)_k}\int_{\R^n_{+}}|\rho(\bm{h})|^{2k+1}d^n \bm{h}. \nonumber
%\end{align}
%\end{footnotesize}
%For $k=0$ we have
%$\int_{\R^n_{+}} |\rho_{Y^*_{\nu}}(\bm{h})|d^n \bm{h}$
%to $\int_{\R^n_{+}} |\rho(\bm{h})|d^n \bm{h}=\infty$. This implies $\int_{\R^n_{+}} |\rho_{Y^*_{\nu}}(\bm{h})|d^n \bm{h}=\infty$
%that is $Y^*_{\nu}$ is long range dependence.
%}

Finally, %it is apparent from (\ref{CC}) that $\rho_{Y^*_{\nu}}(\bm{h})=0 \Leftrightarrow \rho(\bm{h})=0$  and
note that %$0<a(\nu)\leq 1$ for $\nu>2$. Additionally,
if  $\nu>2$ then $a(\nu){}_2F_1\left(\frac{1}{2},\frac{1}{2};\frac{\nu}{2};0\right)=a(\nu)$, $a(\nu){}_2F_1\left(\frac{1}{2},\frac{1}{2};\frac{\nu}{2};1\right)=1$
and $a(\nu) {}_2F_1\left(\frac{1}{2},\frac{1}{2};\frac{\nu}{2};x^2\right)$ is not decreasing  in  $0\leq x \leq1$.
This implies
$a(\nu) {}_2F_1\left(\frac{1}{2},\frac{1}{2};\frac{\nu}{2};x^2\right) \leq 1$ that is
$\rho_{Y^*_{\nu}}(\bm{h})\leq \rho(\bm{h})$. Moreover, $\lim\limits_{\nu \to \infty }a(\nu)=1$ and using   series expansion of the hypergeometric function:
%it can be shown that
\begin{footnotesize}
\begin{align*}
\lim\limits_{\nu\rightarrow\infty}{}_2F_1\left(\frac{1}{2},\frac{1}{2};\frac{\nu}{2};\rho(\bm{h})^2\right)%&=\lim\limits_{\nu\rightarrow\infty}\sum\limits_{k=0}^{\infty}\frac{\left(\frac{1}{2}\right)^2_k %\rho(\bm{h})^2}{k!\left(\frac{\nu}{2}\right)_k}\nonumber\\
   &=\lim\limits_{\nu\rightarrow\infty}\bigg[1+\frac{\left(\frac{1}{2}\right)^2_1 \rho(\bm{h})^2}{1\left(\frac{\nu}{2}\right)_1}+\frac{\left(\frac{1}{2}\right)^2_2 \rho(\bm{h})^4}{1\left(\frac{\nu}{2}\right)_2}%\nonumber\\
   %&
   +\ldots+\frac{\left(\frac{1}{2}\right)^2_k \rho(\bm{h})^{2k}}{1\left(\frac{\nu}{2}\right)_k}+\ldots\bigg]\nonumber\\
   &=\lim\limits_{\nu\rightarrow\infty}\bigg[1+\frac{2\rho(\bm{h})^2}{\nu}+\frac{\left(\frac{1}{2}\right)^2\left(\frac{1}{2}+1\right)^2\rho(\bm{h})^2}{2!\left(\frac{\nu}{2}\right)\left(\frac{\nu}{2}+1\right)}\nonumber\\
   &+\ldots+\frac{\left(\frac{1}{2}\right)^2\left(\frac{1}{2}+1\right)^2\cdots\left(\frac{1}{2}+k-1\right)^2\rho(\bm{h})^{2k}}{k!\left(\frac{\nu}{2}\right)\left(\frac{\nu}{2}+1\right)\cdots\left(\frac{\nu}{2}+k-1\right)}+\ldots\bigg]\nonumber\\
   &=1.\nonumber
\end{align*}
\end{footnotesize}
%$\lim\limits_{\nu \to \infty }{}_2F_1\left(\frac{1}{2},\frac{1}{2};\frac{\nu}{2}; \rho^2(\bm{h})\right)=1$.
This implies
$\lim\limits_{\nu \to \infty } \rho_{Y^*_{\nu}}(\bm{h})=\rho(\bm{h})$.
\end{proof}

\subsection{Proof  of Theorem \ref{theo3}}

\begin{proof}
 Using  the  identity
${}_0F_1(;b;x)=\Gamma(b)x^{(1-b)/2}I_{b-1}(2\sqrt{x})$
 and the  series expansion of hypergeometric function ${}_0F_1$, then
under the transformation $g_i=y_i\sqrt{w_i}$ and
$g_j=y_j\sqrt{w_j}$ with Jacobian $J((g_i,g_j)\to
(y_i,y_j))=(w_iw_j)^{1/2}$, we have:
\begin{footnotesize}
\begin{align}\label{degam}
f_{\bm{Y}^*_{ij}}(\bm{y}_{ij})&=\int\limits_{\mathbb{R}_+^2}f_{\bm{G}_{ij}|\bm{W}_{ij}}(\bm{g}_{ij}|\bm{w}_{ij})f_{\bm{W}_{ij}}(\bm{w}_{ij})Jd\bm{w}_{ij}\nonumber\\
&=\frac{2^{-\nu}\nu^{\nu}}{2\pi\Gamma^2\left(\frac{\nu}{2}\right)(1-\rho^2(\bm{h}))^{(\nu+1)/2}}\int\limits_{\mathbb{R}^2_+}(w_iw_j)^{(\nu+1)/2-1}
e^{-\frac{1}{2(1-\rho^2(\bm{h}))}\left[w_iy_i^2+w_jy_j^2-2\rho(\bm{h})\sqrt{w_iw_j}y_iy_j\right]}\nonumber\\
&\times e^{-\frac{\nu(w_i+w_j)}{2(1-\rho^2(\bm{h}))}}{}_0F_1\left(\frac{\nu}{2};\frac{\nu^2\rho^2(\bm{h})w_iw_j}{4(1-\rho^2(\bm{h}))^2}\right)d\bm{w}_{ij}\nonumber\\
&=\frac{2^{-\nu}\nu^{\nu}}{2\pi\Gamma^2\left(\frac{\nu}{2}\right)(1-\rho^2(\bm{h}))^{(\nu+1)/2}}\int\limits_{\mathbb{R}^2_+}(w_iw_j)^{(\nu+1)/2-1}
e^{-\frac{1}{2(1-\rho^2(\bm{h}))}\left[y_i^2-2\rho(\bm{h})\sqrt{\frac{w_j}{w_i}}y_iy_j+\nu\right]w_i}e^{-\frac{(y_j^2+\nu)w_j}{2(1-\rho^2(\bm{h}))}}\nonumber\\
&\times\sum\limits_{k=0}^{\infty}\frac{1}{k!\left(\frac{\nu}{2}\right)_k}\left(\frac{\nu^2\rho^2(\bm{h})w_iw_j}{4(1-\rho^2(\bm{h}))^2}\right)^kd\bm{w}_{ij}\nonumber\\
&=\frac{2^{-\nu}\nu^{\nu}}{2\pi\Gamma^2\left(\frac{\nu}{2}\right)(1-\rho^2(\bm{h}))^{(\nu+1)/2}}
\sum\limits_{k=0}^{\infty}\frac{I(k)}{k!\left(\frac{\nu}{2}\right)_k}\left(\frac{\nu^2\rho^2(\bm{h})}{4(1-\rho^2(\bm{h}))^2}\right)^k
\end{align}
\end{footnotesize}
using (3.462.1) of \cite{Gradshteyn:Ryzhik:2007}, we obtain
\begin{footnotesize}
\begin{align}\label{demt1}
I(k)&=\int\limits_{\mathbb{R}_+}w_j^{(\nu+1)/2+k-1}e^{-\frac{(y_j^2+\nu)w_j}{2(1-\rho^2(\bm{h}))}}\left[\int\limits_{\mathbb{R}_+}w_i^{(\nu+1)/2+k-1}e^{\left[-\frac{(y_i^2+\nu)}{2(1-\rho^2(\bm{h}))}w_i-\frac{\rho(\bm{h})\sqrt{w_j}y_iy_j}{(\rho^2(\bm{h})-1)}\sqrt{w_i}\right]}dw_i\right]dw_j\nonumber\\
&=2\left(\frac{y_i^2+\nu}{(1-\rho^2(\bm{h}))}\right)^{-\left(\frac{\nu+1}{2}+k\right)}\Gamma\left(\nu+1+2k\right)\int\limits_{\mathbb{R}_+}w_j^{(\nu+1)/2+k-1}
e^{\left[\frac{\rho^2(\bm{h})y_i^2y_j^2}{4(1-\rho^2(\bm{h}))(y_i^2+\nu)}-\frac{(y_j^2+\nu)}{2(1-\rho^2(\bm{h}))}\right]w_j}\nonumber\\
&\times D_{-(\nu+1+2k)}\left(-\frac{\rho(\bm{h})y_iy_j\sqrt{w_j}}{\sqrt{(1-\rho^2(\bm{h}))(y_i^2+\nu)}}\right)dw_j\nonumber\\
&=2\left(\frac{y_i^2+\nu}{(1-\rho^2(\bm{h}))}\right)^{-\left(\frac{\nu+1}{2}+k\right)}\Gamma\left(\nu+1+2k\right)A(k)
\end{align}
\end{footnotesize}
where $D_{n}(x)$ is the parabolic cylinder function.
%Note that a
%partial solution can be obtained using (7.725.6) in
%\cite{Gradshteyn:Ryzhik:2007}. If $\rho(\bm{h})y_iy_j<0$ then
%\begin{footnotesize}
%\begin{eqnarray}\label{demt8}
%I(k)&=&2\left(\frac{y_i^2+\nu}{(1-\rho^2(\bm{h}))}\right)^{-\left(\frac{\nu+1}{2}+k\right)}\frac{\sqrt{\pi}\Gamma^2\left(\nu+1+2k\right)2^{-\frac{1}{2}(3\nu+6k+1)}}{\Gamma\left(\nu+2k+\frac{3}{2}\right)}
%\left[\frac{(y_j^2+\nu)}{2(1-\rho^2(\bm{h}))}\right]^{-\left(\frac{\nu+1}{2}+k\right)}\nonumber\\
%&\times&{}_2F_1\left(\frac{\nu+1}{2}+k,\frac{\nu+1}{2}+k;\nu+2k+\frac{3}{2};1-\frac{\rho^2(\bm{h})y_i^2y_j^2}{(y_i^2+\nu)(y_j^2+\nu)}\right)
%\end{eqnarray}
%\end{footnotesize}
%and combining equations (\ref{demt8}) and (\ref{degam}), the
%solution is
%\begin{footnotesize}
%\begin{eqnarray*}
%f_{\bm{Y}^*_{ij}}(\bm{y}_{ij})&=&\frac{2^{-2\nu}\nu^{\nu}[(y_i^2+\nu)(y_j^2+\nu)]^{-(\nu+1)/2}}{\sqrt{\pi}\Gamma^2\left(\frac{\nu}{2}\right)(1-\rho^2(\bm{h}))^{-(\nu+1)/2}}
%\sum\limits_{k=0}^{\infty}\frac{\Gamma^2\left(\nu+1+2k\right)}{k!\left(\frac{\nu}{2}\right)_k\Gamma\left(\nu+2k+\frac{3}{2}\right)}\left(\frac{\nu^2\rho^2(\bm{h})}{16(y_i^2+\nu)(y_j^2+\nu)}\right)^k\nonumber\\
%&\times&{}_2F_1\left(\frac{\nu+1}{2}+k,\frac{\nu+1}{2}+k;\nu+2k+\frac{3}{2};1-\frac{\rho^2(\bm{h})y_i^2y_j^2}{(y_i^2+\nu)(y_j^2+\nu)}\right)
%\end{eqnarray*}
%\end{footnotesize}
Now, considering (9.240) of
\cite{Gradshteyn:Ryzhik:2007}:
\begin{footnotesize}
\begin{align}\label{cylinderp}
D_{-(\nu+1+2k)}\left(-\frac{\rho(\bm{h})y_iy_j\sqrt{w_j}}{\sqrt{(1-\rho^2(\bm{h}))(y_i^2+\nu)}}\right)&=b_1e^{-\frac{\rho^2(\bm{h})y^2_iy^2_jw_j}{4(1-\rho^2(\bm{h}))(y_i^2+\nu)}}\nonumber\\
&\times {}_1F_1\left(\frac{\nu+1}{2}+k;\frac{1}{2};\frac{\rho^2(\bm{h})y^2_iy^2_jw_j}{2(1-\rho^2(\bm{h}))(y_i^2+\nu)}\right)\nonumber\\
&+b_2\sqrt{w_j}e^{-\frac{\rho^2(\bm{h})y^2_iy^2_jw_j}{4(1-\rho^2(\bm{h}))(y_i^2+\nu)}}\nonumber\\
&\times {}_1F_1\left(\frac{\nu}{2}+k+1;\frac{3}{2};\frac{\rho^2(\bm{h})y^2_iy^2_jw_j}{2(1-\rho^2(\bm{h}))(y_i^2+\nu)}\right)
\end{align}
\end{footnotesize}
where $b_1=\frac{2^{-(\nu+1)/2+k}\sqrt{\pi}}{\Gamma\left(\frac{\nu}{2}+k+1\right)}$ and $b_2=\frac{2^{-\nu/2-k}\sqrt{\pi}\rho(\bm{h})y_iy_j}{\Gamma\left(\frac{\nu+1}{2}+k\right)\sqrt{(1-\rho^2(\bm{h}))(y_i^2+\nu)}}$. Replacing equations (\ref{cylinderp}) in
(\ref{demt1}) and using (7.621.4) of
\cite{Gradshteyn:Ryzhik:2007}, we obtain
\begin{footnotesize}
\begin{align}\label{demtff}
A(k)&=b_1\int\limits_{\mathbb{R}_+}w_j^{(\nu+1)/2+k-1}
e^{-\frac{(y^2_j+\nu)}{2(1-\rho^2(\bm{h}))}w_j}{}_1F_1\left(\frac{\nu+1}{2}+k;\frac{1}{2};\frac{\rho^2(\bm{h})y^2_iy^2_jw_j}{2(1-\rho^2(\bm{h}))(y_i^2+\nu)}\right)dw_j\nonumber\\
&+b_2\int\limits_{\mathbb{R}_+}w_j^{\nu/2+k+1-1}
e^{-\frac{(y^2_j+\nu)}{2(1-\rho^2(\bm{h}))}w_j}{}_1F_1\left(\frac{\nu}{2}+k+1;\frac{3}{2};\frac{\rho^2(\bm{h})y^2_iy^2_jw_j}{2(1-\rho^2(\bm{h}))(y_i^2+\nu)}\right)dw_j\nonumber\\
&=b_1\Gamma\left(\frac{\nu+1}{2}+k\right)\left(\frac{y^2_j+\nu}{2(1-\rho^2(\bm{h}))}\right)^{-\frac{(\nu+1)}{2}-k}{}_2F_1\left(\frac{\nu+1}{2}+k,\frac{\nu+1}{2}+k;\frac{1}{2};\frac{\rho^2(\bm{h})y^2_iy^2_j}{(y_i^2+\nu)(y_j^2+\nu)}\right)\nonumber\\
&+b_2\Gamma\left(\frac{\nu}{2}+k+1\right)\left(\frac{y^2_j+\nu}{2(1-\rho^2(\bm{h}))}\right)^{-\frac{\nu}{2}-k-1}{}_2F_1\left(\frac{\nu}{2}+k+1,\frac{\nu}{2}+k+1;\frac{3}{2};\frac{\rho^2(\bm{h})y^2_iy^2_j}{(y_i^2+\nu)(y_j^2+\nu)}\right)
\end{align}
\end{footnotesize}
finally, combining equations (\ref{demtff}), (\ref{demt1}) and
(\ref{degam}), we obtain
\begin{footnotesize}
\begin{align*}
f_{\bm{Y}^*_{ij}}(\bm{y}_{ij})&=\frac{\nu^{\nu}[(y_i^2+\nu)(y_j^2+\nu)]^{-(\nu+1)/2}\Gamma^2\left(\frac{\nu+1}{2}\right)}{\pi\Gamma^2\left(\frac{\nu}{2}\right)(1-\rho^2(\bm{h}))^{-(\nu+1)/2}}
\sum\limits_{k=0}^{\infty}\frac{\left(\frac{\nu+1}{2}\right)^2_k}{k!\left(\frac{\nu}{2}\right)_k}\left(\frac{\nu^2\rho^2(\bm{h})}{(y_i^2+\nu)(y_j^2+\nu)}\right)^k\nonumber\\
&\times{}_2F_1\left(\frac{\nu+1}{2}+k,\frac{\nu+1}{2}+k;\frac{1}{2};\frac{\rho^2(\bm{h})y_i^2y_j^2}{(y_i^2+\nu)(y_j^2+\nu)}\right)\nonumber\\
&+\frac{\rho(\bm{h})y_iy_j\nu^{\nu+2}[(y_i^2+\nu)(y_j^2+\nu)]^{-\nu/2-1}}{2\pi(1-\rho^2(\bm{h}))^{-(\nu+1)/2}}
\sum\limits_{k=0}^{\infty}\frac{\left(\frac{\nu}{2}+1\right)^2_k}{k!\left(\frac{\nu}{2}\right)_k}\left(\frac{\nu^2\rho^2(\bm{h})}{(y_i^2+\nu)(y_j^2+\nu)}\right)^k\nonumber\\
&\times{}_2F_1\left(\frac{\nu}{2}+k+1,\frac{\nu}{2}+k+1;\frac{3}{2};\frac{\rho^2(\bm{h})y_i^2y_j^2}{(y_i^2+\nu)(y_j^2+\nu)}\right)\nonumber\\
%&=&\frac{\nu^{\nu}[(y_i^2+\nu)(y_j^2+\nu)]^{-(\nu+1)/2}\Gamma^2\left(\frac{\nu+1}{2}\right)}{\pi\Gamma^2\left(\frac{\nu}{2}\right)(1-\rho^2(\bm{h}))^{-(\nu+1)/2}}
%F_4\left(\frac{\nu+1}{2},\frac{\nu+1}{2},\frac{1}{2},\frac{\nu}{2};\frac{\rho^2(\bm{h})y_i^2y_j^2}{(y_i^2+\nu)(y_j^2+\nu)},\frac{\nu^2\rho^2(\bm{h})}{(y_i^2+\nu)(y_j^2+\nu)}\right)\nonumber\\
%&+&\frac{\rho(\bm{h})y_iy_j\nu^{\nu+2}[(y_i^2+\nu)(y_j^2+\nu)]^{-\nu/2-1}}{2\pi(1-\rho^2(\bm{h}))^{-(\nu+1)/2}}
%F_4\left(\frac{\nu}{2}+1,\frac{\nu}{2}+1,\frac{3}{2},\frac{\nu}{2};\frac{\rho^2(\bm{h})y_i^2y_j^2}{(y_i^2+\nu)(y_j^2+\nu)},\frac{\nu^2\rho^2(\bm{h})}{(y_i^2+\nu)(y_j^2+\nu)}\right)
\end{align*}
\end{footnotesize}
and using (\ref{apell4})
we obtain theorem \ref{theo3}.
\end{proof}

\subsection{Proof of Theorem \ref{theo1}}

\begin{proof}
Consider $\bm{U}=(U(\bm{s}_1),\ldots,U(\bm{s}_n))^T$, $\bm{V}=
(|X_1(\bm{s}_1)|,\ldots,|X_1(\bm{s}_n)|)^T$, $\bm{Q}=
(X_2(\bm{s}_1),\ldots,X_2(\bm{s}_n))^T$  where
$\bm{X}_k=(X_k(\bm{s}_1),\ldots,X_k(\bm{s}_n))^T\sim
N_n(\bm{0},\Omega)$, for $k= 1,2$, which are assumed to be
independent. By definition of the skew-Gaussian process
in~(\ref{repskew}) we have:
\begin{equation*}
\bm{U}=\bm{\alpha}+\eta\bm{V}+\omega\bm{Q}
\end{equation*}
where, by assumption $\bm{V}$ and $\bm{Q}$ are independent. Thus,
by conditioning on $\bm{V}=\bm{v}$, we have
$\bm{U}|\bm{V}=\bm{v}\sim
N_n(\bm{\alpha}+\eta\bm{v},\omega^2\Omega)$, from which we obtain
\begin{equation*}\label{skk}
f_{\bm{U}}(\bm{u})=\int\limits_{\mathbb{R}^n}\phi_n(\bm{u};\bm{\alpha}+\eta\bm{v},\omega^2\Omega)f_{\bm{V}}(\bm{v})d\bm{v}
\end{equation*}

To solve this integral we need $f_{\bm{V}}(\bm{v})$, $i.e$., the
joint density of $\bm{V}=
(|X_1(\bm{s}_1)|,\ldots,|X_1(\bm{s}_n)|)^T$. Let
$\bm{X}_k=(X_1,\ldots,X_n)^T=(X_1(\bm{s}_1),\ldots,X_1(\bm{s}_n))^T$
and $\bm{V}= (|X_1|,\ldots,|X_n|)^T$. Additionally, consider the diagonal
matrices $\bm{D}(\bm{l})=\text{diag}\{l_1,\ldots,l_n\}$, with
$\bm{l}=(l_1,\ldots,l_n)\in\{-1,+1\}^n$, which are such that
$\bm{D}(\bm{l})^2$ is the identity matrix.
Since $\bm{l}\circ\bm{v}=\bm{D}(\bm{l})\bm{v}$ (the componentwise
product) and $\bm{X}\sim N_n(\bm{0},\Omega)$, we then have
\begin{align*}
F_{\bm{V}}(\bm{v})&=Pr(\bm{V}\leq\bm{v})=Pr(|\bm{X}|\leq\bm{v})=Pr(-\bm{v}\leq\bm{X}\leq \bm{v})\\
  &=\sum\limits_{\bm{l}\in\{-1,+1\}^n} (-1)^{N_{-}}\Phi_n(\bm{D}(\bm{l})\bm{v};\bm{0},\Omega),\;\;\;\;(N_{-}=\sum\limits_{i=1}^n I_{l_i=1}\text{det}\{\bm{D}(\bm{l})\}) \\
  &=\sum\limits_{\bm{l}\in\{-1,+1\}^n} \text{det}\{\bm{D}(\bm{l})\}\Phi_n(\bm{D}(\bm{l})\bm{v};\bm{0},\Omega)
\end{align*}
Hence, by using that
\begin{equation*}
\frac{\partial^n\Phi_n(\bm{D}(\bm{l})\bm{v};\bm{0},\Omega)}{\partial
v_1\cdots\partial
v_n}=\text{det}\{\bm{D}(\bm{l})\}\Phi_n(\bm{D}(\bm{l})\bm{v};\bm{0},\Omega)
\end{equation*}
we find that the joint density of $\bm{V}$ is
\begin{align*}
f_{\bm{V}}(\bm{v})&=\sum\limits_{\bm{l}\in\{-1,+1\}^n}[\text{det}\{\bm{D}(\bm{l})\}]^2\phi_n(\bm{D}(\bm{l})\bm{v};\bm{0},\Omega)   \\
  &=\sum\limits_{\bm{l}\in\{-1,+1\}^n}\phi_n(\bm{D}(\bm{l})\bm{v};\bm{0},\Omega),\;\;\;\;([\text{det}\{\bm{D}(\bm{l})\}]^2=1) \\
  &=\sum\limits_{\bm{l}\in\{-1,+1\}^n}
  |\text{det}\{\bm{D}(\bm{l})\}|\phi_n(\bm{v};\bm{0},\Omega_l),\;\;\;\;(\Omega_l=\bm{D}(\bm{l})\Omega\bm{D}(\bm{l})=(l_il_j\rho_{ij}))\\
  &=\sum\limits_{\bm{l}\in\{-1,+1\}^n}\phi_n(\bm{v};\bm{0},\Omega_l),\;\;\;\;(|\text{det}\{\bm{D}(\bm{l})\}|=1)\\
  &=2\sum\limits_{\bm{l}\in\{-1,+1\}^n:l\neq-l}\phi_n(\bm{v};\bm{0},\Omega_l)
\end{align*}
where the last identity is due to
$\Omega_{-l}=\bm{D}(-\bm{l})\Omega\bm{D}(-\bm{l})=\bm{D}(\bm{l})\Omega\bm{D}(\bm{l})=\Omega_{l}$
for all $\bm{l}\in\{-1,+1\}^n$; e.g. for $n=3$ the sum must be
performed on
\begin{equation*}
\bm{l}\in\{(+1,+1,+1),(+1,+1,-1),(+1,-1,+1),(-1,+1,+1)\}
\end{equation*}
since
\begin{equation*}
-\bm{l}\in\{(-1,-1,-1),(-1,-1,+1),(-1,+1,-1),(+1,-1,-1)\}
\end{equation*}
and both sets produce the same correlation matrices. The joint
density of $\bm{U}$ is thus given by
\begin{align*}
f_{\bm{U}}(\bm{u})&=2\sum\limits_{w\in\{-1,+1\}^n:w\neq-w}\int\limits_{\mathbb{R}^n_+}\phi_n(\bm{u};\bm{\alpha}+\eta\bm{v},\omega^2\Omega)\phi_n(\bm{v};\bm{0},\Omega_l)d\bm{v}\\
  &=2\sum\limits_{w\in\{-1,+1\}^n:w\neq-w}\phi_n(\bm{u};\bm{\alpha},\bm{A}_l)\int\limits_{\mathbb{R}^n_+}\phi_n(\bm{v};\bm{c}_l,\bm{B}_l)d\bm{v}\\
  &=2\sum\limits_{w\in\{-1,+1\}^n:w\neq-w}\phi_n(\bm{u};\bm{\alpha},\bm{A}_l)\Phi_n(\bm{c}_l;\bm{0},\bm{B}_l)
\end{align*}
where $\bm{A}_l=\omega^2\Omega+\eta^2\Omega_l$,
$\bm{c}_l=\eta\Omega_l\bm{A}_l^{-1}(\bm{u}-\bm{\alpha})$,
$\bm{B}_l=\Omega_l-\eta^2\Omega_l\bm{A}_l^{-1}\Omega_l$, and we
have used the identity
$\phi_n(\bm{u};\bm{\alpha}+\eta\bm{v},\omega^2\Omega)\phi_n(\bm{v};\bm{0},\Omega_l)=\phi_n(\bm{u};\bm{\alpha},\bm{A}_l)\phi_n(\bm{v};\bm{c}_l,\bm{B}_l)$
which follows straightforwardly from the standard
marginal-conditional factorizations of the underlying multivariate
normal joint density.
\end{proof}

%\section*{Supplementary Materials}
\bibliographystyle{ECA_JASA}
\bibliography{mybib2}
\end{document}